\documentclass[12pt]{amsart}
\usepackage{setspace}
\usepackage{amsmath,amscd, amsthm}
\usepackage{amsfonts}
\usepackage[english]{babel}
\usepackage{amssymb}
\usepackage{bbm}
\usepackage[top=1in, bottom=1in, left=.9in, right=.9in]{geometry}
\usepackage{enumerate}
\usepackage{mathrsfs}

\usepackage{xypic}
\input xy

\newcommand{\R}{\mathbb{R}}

\newcommand{\Z}{\mathbb{Z}}

\newcommand{\h}{\mathfrak{H}}

\newcommand{\tab}{\hspace*{15pt}}

\newcommand{\xqedhere}[2]{%
  \rlap{\hbox to#1{\hfil\llap{\ensuremath{#2}}}}}

\usepackage{tikz-cd}
\usepackage{enumerate}
\usepackage{mathrsfs}
\usepackage[utf8]{inputenc}
\usepackage[T1]{fontenc}
\usepackage{xypic}
\input xy

\newtheorem{thm}{Theorem}

\newtheorem{cor}[thm]{Corollary}
\newtheorem{lem}[thm]{Lemma}

\xyoption{all}

\title{Differential equations in automorphic forms}
\author{Kim Klinger-Logan}
\date{07.06.2018} 

\begin{document}

\maketitle

\begin{quote}{\sc Abstract:} Physicists such as Green, Vanhove, et al show that differential equations involving automorphic forms govern the behavior of gravitons.  One particular point of interest is solutions to $(\Delta-\lambda)u=E_{\alpha} E_{\beta}$ on an arithmetic quotient of the exceptional group $E_8$.  We establish that the existence of a solution to $(\Delta-\lambda)u=E_{\alpha}E_{\beta}$ on the simpler space $SL_2(\Z)\backslash SL_2(\R)$ for certain values of $\alpha$ and $\beta$ depends on nontrivial zeros of the Riemann zeta function $\zeta(s)$.  Further, when such a solution exists, we use spectral theory to solve $(\Delta-\lambda)u=E_{\alpha}E_{\beta}$ on $SL_2(\Z)\backslash SL_2(\R)$ and provide proof of the meromorphic continuation of the solution.  The construction of such a solution uses Arthur truncation, the Maass-Selberg formula, and automorphic Sobolev spaces.
\\ \end{quote}

\section{Introduction}

\tab In \cite{Greenetal2010}, Green, Miller, Russo and Vanhove study the low energy expansions of string theory amplitudes that generalize  the amplitudes of classical supergravity. In doing so they derive differential equations that model the behavior of the 4-loop supergraviton.  Such differential equations govern the amplitudes of closed type II superstring theory.  These differential equations involve combinations of Eisenstein series in their expression.\\

\tab The differential equations presented in \cite{Greenetal2010} are of the forms:
\begin{equation*}\begin{split}(\Delta-\lambda_s)u_w & =0\\
(\Delta-\lambda_s)u_w & =c\\
(\Delta-\lambda_s)u_w & =E_{\alpha} \\
(\Delta-\lambda_s)u_w & =E_{\alpha}\cdot E_{\beta} \end{split}\end{equation*}

on the exceptional group $E_8$ where $c$ is a constant and $E_{\alpha}$ and $E_{\beta}$ are Eisenstein series.  Solutions for the first three such equations are known.  Green, Miller, Russo and Vanhove express a version of these solutions in \cite{Greenetal2010}.  Furthermore, spectral solutions to similar equations is understood (see the work of P. Garrett \cite{Garrett2011p}, \cite{Garrett2013}).  The last equation, however, is more challenging to solve. It should be noted that in \cite{Greenetal2010},   \cite{DHoker2015}  and \cite{GMV2015}, the form of this last equation was given where $\alpha=\beta$.  \\

\tab As a precedent for solving such an equation, we will solve $(\Delta-\lambda_w)u_w=E_{\alpha}\cdot E_{\beta}$ on $\Gamma\backslash \mathfrak{H}$ where $\Gamma=SL_2(\Z)$ and $\mathfrak{H}$ is the upper half plane, $\displaystyle\Delta=y^2\left(\frac{\partial^2}{\partial x^2}+ \frac{\partial^2}{\partial y^2}\right)$ is the invariant Laplacian and $\lambda_w =w(w-1)$.   There are of course many differences in these domains but examining the simpler domain will illuminate some of the necessary techniques for analyzing solutions elsewhere.  Furthermore, this technique allows us to compute the solution to the differential equation in many cases at once.  In \cite{GMV2015} Green, Miller and Vanhove present a solution on $\Gamma\backslash \mathfrak{H}$ where  $\alpha=\beta=3/2$ and $\lambda_w=12$ and D'Hoker, Green, G\"{u}rdo\u{g}an and Vanhove give a solution for integer values of $\alpha$ and $\beta$ in \cite{DHokerGurdogen2016}.  Our solution will subsume these findings. \\

 \tab In what follows, we will solve $(\Delta-\lambda_w)u_w=E_{\alpha}\cdot E_{\beta}$ on $\Gamma\backslash \mathfrak{H}$ using spectral theory.  This involves finding a spectral expansion for  $E_{\alpha}\cdot E_{\beta}$; however, given that $E_{\alpha}\cdot E_{\beta}\notin L^2(\Gamma\backslash\mathfrak{H})$ no such expansion can be directly computed as methods for computing $L^2$-spectral expansions do not directly apply.  Thus, in order to guarantee convergence of the spectral integrals, we will subtract a linear combination of Eisenstein series from $E_{\alpha}\cdot E_{\beta}$ and compute the spectral expansion of this new function.  We will then be able to solve the differential equation in the usual way using global automorphic Sobolev spaces. The computation of the spectral expansion for this new function involves implementing tools developed by Zagier \cite{Zagier1982} and Casselman \cite{Casselman1993} related to the extending the Rankin-Selberg method for functions not of rapid decay (explanation of this phenomenon can also be found in \cite{Garrett2016}).  This method makes use of Arthur truncation and the Maass-Selberg formula. \\
 
 \tab In Section \ref{results} and \ref{limits}, we will state our main results and prove the existence and uniqueness of solutions to $(\Delta-\lambda_w)u_w=E_{\alpha}\cdot E_{\beta}$ on $\Gamma\backslash \mathfrak{H}$ for almost all values of $\alpha$ and $\beta$. In Sections \ref{cusp}, \ref{cont} and \ref{res}, we will compute the spectral expansion of this solution.  After computing an explicit form of the solution, we will meromorphically continue the solution in $w$ to the left-half plane in section \ref{mero}.  This proof relies upon the constructions involving vector-valued integrals as presented by Gelfand, Pettis, and Grothendieck.  A brief summary of these constructions is provided in the appendix (Section \ref{app}).\\
 
 \subsection{Background and Motivation}\label{background}

\tab Let $E_{\alpha}$ and $E_{\beta}$ be two Eisenstein series on $\Gamma\backslash SL_2(\R)$ for $\Gamma=SL_2(\Z)$.  Each $E_s$ can then be described as 
$$E_s(z)=\sum_{\gamma\in P\backslash \Gamma}\text{Im}(\gamma z)^{s}$$ where $P$ is the parabolic of $SL_2(\mathbb{R})$ restricted to $\Gamma$. The following result of the analytic continuation and functional equation is commonly known and its proof can be found many places including (but not limited to) Epstein's \cite{Epstein1903} and Garrett's \cite{Garrett2014} explication of Godement's \cite{Godement1966} 1966 work.

{\thm For each $z\in\h$, $s(s-1)\xi(s)\cdot E_s(z)$ has an analytic continuation to an entire function of $s$ and functional equation given by $$\xi(2s)E_s =\xi(2-2s)E_{1-s}$$ where $\xi(s)= \pi^{-s/2}\Gamma(\frac{s}{2})\zeta(s)$ is the completed Riemann zeta function.}\\

Note that we will employ the notation $\displaystyle c_s=\frac{\xi(2-2s)}{\xi(2s)}$ so that the function equation for $E_s$ is given by $E_s=c_s\cdot E_{1-s}$.\\

\tab Furthermore, it is known (and proof can be found in \cite{Garrett2014}) that the Fourier-Whittaker expansion for $E_s$  (for $s\neq 1$) is given by 
$$E_s(x+iy) = y^s+c_sy^{1-s}+\frac{1}{\pi^{-s} \Gamma(s)\zeta(2s)}\sum_{n\neq 0}\frac{\sigma_{2s-1}(|n|)}{|n|^{s-\frac{1}{2}}}\cdot\sqrt{y} \int_0^{\infty}t^{s-1/2}e^{-(t+\frac{1}{t}) \pi |n| y}\frac{dt}{t}\cdot e^{2\pi i n x}$$
$$=y^{s}+c_{s }y^{1-s}+\sum_{n\neq 0}\varphi(n,s)\cdot W_{s}(|n|y)\cdot e^{2\pi i n x}$$
where 
$$\displaystyle W_{s}(|n|y)= \sqrt{y} \int_0^{\infty}t^{s-1/2}e^{-(t+\frac{1}{t}) \pi |n| y}\frac{dt}{t}$$ is the Whittaker function -- the unique (up to scalars) solution $u$ of $\displaystyle u''-\left(\frac{\lambda_{s}}{y^2} +4\pi^2n^2\right)\cdot u = 0$ for $\lambda_s=s(s-1)$ -- 
and  $\displaystyle\varphi(n,s)= \frac{1}{\pi \Gamma(s)\zeta(2s)}\frac{\sigma_{2s-1}(|n|)}{|n|^{s-\frac{1}{2}}}$ and $\sigma_{2s-1}(|n|)$ is the sum of the $(2s-1)^{th}$ powers of positive divisors of $n$.  \\

\tab Recall that $E_s$ has a simple pole at $s=1$ so the constant term $c_PE_1^*$ for the $a_{-1}$ coefficient of the Laurent expansion $E_s$ at $s=1$ will not have the form $y^s+c_sy^{1-s}$. Instead, the Fourier-Whittaker expansion for $E_1^*$ is given by 
$$E_1^*(x+iy)=y+C-\frac{3}{\pi }\log y+\sum_{n\neq 0}\varphi(n,1)\cdot W_{1}(|n|y)\cdot e^{2\pi i n x}$$ where $C=\frac{d}{ds}\left((s-1)c_{s} \right)\Big|_{s=1} $ and $W_s$ is as above.  We will use the notation $c_PE_s$ to refer to the constant term of the Eisenstein series at $s$. \\

\tab We will later also need the Fourier-Whittaker expansion of cuspforms.  Indeed the archimedean parts of that of the Fourier-Whittaker functions for a cuspform with $\Delta$-eigenvalue $\lambda_s=s(s-1)$ are the same as the Eisenstein series (see \cite{Garrett2014}).  We then have for $f$ a cuspform on $\Gamma\backslash \h$ that 
$$f(x+iy)=\sum_{n\neq 0} c_n \cdot W_s(|n|y)\cdot e^{2\pi i n x}$$ for some constants $c_n$ with $W_s(|n|y)$ as above.\\

\tab In what follows, we solve $$(\Delta-\lambda_w)u_w=E_{\alpha}\cdot E_{\beta}$$ on $\Gamma\backslash \mathfrak{H}$ where $\Gamma=SL_2(\Z)$ and $\mathfrak{H}$ is the upper half plane, $\displaystyle\Delta=y^2\left(\frac{\partial^2}{\partial x^2}+ \frac{\partial^2}{\partial y^2}\right)$ is the invariant Laplacian and $\lambda_w =w(w-1)$. First we must write out a spectral expansion for $E_{\alpha}\cdot E_{\beta}$.  \\

\tab  For $0\leq k\in\mathbb{Z}$, the $k^{th}$-Sobolev norm on $ C_c^{\infty}(\Gamma\backslash\h)$ is given by $$|f|_k^2:=\langle(1-\Delta)^k f,f \rangle_{L^2(\Gamma\backslash \mathfrak{H})}$$ and we define the global automorphic Sobolev space $H^k(\Gamma\backslash \mathfrak{H})$ to be the completion of $C_c^{\infty}(\Gamma\backslash\h)$ with respect to $|\cdot|_k$.  Ordinarily, for $S$ in some Sobolev space $H^k(\Gamma\backslash\mathfrak{H})$ we can write
 $$S=\sum_{f\text{ cfm}} \langle S,f\rangle\cdot f+ \frac{\langle S,1\rangle\cdot 1}{\langle 1,1\rangle}+\frac{1}{4\pi i}\int_{(1/2)}\langle S,E_s\rangle\cdot E_s\,ds$$
 (see Section \ref{app}, \cite{Garrett2011sob} or \cite{DeCelles2016} for further explanation of global automorphic Sobolev spaces and \cite{Iwaniec} for the spectral expansion). The problem is that $E_{\alpha}\cdot E_{\beta}$ is not in such a Sobolev space so we cannot properly write this spectral decomposition for $S=E_{\alpha}\cdot E_{\beta}$.  The trick we use is subtraction of a finite linear combination of $E_{\alpha}$ and $E_{\beta}$ so that $$S=E_{\alpha}\cdot E_{\beta}-\sum_ic_i E_{s_i}$$ which will be in $L^2$ or even possibly in $H^{\infty}$ and we can give a decomposition for $\Delta$.\\
 
 \subsection{Results}\label{results}
 
\tab We will use the spectral relation in Section \ref{app} to solve $\displaystyle(\Delta-\lambda_w)u=E_{\alpha}\cdot E_{\beta}$ on $\Gamma\backslash SL_2(\R)$.  Also, note that the automorphic Sobolev space $H^k$ in which this solution exists is also defined in Section \ref{app}.  Furthermore, we will show that the solution we have found is unique.\\

\tab Consider the set $$\mathcal{C}:=\{ \alpha,\beta\in \mathbb{C}-\{1\} ~|~ \text{Re}(\alpha)\geq 1/2, \text{Re}(\beta)\geq1/2,  \text{Re}(\alpha+\beta)\neq 3/2, \text{Re}(\beta)\neq \pm 1/2+\text{Re}(\alpha) \}.$$  The following guarantees the existence of a unique solution to $\displaystyle (\Delta-\lambda)u=E_{\alpha}\cdot E_{\beta}\phantom{e}\text{on}\phantom{e}\Gamma\backslash\mathfrak{H}$ for all $\alpha, \beta\in \mathcal{C}$. There are a few complex values eliminated from the set $\mathcal{C}$.  We will address what happens with the solution when $\text{Re}(\alpha+\beta)= 3/2$ and $\text{Re}(\beta)= \pm 1/2+\text{Re}(\alpha)$ in Section \ref{limits}.  However, it should be noted that the reason for the exclusion of the value $1$ is that $E_s$ has a pole at $s=1$.\\

\tab   Let  $\mathcal{E}$ be the vector space consisting of finite linear combinations of Eisenstein series so that 
$$\mathcal{E}(\Gamma\backslash \mathfrak{H}):= \left\{\sum_i a_i F_{s_i}(z) ~\Big|~ a_i\in \mathbb{C}\text{ and }F_{s_i}(z)\in \left\{\mathbb{C},  E_1^*(z), E_{s_i}(z) \text{ for }  s_i\in\mathbb{C}-1\right\}  \right\}.$$
This space has an LF-space structure as locally convex colimit of finite-dimensional spaces.\\

{\thm\label{existencethm}  In $\text{Re}(w)>1/2$, for $\alpha, \beta\in \mathcal{C}$,  $\displaystyle (\Delta-\lambda)u=E_{\alpha}\cdot E_{\beta}\phantom{e}\text{on}\phantom{e}\Gamma\backslash\mathfrak{H}$ has a unique solution in $H^{-\infty}(\Gamma\backslash \mathfrak{H})\oplus\mathcal{E}(\Gamma\backslash \mathfrak{H})$ with spectral expansion which lies in $H^2(\Gamma\backslash \mathfrak{H})\oplus \mathcal{E}(\Gamma\backslash \mathfrak{H})$.}\\

\begin{proof}   The existence of the solution can be seen in the computation of the spectral expansion.  First, we will subtract a finite linear combination of Eisenstein series $E_{s_i}$ so that $$S=E_{\alpha}\cdot E_{\beta}-\sum_ic_i E_{s_i}$$ which will be in $L^2(\Gamma\backslash\h)$.

\tab If $S\in L^2(\Gamma\backslash\h)$, we can write a convergent spectral expansion
$$S=\sum_{f\text{ cfm}} \langle S,f\rangle\cdot f+ \frac{\langle S,1\rangle\cdot 1}{\langle 1,1\rangle}+\frac{1}{4\pi i}\int_{(1/2)}\langle S,E_s\rangle\cdot E_s\,ds$$ where this convergence occurs in $L^2$.  Furthermore, this expansion can be extended by isometry to all of $H^{-\infty}.$  It the follows that we can write $$E_{\alpha}\cdot E_{\beta}=\sum_ic_i E_{s_i}+ \sum_{f\text{ cfm}} \langle S,f\rangle\cdot f+ \frac{\langle S,1\rangle\cdot 1}{\langle 1,1\rangle}+\frac{1}{4\pi i}\int_{(1/2)}\langle S,E_s\rangle\cdot E_s\,ds$$ which also converges in $L^2$.
Then, given that the spectral data in the expansions above is given by eigenfunctions for $\Delta$, the solution to $\displaystyle(\Delta-\lambda_w)u=E_{\alpha}\cdot E_{\beta}$ is given by division by the corresponding eigenvalues.  

\tab It can be found in many sources such as \cite{Garrett2014} that the theory of the constant term implies that $E_{\alpha}=y^{\alpha}+c_{\alpha} y^{1-\alpha}+R_{\alpha}$ where $R_{\alpha}$ is rapidly decreasing. Thus  
$$E_{\alpha}E_{\beta}=(y^{\alpha}+c_{\alpha} y^{1-\alpha}+R_{\alpha})(y^{\beta}+c_{\beta} y^{1-\beta}+R_{\beta})=y^{\alpha+\beta}+c_{\beta} y^{1+\alpha-\beta}+c_{\alpha} y^{1-\alpha+\beta}+c_{\alpha}c_{\beta}  y^{2-\alpha-\beta}+R$$
where $R$ is rapidly decreasing since $y^{\alpha}+c_{\alpha} y^{1-\alpha}$ and $y^{\beta}+c_{\beta} y^{1-\beta}$ are of moderate growth (and rapidly decreasing times moderate growth is rapidly decreasing).
Notice that different values of $\alpha$ and $\beta$ will imply different vanishing for terms of $E_{\alpha}\cdot E_{\beta}$.  Thus in different regimes, we will be required to subtract different linear combinations of Eisenstein series as follows.  

\tab Assume that $\alpha \neq 1$ and $\beta\neq 1$ since $E_s$ has a pole at $s=1$.  Also, without loss of generality, assume that $\text{Re}(\alpha)\leq \text{Re}(\beta)$.\\


{\bf (I):} Suppose that $1/2 \leq \text{Re}(\alpha)< \text{Re}(\alpha)+1/2<\text{Re}(\beta)$.

$$\sum_ic_i E_{s_i}= E_{\alpha+\beta}+c_{\alpha}\cdot E_{1-\alpha+\beta} $$
$$= y^{\alpha+\beta}+c_{\alpha+\beta} y^{1-\alpha-\beta}
+c_{\alpha}y^{1-\alpha+\beta}+c_{\alpha}c_{1-\alpha+\beta} y^{\alpha-\beta}+R_{\alpha+\beta}+c_{\alpha}R_{1-\alpha+\beta}$$

Thus
$$S=E_{\alpha}\cdot E_{\beta}-\sum_ic_i E_{s_i}$$
$$= c_{\beta} y^{1+\alpha-\beta}+c_{\alpha}c_{\beta} y^{2-\alpha-\beta}-c_{\alpha+\beta} y^{1-\alpha-\beta}-c_{\alpha}c_{1-\alpha+\beta} y^{\alpha-\beta}+R-R_{\alpha+\beta}-c_{\alpha}R_{1-\alpha+\beta} $$

Since $1/2 \leq \text{Re}(\alpha)< \text{Re}(\alpha)+1/2<\text{Re}(\beta)$,  $$c_{\beta} y^{1+\alpha-\beta}+c_{\alpha}c_{\beta} y^{2-\alpha-\beta}-c_{\alpha+\beta} y^{1-\alpha-\beta}-c_{\alpha}c_{1-\alpha+\beta} y^{\alpha-\beta}\in L^2(\Gamma\backslash\h).$$\\


{\bf(II):} Suppose $1/2\leq\text{Re}(\alpha)\leq \text{Re}(\beta)<  \text{Re}(\alpha)+1/2$ but that $\alpha\neq\beta$.
This case yields two subcases depending on $\text{Re}(\alpha+\beta)$:

\tab {\bf (IIa)}  Suppose also that $\text{Re}(\alpha+\beta)> 3/2$.

$$\sum_ic_i E_{s_i}= E_{\alpha+\beta}+c_{\beta}\cdot E_{1+\alpha-\beta}+c_{\alpha}\cdot E_{1-\alpha+\beta} $$
$$= y^{\alpha+\beta}+c_{\alpha+\beta} y^{1-\alpha-\beta}+R_{\alpha+\beta}
+c_{\beta}y^{1+\alpha-\beta}+c_{\beta}c_{1+\alpha-\beta} y^{-\alpha+\beta}+c_{\beta}R_{1+\alpha-\beta}$$
$$+c_{\alpha}y^{1-\alpha+\beta}+c_{\alpha}c_{1-\alpha+\beta} y^{\alpha-\beta}+c_{\alpha}R_{1-\alpha+\beta}$$

Thus
$$S=E_{\alpha}\cdot E_{\beta}-\sum_ic_i E_{s_i}$$
$$=c_{\alpha}c_{\beta}  y^{2-\alpha-\beta}-c_{\alpha+\beta} y^{1-\alpha-\beta}
-c_{\beta}c_{1+\alpha-\beta} y^{-\alpha+\beta}
-c_{\alpha}c_{1-\alpha+\beta} y^{\alpha-\beta}$$
$$+R-R_{\alpha+\beta}-c_{\beta}R_{1+\alpha-\beta}- c_{\alpha}R_{1-\alpha+\beta}$$

Since $1/2\leq\text{Re}(\alpha)\leq \text{Re}(\beta)<  \text{Re}(\alpha)+1/2$ and $\text{Re}(\alpha+\beta)> 3/2$,

$$c_{\alpha}c_{\beta}  y^{2-\alpha-\beta}-c_{\alpha+\beta} y^{1-\alpha-\beta}
-c_{\beta}c_{1+\alpha-\beta} y^{-\alpha+\beta}
-c_{\alpha}c_{1-\alpha+\beta} y^{\alpha-\beta}\in L^2(\Gamma\backslash\h).$$

\tab {\bf (IIb)} Now suppose instead that $\text{Re}(\alpha+\beta)< 3/2$.

$$\sum_ic_i E_{s_i}= E_{\alpha+\beta}+c_{\beta}\cdot E_{1+\alpha-\beta}+c_{\alpha}\cdot E_{1-\alpha+\beta} +c_{\alpha}c_{\beta}\cdot E_{2-\alpha-\beta}$$

$$= y^{\alpha+\beta}+c_{\alpha+\beta} y^{1-\alpha-\beta}+R_{\alpha+\beta}
+c_{\beta}y^{1+\alpha-\beta}+c_{\beta}c_{1+\alpha-\beta} y^{-\alpha+\beta}+c_{\beta}R_{1+\alpha-\beta}$$
$$+c_{\alpha}y^{1-\alpha+\beta}+c_{\alpha}c_{1-\alpha+\beta} y^{\alpha-\beta}+c_{\alpha}R_{1-\alpha+\beta}
+c_{\alpha}c_{\beta}y^{2-\alpha-\beta}+c_{\alpha}c_{\beta}c_{2-\alpha-\beta} y^{\alpha+\beta-1}+c_{\alpha}c_{\beta}R_{2-\alpha-\beta}$$

Thus
$$S=E_{\alpha}\cdot E_{\beta}-\sum_ic_i E_{s_i}$$
$$=-c_{\alpha+\beta} y^{1-\alpha-\beta}
-c_{\beta}c_{1+\alpha-\beta} y^{-\alpha+\beta}-c_{\alpha}c_{1-\alpha+\beta} y^{\alpha-\beta}
-c_{\alpha}c_{\beta}c_{2-\alpha-\beta} y^{\alpha+\beta-1}$$
$$+R-R_{\alpha+\beta}-c_{\beta}R_{1+\alpha-\beta}-c_{\alpha}R_{1-\alpha+\beta}-c_{\alpha}c_{\beta}R_{2-\alpha-\beta}$$

Since  $1/2\leq\text{Re}(\alpha)\leq\text{Re}(\beta)< \text{Re}(\alpha)+1/2$ and $\text{Re}(\alpha+\beta)< 3/2$,
$$-c_{\alpha+\beta} y^{1-\alpha-\beta}
-c_{\beta}c_{1+\alpha-\beta} y^{-\alpha+\beta}-c_{\alpha}c_{1-\alpha+\beta} y^{\alpha-\beta}
-c_{\alpha}c_{\beta}c_{2-\alpha-\beta} y^{\alpha+\beta-1}\in L^2(\Gamma\backslash\h).$$\\


{\bf (III):} Suppose that $\alpha=\beta$. This will again yield two different cases based on $\text{Re}(\alpha)$:

\tab {\bf (IIIa)} Suppose also that $\text{Re}(\alpha)>3/4$.

$$\sum_ic_i E_{s_i}=E_{2\alpha}+2c_{\alpha}E_1^*- \frac{\pi}{3}C_{\alpha}
= y^{2\alpha}+c_{2\alpha}y^{1-2\alpha}+R_{2\alpha}+  2c_{\alpha}(y-\frac{3}{\pi }\log y+C - \frac{\pi}{3}C_{\alpha}+R_1 )$$  where $C_{\alpha}= \frac{d}{ds}c_s\Big|_{s=\alpha}$.
Thus 
$$S= (E_\alpha)^2- \sum_ic_i E_{s_i}= (E_{\alpha})^2-  E_{2\alpha}-2c_{\alpha}E_1^* + \frac{\pi}{3}C_{\alpha}$$
$$=c_{\alpha}^2y^{2-2\alpha}-c_{2\alpha}y^{1-2\alpha}+2c_{\alpha}\frac{3}{\pi }\log y+C + \frac{\pi}{3}C_{\alpha}+R $$
and so $(E_{\alpha})^2-  E_{2\alpha}-2c_{\alpha}E_1\in L^2(\Gamma\backslash\mathfrak{H})$ for $\text{Re}(\alpha)>3/4$.

\tab Note that adding the constant $ \frac{\pi}{3}C_{\alpha}$ does not affect whether $S$ is in $L^2$; however, this regime will aid computation later in the paper and arrises when taking the limit as $\beta\to\alpha$ as seen in Lemma \ref{a=b} below.

\tab {\bf (IIIb)}  Instead suppose that $1/2\leq \text{Re}(\alpha)<3/4$.

$$\sum_ic_i E_{s_i}=E_{2\alpha}+2c_{\alpha}E_1^*+c_{\alpha}^2E_{2-2\alpha}- \frac{\pi}{3}C_{\alpha}$$
$$= y^{2\alpha}+c_{2\alpha}y^{1-2\alpha}+R_{2\alpha}+  2c_{\alpha}(y-\frac{3}{\pi }\log y+C+R_1 )+c_{\alpha}^2(y^{2-2\alpha}+c_{2-2\alpha}y^{1-(2-2\alpha)}+R- \frac{\pi}{3}C_{\alpha}$$ 

Thus
$$S=(E_{\alpha})^2-\sum_ic_i E_{s_i}= (E_{\alpha})^2-  E_{2\alpha}-2c_{\alpha}E_1^*-c_{\alpha}^2E_{2-2\alpha}+\frac{\pi}{3}C_{\alpha}$$
$$=-c_{2\alpha}y^{1-2\alpha}+2c_{\alpha}\frac{3}{\pi }\log y
-c_{\alpha}^2c_{2-2\alpha}y^{2\alpha-1}+C+ \frac{\pi}{3}C_{\alpha}+R$$

so $(E_{\alpha})^2-  E_{2\alpha}-2c_{\alpha}E_1-c_{\alpha}^2E_{2-2\alpha}\in L^2(\Gamma\backslash\mathfrak{H})$ for $1/2\leq\text{Re}(\alpha)<3/4$.\\

\tab We have shown that for each $\alpha$ and $\beta$ there is a linear combination of Eisenstein series $\sum_ic_i E_{s_i}$ so that $S=E_{\alpha}\cdot E_{\beta}-\sum_ic_i E_{s_i}\in L^2.$  We can thus write a spectral expansion for each case for $S$ and get
$$E_{\alpha}\cdot E_{\beta}=\sum_ic_i E_{s_i}+ \sum_{f\text{ cfm}} \langle S,f\rangle\cdot f+ \frac{\langle S,1\rangle\cdot 1}{\langle 1,1\rangle}+\frac{1}{4\pi i}\int_{(1/2)}\langle S,E_s\rangle\cdot E_s\,ds$$ in $L^2(\Gamma\backslash \mathfrak{H})\oplus \mathcal{E}(\Gamma\backslash \mathfrak{H})$.

 \tab To establish uniqueness, suppose that there are two solutions $u$ and $v$ to $\displaystyle(\Delta-\lambda_w)u=E_{\alpha}\cdot E_{\beta}$ in $H^2(\Gamma\backslash\h)\oplus\mathcal{E}(\Gamma\backslash \mathfrak{H})$.   Then $( \Delta-\lambda_w)(u-v)=E_{\alpha}\cdot E_{\beta}-E_{\alpha}\cdot E_{\beta}=0$.  Thus $u-v$ is a solution to the homogeneous equation $\displaystyle(\Delta-\lambda_w)(u-v)=0$ and $\lambda_w\in\mathbb{R}$ but this cannot be the case if $\text{Re}(w)>1/2$ and $\text{Im}(w)>0$. \\
\end{proof}

\tab Observe that for $\text{Re}(s)< 1/2$, the functional equation gives $E_s=c_s\cdot E_{1-s}$. Thus it is sufficient to consider the case where $\text{Re}(\alpha)\geq 1/2$ and  $\text{Re}(\beta)\geq 1/2$. Many of the other values excluded from $\mathcal{C}$ are in fact problematic as will will see in Section \ref{limits}. However, before we consider what is happening at these values, we will  give a spectral expansion for the solution $u_w$. \\

{\thm\label{main}  In $\text{Re}(w)>1/2$, for  $\alpha, \beta\in \mathcal{C}$, $\displaystyle (\Delta-\lambda)u=E_{\alpha}\cdot E_{\beta}\phantom{e}\text{on}\phantom{e}\Gamma\backslash\mathfrak{H}$ has a unique solution in $H^{-\infty}(\Gamma\backslash \mathfrak{H})\oplus\mathcal{E}(\Gamma\backslash \mathfrak{H})$ with spectral expansion which lies in $H^2(\Gamma\backslash \mathfrak{H})\oplus \mathcal{E}(\Gamma\backslash \mathfrak{H})$ and is given by
$$u_w=\sum_i \frac{c_{i} E_{s_i}}{\lambda_{s_i}-\lambda_w}-  \mathbbm{1}_{\alpha=\beta}\cdot \frac{\frac{\pi}{3}C_{\alpha}}{\lambda_{1}-\lambda_w}+\sum_{f\text{ cfm}} \frac{\Lambda(\alpha,\overline{f} \times E_{\beta})\cdot f}{\lambda_{s_f}-\lambda_w}+\frac{1}{4\pi i}\int_{(1/2)}\Lambda (\overline{s},E_{\alpha}\times E_{\beta})\cdot \frac{E_s}{\lambda_{s}-\lambda_w}\,ds$$
where $\mathbbm{1}_{\alpha=\beta}=  \left\{
     \begin{array}{lr}
       1 & \text{if }\alpha=\beta\\
       0 &  \text{if }\alpha\neq\beta
     \end{array}
   \right.$ and  $C_{\alpha}= \frac{d}{ds}c_s\Big|_{s=\alpha}$.\\ }\\

\tab The proof of this result will be given in Sections 2, 3, 4, and 5 where we will construct the solution.  Theorem \ref{cuspthm} in Section \ref{cusp} calculates the cuspidal spectrum, Theorem \ref{contthm} in Section \ref{cont} calculates the continuous spectrum and Theorem \ref{resthm} in Section \ref{res} calculates the residual spectrum.  In these sections, we will follow the regime presented in the proof of Theorem \ref{existencethm} and the final solution will be obtained by division in Section \ref{sol}. Finally, at the end of Section 5, we will prove that the solution can be meromorphically continued in $w$ to $\text{Re}(w)<1/2$.\\

\tab  Before we turn to the derivation of the solution, we will address what appear to be oddities at some of the borderline cases in $\mathcal{C}$. \\

 \subsection{Limits in $\alpha$ and $\beta$}\label{limits}
 
  \tab One would expect that the equality regimes ($\alpha=\beta$) presented in Theorem \ref{existencethm} can be recognized as a limits if those of $\text{Re}(\alpha)=\text{Re}(\beta)$ and this is in fact the case due to the addiction of the constant $ \frac{\pi}{3}C_{\alpha}$.\\
  
  {\lem\label{a=b} $\displaystyle \lim_{\beta\to\alpha} c_{\alpha}E_{1-\alpha+\beta}+c_{\beta}\cdot E_{1+\alpha-\beta}= - \frac{\pi}{3}C_{\alpha}+2c_{\alpha}E_1^*$ where $C_{\alpha}= \frac{d}{ds}c_s\Big|_{s=\alpha}$.}
  
  \begin{proof} Recall that $E_s$ has a simple pole at $s=1$ and thus the Laurent expansion for $E_s$ is given by
  $$E_s= \frac{a_{-1}}{s-1} + a_o + a_1(s-1) + a_2 (s-1)^2 +\dots$$ Using this we have
$$\lim_{\beta\to\alpha} c_{\alpha}E_{1-\alpha+\beta}+c_{\beta}\cdot E_{1+\alpha-\beta}$$
$$=\lim_{\beta\to\alpha} c_{\alpha}\cdot\left(\frac{a_{-1}}{\beta-\alpha} + a_o + a_1(\beta-\alpha)  +\dots\right)+c_{\beta}\left(\frac{a_{-1}}{\alpha-\beta} + a_o + a_1(\alpha-\beta)  +\dots\right)$$
$$=\lim_{\beta\to\alpha} c_{\alpha}\cdot\left(\frac{a_{-1}}{\beta-\alpha} + a_o + a_1(\beta-\alpha)  +\dots\right)+c_{\beta}\left(-\frac{a_{-1}}{\beta-\alpha} + a_o - a_1(\beta-\alpha) -\dots\right)$$
$$=\displaystyle\lim_{\beta\to\alpha} \frac{a_{-1}(c_{\alpha}-c_{\beta})}{\beta-\alpha} +2c_{\alpha}a_o 
=- \frac{\pi}{3} \frac{d}{ds}c_s\Big|_{s=\alpha}+2c_{\alpha}a_o = - \frac{\pi}{3}C_{\alpha}+2c_{\alpha}E_1^* $$
where $C_{\alpha}= \frac{d}{ds}c_s\Big|_{s=\alpha}$
since 
$\displaystyle\lim_{\beta\to\alpha} \frac{a_{-1}(c_{\alpha}-c_{\beta})}{\beta-\alpha} = -a_{-1}\frac{d}{ds}c_s\Big|_{s=\alpha}=- \frac{\pi}{3} \frac{d}{ds}c_s\Big|_{s=\alpha} $.\\ \end{proof}
  
\tab  We can now express the equality case of {\bf (III)} a limit of case {\bf (II)}. Suppose as in case {\bf (II)}, $1/2\geq \text{Re}(\alpha)\leq \text{Re}(\beta)<  \text{Re}(\alpha)+1/2$ where $\alpha\neq \beta$.\\

\tab If we also suppose as in {\bf (IIa)} that $\text{Re}(\alpha+\beta)>3/2$ then
$$S=E_{\alpha}\cdot E_{\beta}-\left(E_{\alpha+\beta}+c_{\alpha}\cdot E_{1-\alpha+\beta}+c_{\beta}\cdot E_{1+\alpha-\beta} \right).$$  As $\beta\to \alpha$ the first two terms become $E_{\alpha^2}- E_{2\alpha}$.  Thus   $\beta\to \alpha$, when $\text{Re}(\alpha)>3/4$, we get that $$S\to  (E_{\alpha})^2-  E_{2\alpha}-2c_{\alpha}E_1^*+\frac{\pi}{3}C_{\alpha}.$$
 Similarly, if we suppose as in {\bf (IIb)} that $\text{Re}(\alpha+\beta)<3/2$ then
$$S=E_{\alpha}\cdot E_{\beta}-\left(E_{\alpha+\beta}+c_{\alpha}\cdot E_{1-\alpha+\beta}+c_{\beta}\cdot E_{1+\alpha-\beta} +  c_{\alpha}c_{\beta}E_{2-\alpha-\beta}\right).$$    Thus   $\beta\to \alpha$, when $\text{Re}(\alpha)>3/4$, we get that $$S\to  (E_{\alpha})^2-  E_{2\alpha}-2c_{\alpha}E_1^* - c_{\alpha}^2E_{2-2\alpha}+\frac{\pi}{3}C_{\alpha}.$$

\tab However, despite this nice continuity where near where $\alpha = \beta$, one can see that we are not guaranteed the existence of the solution when $\alpha = \beta$ and $\text{Re}(\alpha)=3/4$.  In fact, the strategy presented in Theorem \ref{existencethm} breaks down.  When $\text{Re}(\alpha)=3/4$ we have  $E_{\alpha}^2 = y^{3/2+2\text{Im}(\alpha)i}+2c_{\alpha}y+c_{\alpha}^2 y^{1/2-2\text{Im}(\alpha)i}+R$ and subtracting $E_{2-2\alpha}$ for example will cause the first term to vanish but will also introduce a new non-$L^2$ term $y^{1/2-2\text{Im}(\alpha)i}$ to appear.  In fact, we have the following results which only guarantee the existence of a solution under certain conditions.

{\thm\label{weirdexist}\hfill

\begin{enumerate}[(i)]
\item  In $\text{Re}(w)>1/2$, for $\alpha\neq\beta$, $1/2\leq \text{Re}(\alpha)\leq \text{Re}(\beta)< \text{Re}(\alpha)+1/2$ and $\text{Re}(\alpha+\beta)=3/2$, $\displaystyle (\Delta-\lambda)u=E_{\alpha}E_{\beta}\phantom{e}\text{on}\phantom{e}\Gamma\backslash\mathfrak{H}$ has a unique solution in $H^{-\infty}(\Gamma\backslash \mathfrak{H})\oplus\mathcal{E}(\Gamma\backslash \mathfrak{H})$ when $2\alpha -1$  or $2\beta -1$ is a nontrivial zero of $\zeta(s)$. If it is also the case that $\text{Re}(\beta)= \text{Re}(\alpha)+1/2$,   $\displaystyle (\Delta-\lambda)u=E_{\alpha}E_{\beta}\phantom{e}\text{on}\phantom{e}\Gamma\backslash\mathfrak{H}$ has a unique solution in $H^{-\infty}(\Gamma\backslash \mathfrak{H})\oplus\mathcal{E}(\Gamma\backslash \mathfrak{H})$ when $2\alpha -1$  is a nontrivial zero of $\zeta(s)$.
 
 \item In $\text{Re}(w)>1/2$, for  $\text{Re}(\alpha)=3/4$, $\displaystyle (\Delta-\lambda)u=E_{\alpha}^2\phantom{e}\text{on}\phantom{e}\Gamma\backslash\mathfrak{H}$ has a unique solution in $H^{-\infty}(\Gamma\backslash \mathfrak{H})\oplus\mathcal{E}(\Gamma\backslash \mathfrak{H})$ when $2\alpha -1$ is a nontrivial zero of $\zeta(s)$.\\
 
 \end{enumerate}}
 
\tab  Before we proceed with the proof, it should be noted that in the theorem above (ii) is a special instance of (i).  However, we will provided a proof of both for a few reasons. One reason being that we will be using limits from the left and right of the solutions previously found and the solutions appear to be slightly different for $\alpha=\beta$ versus $\alpha\neq \beta$ (since the $S$'s constructed are differently) even though there limits are equal.  However, the main reason is that it is easier to follow the argument in the $\alpha=\beta$ case and then see how it extends to the inequality case.\\

\begin{proof} As shown in Theorem \ref{existencethm}, to demonstrate the existence and uniqueness of the solution, it suffices to construct appropriate $S$ in $L^2(\Gamma\backslash\h)$.


\tab We will begin with a proof of (ii) since it is a simplified case of (i) and exemplifies the same general phenomenon.  Observe that in the regime where $\alpha=\beta$, the $S$ given by {\bf (IIIa)} and {\bf (IIIb)} differ only by one term $ c_{\alpha}^2E_{2-2\alpha}$. This implies that we cannot have a simultaneous solution corresponding to both 
$$S= (E_{\alpha})^2-  E_{2\alpha}-2c_{\alpha}E_1^*+\frac{\pi}{3}C_{\alpha}$$
and 
$$S= (E_{\alpha})^2-  E_{2\alpha}-2c_{\alpha}E_1^* - c_{\alpha}^2E_{2-2\alpha}+\frac{\pi}{3}C_{\alpha}$$ at $\text{Re}(\alpha)=3/4$ since their difference $\displaystyle\frac{c_{\alpha}^2E_{2-2\alpha}}{\lambda_{2-2\alpha}-\lambda_w}$ is not in $L^2(\Gamma\backslash\h)$. In fact, in general, neither of these contrived $S$'s will be in $L^2(\Gamma\backslash\h)$ in general since:

\tab In regime {\bf (IIIa)},
$$S= (E_{\alpha})^2-  E_{2\alpha}-2c_{\alpha}E_1^* + \frac{\pi}{3}C_{\alpha}$$
$$=c_{\alpha}^2y^{2-2\alpha}-c_{2\alpha}y^{1-2\alpha}+2c_{\alpha}\frac{3}{\pi }\log y+C + \frac{\pi}{3}C_{\alpha}+R $$
and $y^{2-2\alpha}\notin L^2(\Gamma\backslash\mathfrak{H})$ for $\text{Re}(\alpha)=3/4$. Thus we would need $c_{\alpha}^2 = 0$ in order for $S$ to be in $L^2$.  Recall that $c_\alpha = \frac{\xi(2-2\alpha)}{\xi(2\alpha)}= \frac{\xi(2\alpha-1)}{\xi(2\alpha)}$.  For $\alpha=3/4+it$, this yields $\xi(2\alpha-1) =  \xi (1/2+2it)$.  Then $S\in L^2(\Gamma\backslash\mathfrak{H})$ when $2\alpha-1$ is a nontrivial zero of $\xi$.  Thus a solution to $\displaystyle (\Delta-\lambda)u=E_{\alpha}^2\phantom{e}\text{on}\phantom{e}\Gamma\backslash\mathfrak{H}$ exists when $2\alpha-1$ is a nontrivial zero of $\zeta$.

\tab In regime {\bf (IIIb)}, 
$$S= (E_{\alpha})^2-  E_{2\alpha}-2c_{\alpha}E_1^*-c_{\alpha}^2E_{2-2\alpha}+\frac{\pi}{3}C_{\alpha}$$
$$=-c_{2\alpha}y^{1-2\alpha}+2c_{\alpha}\frac{3}{\pi }\log y
-c_{\alpha}^2c_{2-2\alpha}y^{2\alpha-1}+C+ \frac{\pi}{3}C_{\alpha}+R$$
and  $y^{2\alpha-1}\notin L^2(\Gamma\backslash\mathfrak{H})$ for $\text{Re}(\alpha)=3/4$.  Thus we would need either $c_{\alpha}^2 = 0$ or $c_{2-2\alpha}=0.$  Observe that $c_{2-2\alpha}=0$ when $\xi(4ti)=0$ for $\alpha=3/4+it$.  Since $\xi$ have no zeros on the imaginary axis, we need only consider where $c_{\alpha}^2 = 0$.  As above, a solution to $\displaystyle (\Delta-\lambda)u=E_{\alpha}^2\phantom{e}\text{on}\phantom{e}\Gamma\backslash\mathfrak{H}$ exists when $2\alpha-1$ is a nontrivial zero of $\zeta$.

\tab Since, as previously stated, these solutions, given by regime {\bf (IIIa)} and {\bf (IIIb)} may not be distinct.  In fact, for a solution to exist, we need $c_{\alpha}^2E_{2-2\alpha}\to 0$ as  $\text{Re}(\alpha)\to 3/4^-$.  This will happen when  $c_{\alpha}^2 =0$ or when $E_{2-2\alpha}=0$.\footnote{Note that when  $E_{2-2\alpha}=0$, the limits of $S$ from the left and right of $\text{Re}(\alpha)=3/4$ will be equal but that neither $S$ will be in $L^2(\Gamma\backslash\mathfrak{H})$.}  When $c_{\alpha}^2 =0$ we see that the limit of the solution from each side of $\text{Re}(\alpha)=3/4$ will approach the above solution at  $\text{Re}(\alpha)=3/4$.   \\


\tab Now let's turn to case (i).  Observe that in the regime where $\alpha\neq \beta$ and $1/2\leq\text{Re}(\alpha)\leq \text{Re}(\beta)<  \text{Re}(\alpha)+1/2$ , the $S$ given by {\bf (IIa)} and {\bf (IIb)} differ only by one term $c_{\alpha}c_{\beta}\cdot E_{2-\alpha-\beta}$. This implies that we cannot have a simultaneous solution corresponding to both 
$$S=E_{\alpha}\cdot E_{\beta} -  E_{\alpha+\beta}-c_{\beta}\cdot E_{1+\alpha-\beta}-c_{\alpha}\cdot E_{1-\alpha+\beta} $$
and 
$$S= E_{\alpha}\cdot E_{\beta}-   E_{\alpha+\beta}-c_{\beta}\cdot E_{1+\alpha-\beta}-c_{\alpha}\cdot E_{1-\alpha+\beta} -c_{\alpha}c_{\beta}\cdot E_{2-\alpha-\beta}$$ at $\text{Re}(\alpha+\beta)=3/2$ since their difference $\displaystyle\frac{c_{\alpha}c_{\beta}\cdot E_{2-\alpha-\beta}}{\lambda_{2-\alpha-\beta}-\lambda_w}$ is not in $L^2(\Gamma\backslash\h)$. In fact, in general, neither of these contrived $S$'s will be in $L^2(\Gamma\backslash\h)$ in general since:

\tab In regime {\bf (IIa)},
$$S= E_{\alpha}\cdot E_{\beta} -  E_{\alpha+\beta}-c_{\beta}\cdot E_{1+\alpha-\beta}-c_{\alpha}\cdot E_{1-\alpha+\beta} $$
$$=c_{\alpha}c_{\beta}  y^{2-\alpha-\beta}-c_{\alpha+\beta} y^{1-\alpha-\beta}
-c_{\beta}c_{1+\alpha-\beta} y^{-\alpha+\beta}
-c_{\alpha}c_{1-\alpha+\beta} y^{\alpha-\beta}+R$$
and $y^{2-\alpha-\beta}\notin L^2(\Gamma\backslash\mathfrak{H})$ for $\text{Re}(\alpha +\beta)=3/2$. Thus we would need $c_{\alpha} = 0$  or $c_{\beta}=0$ in order for $S$ to be in $L^2$.

\tab In regime {\bf (IIb)}, 
$$S=E_{\alpha}\cdot E_{\beta} -  E_{\alpha+\beta}-c_{\beta}\cdot E_{1+\alpha-\beta}-c_{\alpha}\cdot E_{1-\alpha+\beta} $$
$$=-c_{\alpha+\beta} y^{1-\alpha-\beta}
-c_{\beta}c_{1+\alpha-\beta} y^{-\alpha+\beta}-c_{\alpha}c_{1-\alpha+\beta} y^{\alpha-\beta}
-c_{\alpha}c_{\beta}c_{2-\alpha-\beta} y^{\alpha+\beta-1}+R$$
and  $y^{\alpha+\beta-1}\notin L^2(\Gamma\backslash\mathfrak{H})$ for $\text{Re}(\alpha+\beta)=3/2$.  Thus we would need either $c_{\alpha}= 0$, $c_{\beta}=0$  or $c_{1-\alpha+\beta}=0.$

\tab Thus solution to $\displaystyle (\Delta-\lambda)u=E_{\alpha}\cdot E_{\beta}\phantom{e}\text{on}\phantom{e}\Gamma\backslash\mathfrak{H}$ exists when $2\alpha-1$ or $2\beta-1$is a nontrivial zero of $\zeta$. Furthermore, when $\text{Re}(\beta)= \text{Re}(\alpha)+1/2$, we also have $y^{\alpha-\beta}\notin L^2(\Gamma\backslash\mathfrak{H})$.  We will then need either $c_{\alpha}=0$ or $c_{1-\alpha +\beta}=0$.  However, in this case, we also have $c_{1-\alpha +\beta}=c_{3/2+it}\neq 0$.

\tab Since, as previously stated, these solutions, given by regime {\bf (IIa)} and {\bf (IIb)} may not be distinct.  In fact, for a solution to exist, we need $c_{\alpha}c_{\beta}\cdot E_{2-\alpha-\beta}\to 0$ as  $\text{Re}(\alpha+\beta)\to 3/2^-$.  This will happen when  $c_{\alpha} =0$,  $c_{\beta} =0$ or when $E_{2-\alpha-\beta}=0$. \footnote{Note that when  $E_{2-\alpha-\beta}=0$, the limits of $S$ from the left and right of $\text{Re}(\alpha+\beta)=3/2$ will be equal but that neither $S$ will be in $L^2(\Gamma\backslash\mathfrak{H})$.}  When $c_{\alpha} =0$ or when $c_{\beta}=0$ we see that the limit of the solution from each side of $\text{Re}(\alpha+\beta)=3/2$ will approach the above solution at  $\text{Re}(\alpha+\beta)=3/2$.   \\

\end{proof}

\tab The solution on these regions where say  $\text{Re}(\alpha+\beta)=3/2$ will present itself as a limit and will thus be identified with the corresponding limit of the solution in Theorem \ref{main}. Explicitly, when  $\alpha\neq\beta$, $1/2\leq \text{Re}(\alpha)\leq \text{Re}(\beta)< \text{Re}(\alpha)+1/2$ and $\text{Re}(\alpha+\beta)=3/2$ and $2\alpha-1$ is a zero of $\zeta(s)$ (i.e. $c_{\alpha}=0$), $$S=E_{\alpha}\cdot E_{\beta} -  E_{\alpha+\beta}-c_{\beta}\cdot E_{1+\alpha-\beta}$$ is in $L^2$ for $\text{Re}(\alpha+\beta)=3/2$. In this case, the equation  $\displaystyle (\Delta-\lambda)u=E_{\alpha}\cdot E_{\beta}\phantom{e}\text{on}\phantom{e}\Gamma\backslash\mathfrak{H}$ has a unique solution in $H^{-\infty}(\Gamma\backslash \mathfrak{H})\oplus\mathcal{E}(\Gamma\backslash \mathfrak{H})$ with spectral expansion which lies in $H^2(\Gamma\backslash \mathfrak{H})\oplus \mathcal{E}(\Gamma\backslash \mathfrak{H})$ and is given by
$$u_w= \frac{  E_{\alpha+\beta}}{\lambda_{\alpha+\beta}-\lambda_w}+\frac{c_{\beta}\cdot E_{1+\alpha-\beta}}{\lambda_{1+\alpha-\beta}-\lambda_w}+\sum_{f\text{ cfm}} \frac{\Lambda(\alpha,\overline{f} \times E_{\beta})\cdot f}{\lambda_{s_f}-\lambda_w}+\frac{1}{4\pi i}\int_{(1/2)}\Lambda (\overline{s},E_{\alpha}\times E_{\beta})\cdot \frac{E_s}{\lambda_{s}-\lambda_w}\,ds.$$\\
 
 \tab We will conclude this section by showing that there are no solutions in $H^{2}(\Gamma\backslash \mathfrak{H})\oplus\mathcal{E}(\Gamma\backslash \mathfrak{H})$ on the lines $\text{Re}(\alpha+\beta)=3/2$ or $\text{Re}(\alpha)=3/4$ where neither $c_\alpha$ nor $c_\beta$ are zero.  We will need the following preliminary results in order to establish the other direction of the implied biconditional.\\

   { \lem\label{notinL2} Let $\xi_i,\dots,\xi_n$ be distinct real numbers and $\sigma_1,\dots,\sigma_n$ real.   For non-zero complex $c_1,\dots , c_n$, the function $f(y)=\sum_j c_j y^{\sigma_j+i\xi_j}$ is in $L^2([1,\infty), \frac{dy}{y^2})$ for if and only if $\sigma_j<1/2$ for all $j$.}
   
   \begin{proof} If $\mu:=\max_j \sigma_j<1/2$, then $f(y)\in L^2([1,\infty), \frac{dy}{y^2})$.
   
   \tab On the other hand, suppose that $\mu=1/2$.  Observe that 
   $$\left|\sum_j c_j y^{\sigma_j+i\xi_j}\right|^2_{L^2([1,\infty), \frac{dy}{y^2})} 
  = \lim_{b\to\infty}\int_1^{b} \left|\sum_j c_j y^{\sigma_j+i\xi_j}\right|^2\,\frac{dy}{y^2} 
   = \lim_{b\to\infty}\int_1^{b} \sum_{j,k} c_j\overline{c}_k\, y^{i(\xi_j-\xi_k)}\,y^{\sigma_{j}+\sigma_{k}}\,\frac{dy}{y^2}. $$  If $\mu=1/2$ then all the terms $ c_j\overline{c}_k\, y^{i(\xi_j-\xi_k)}\,y^{\sigma_{j}+\sigma_{k}}$ are in $L^1$ except for possibly the sum over $j,k$ with $\sigma_j=1/2=\sigma_k$.

  \tab Suppose now that $\mu=1/2$ and $\sigma_j=1/2=\sigma_k$.  Among the tails for the improper integral for the $L^2$-norm-squared integrals are 
   $$\sum_{j,k} c_j\overline{c}_k\, \int_T^{T^2} y^{i(\xi_j-\xi_k)}\,\frac{dy}{y}. $$
   For $j=k$, the term is $|c_j|^2\cdot\log T$.  For $j\neq k$, the term is 
   $\displaystyle c_j\overline{c}_k\frac{(T^2)^{i(\xi_j-\xi_k)}-T^{i(\xi_j-\xi_k)}}{i(\xi_j-\xi_k)}$.  The sum of the $j\neq k$ is uniformly bounded in $T$.  The sum of the $j=k$ term is a strictly positive real multiple of $\log T$ and goes to $\infty$ as $T\to\infty$.  Thus an expression of the form $\sum_j c_j y^{1/2+i\xi_j}$ will be in $L^2([1,\infty), \frac{dy}{y^2})$ only when $c_j=0$ for each $j$. Furthermore, if $\mu=1/2$ then $f(y)$ cannot be in $L^2([1,\infty), \frac{dy}{y^2})$.
   
\tab Finally, in the case of $\mu>1/2$, $y^{1/2-\mu}\cdot f(y)$ is in $L^2([1,\infty), \frac{dy}{y^2})$ if $f(y)$ is and this reduces to the case where $\mu=1/2$ just treated.
   
   \end{proof}

  {\lem\label{lem7} \hfill
 
\begin{enumerate}[(i)]
\item  For $\alpha\neq\beta$, $1/2\leq \text{Re}(\alpha)\leq \text{Re}(\beta)< \text{Re}(\alpha)+1/2$ and $\text{Re}(\alpha+\beta)=3/2$, then $E_\alpha E_\beta \notin L^2(\Gamma\backslash \mathfrak{H})\oplus\mathcal{E}(\Gamma\backslash \mathfrak{H})$ unless $2\alpha -1$ or  $2\beta-1$ is a zero of $\zeta(s)$.
 \item For  $\text{Re}(\alpha)=3/4$, then $E_\alpha ^2\notin L^2(\Gamma\backslash \mathfrak{H})\oplus\mathcal{E}(\Gamma\backslash \mathfrak{H})$ unless $2\alpha -1$ is a zero of $\zeta(s)$.\end{enumerate}}
 
 \begin{proof}  We will again first establish the result of (ii) first.  Assume $\alpha = \beta$ and  $\text{Re}(\alpha)=3/4$. We have $E_{\alpha}^2 = y^{3/2+2\text{Im}(\alpha)i}+2c_{\alpha}y+c_{\alpha}^2 y^{1/2-2\text{Im}(\alpha)i}+R$.   Subtracting $E_{2\alpha}$ and $2c_\alpha E_1^*$ will eliminate the first term terms and what remains will be in $L^2$ with the exception of the term $c_{\alpha}^2 y^{1/2-2\text{Im}(\alpha)i}$.   Subtracting $c_{\alpha}^2E_{2-2\alpha}$ will cause the last term to vanish but will also introduce a new non-$L^2$ term $c_{\alpha}^2  c_{2-2\alpha}y^{1/2+2\text{Im}(\alpha)i}$ to appear.  Furthermore, observe that $ c_{2-2\alpha} $ cannot be zero for $\text{Re}(\alpha)=3/4$ since $\zeta(s)$ has no zeros on the line $\text{Re}(s)=1$.  More formally, the non-rapidly decreasing terms of $\mathcal{E}(\Gamma\backslash\h)$ can be written as linear combinations of the form $\sum_j c_j y^{\sigma_j+i\xi_j}$ and so by Lemma \ref{notinL2} $c_{\alpha}^2 y^{1/2-2\text{Im}(\alpha)i}+\sum_j c_j y^{\sigma_j+i\xi_j}$ is not in $L^2$ except when $c_{\alpha}=0$.  Thus the only way for 
 $E_\alpha^2$ to be in $L^2(\Gamma\backslash \mathfrak{H})\oplus\mathcal{E}(\Gamma\backslash \mathfrak{H})$ is by $c_{\alpha}$ being 0 and thus it is necessary that $\zeta(2\alpha-1)=0$.
 
 \tab For (i), assume  $\alpha\neq\beta$, $1/2\leq \text{Re}(\alpha)\leq \text{Re}(\beta)< \text{Re}(\alpha)+1/2$ and $\text{Re}(\alpha+\beta)=3/2$. Then $E_{\alpha}E_{\beta}=y^{3/2+\text{Im}(\alpha+\beta)i}+c_{\beta} y^{1+\alpha-\beta}+c_{\alpha} y^{1-\alpha+\beta}+c_{\alpha}c_{\beta}  y^{1/2-\text{Im}(\alpha+\beta)i}+R$.  Again, the first three terms can be eliminated putting what remains in $L^2$ with the exception of the term $c_{\alpha}c_{\beta}  y^{1/2-\text{Im}(\alpha+\beta)i}$.  Subtracting $c_{\alpha}c_{\beta}E_{2-\alpha-\beta}$ will cause the first term to vanish but will also introduce a new non-$L^2$ term $c_{\alpha}c_{\beta}c_{2-\alpha-\beta}y^{1/2+\text{Im}(\alpha+\beta)i}$ to appear. Again $ c_{2-\alpha-\beta} $ cannot be zero for $\text{Re}(\alpha+\beta)=3/4$ since $\zeta(s)$ has no zeros on the line $\text{Re}(s)=1$.  Furthermore, by Lemma \ref{notinL2} $c_{\alpha}c_{\beta}c_{2-\alpha-\beta}y^{1/2+\text{Im}(\alpha+\beta)i}+\sum_j c_j y^{\sigma_j+i\xi_j}$ is not in $L^2$ except when $c_{\alpha}=0$ or $c_{\beta}=0$. Thus the only way for 
 $E_\alpha E_\beta$ to be in $L^2(\Gamma\backslash \mathfrak{H})\oplus\mathcal{E}(\Gamma\backslash \mathfrak{H})$ is by $c_{\alpha}$ or $c_\beta$ being 0 and thus it is necessary that $\zeta(2\alpha-1)=0$ or $\zeta(2\beta-1)=0$.

 \end{proof}

 {\lem\label{lem8} If there exists a solution $u$ to $\displaystyle (\Delta-\lambda)u=E_{\alpha}\cdot E_{\beta}\phantom{e}\text{on}\phantom{e}\Gamma\backslash\mathfrak{H}$ in $H^2(\Gamma\backslash \mathfrak{H})\oplus\mathcal{E}(\Gamma\backslash \mathfrak{H})$ then $E_\alpha E_\beta \in L^2(\Gamma\backslash \mathfrak{H})\oplus\mathcal{E}(\Gamma\backslash \mathfrak{H})$. }
 
 \begin{proof} Suppose $u$ is a solution to $\displaystyle (\Delta-\lambda)u=E_{\alpha}\cdot E_{\beta}\phantom{e}\text{on}\phantom{e}\Gamma\backslash\mathfrak{H}$ and $u \in H^2(\Gamma\backslash \mathfrak{H})\oplus\mathcal{E}(\Gamma\backslash \mathfrak{H})$.  Say $u= f+\sum_k a_k F_{s_k}$ for $f\in H^2(\Gamma\backslash \mathfrak{H})$ and $\sum_k a_k F_{s_k}\in \mathcal{E}(\Gamma\backslash \mathfrak{H})$.  Then 
 $$E_\alpha\cdot E_\beta = (\Delta-\lambda_w)u = (\Delta-\lambda_w)\left( f+\sum_k a_k F_{s_k}\right)$$
 $$\phantom{we}= (\Delta-\lambda_w)f +  \sum_k a_k(\lambda_{s_k}-\lambda_w) F_{s_k}$$
 Thus
 $$E_\alpha\cdot E_\beta -\sum_k a_k (\lambda_{s_k}-\lambda_w)F_{s_k}
 = (\Delta-\lambda_w)f\in H^0(\Gamma\backslash \mathfrak{H}) = L^2(\Gamma\backslash \mathfrak{H})$$
 
 since $f\in H^2(\Gamma\backslash \mathfrak{H})$. 
 
 \end{proof}

\tab Combining the last two results, we see that if there were a solution $u$ in $H^2(\Gamma\backslash \mathfrak{H})\oplus\mathcal{E}(\Gamma\backslash \mathfrak{H})$ it would be contrived as above and thus there is no such solutions on $\text{Re}(\alpha+\beta)=3/2$ where neither $2\alpha-1$ nor $2\beta-1$ is a zero of $\zeta$.

{\thm \hfill

\begin{enumerate}[(i)]
\item  In $\text{Re}(w)>1/2$, for $\alpha\neq\beta$, $1/2\leq \text{Re}(\alpha)\leq \text{Re}(\beta)< \text{Re}(\alpha)+1/2$ and $\text{Re}(\alpha+\beta)=3/2$, $\displaystyle (\Delta-\lambda)u=E_{\alpha}E_{\beta}\phantom{e}\text{on}\phantom{e}\Gamma\backslash\mathfrak{H}$ has a unique solution in $H^{2}(\Gamma\backslash \mathfrak{H})\oplus\mathcal{E}(\Gamma\backslash \mathfrak{H})$ if and only if $2\alpha -1$  or $2\beta -1$ is a nontrivial zero of $\zeta(s)$. 
 \item In $\text{Re}(w)>1/2$, for  $\text{Re}(\alpha)=3/4$, $\displaystyle (\Delta-\lambda)u=E_{\alpha}^2\phantom{e}\text{on}\phantom{e}\Gamma\backslash\mathfrak{H}$ has a unique solution in $H^{2}(\Gamma\backslash \mathfrak{H})\oplus\mathcal{E}(\Gamma\backslash \mathfrak{H})$ if and only if $2\alpha -1$ is a nontrivial zero of $\zeta(s)$.\\

 \end{enumerate}}

 \tab The proof of this result follows directly from Theorem \ref{weirdexist} in conjunction with Lemma \ref{lem7} and Lemma \ref{lem8}.

 \vspace{.5cm}

\section{The Cuspidal Spectrum}\label{cusp}

\tab We will now compute the cuspidal spectrum for the expansion of the solution. Let $f$ be a cuspform with Fourier expansion $$f(z)=\sum_{n\neq 0} c_n \cdot W_s(|n|y)\cdot e^{2\pi i n x}$$ 

{\thm\label{cuspthm}  For $f$ a cuspform and $\alpha, \beta\in \mathcal{C}$, $$ \langle S,f\rangle_{L^2}=  L(\alpha,\overline{f} \times E_{\beta})\cdot 
\frac{\pi^{\beta+\overline{s}-\alpha}}{2\,\Gamma(\beta)\Gamma(\overline{s})}\cdot
\frac{\Gamma(\frac{\alpha+\beta -\overline{s}}{2})\Gamma(\frac{\alpha-\beta +\overline{s}}{2})\Gamma(\frac{\alpha+1-\beta -\overline{s}}{2})\Gamma(\frac{\alpha-1+\beta +\overline{s}}{2})}{\Gamma(\alpha)}$$
$$=\Lambda(\alpha,\overline{f} \times E_{\beta})$$ for each $S$ proposed in Theorem \ref{existencethm}.}\\

\tab The proof of this result is given in the what remains of this section.  Before we investigate each case for each different $S$, we will first perform two useful computations.  Many examples of the following computations can be found in relevant literature -- for example, in \cite{Garrett2015} or \cite{GradRyz}.

{\lem\label{EaEbf}  For each $ \alpha$ and $\beta$,

$$\int_{\Gamma\backslash \mathfrak{H}}E_{\alpha}\cdot E_{\beta}\cdot  \overline{f}\, \frac{dx\,dy}{y^2}=L(\alpha,\overline{f}  \times E_{\beta})\cdot 
\frac{\pi^{\beta+\overline{s}-\alpha}}{2\,\Gamma(\beta)\Gamma(\overline{s})}\cdot
\frac{\Gamma(\frac{\alpha+\beta -\overline{s}}{2})\Gamma(\frac{\alpha-\beta +\overline{s}}{2})\Gamma(\frac{\alpha+1-\beta -\overline{s}}{2})\Gamma(\frac{\alpha-1+\beta +\overline{s}}{2})}{\Gamma(\alpha)}$$
$$=\Lambda(\alpha,\overline{f}  \times E_{\beta})$$}

\begin{proof} The computation that follows we can will begin by examining $1/2 < \text{Re}(\alpha)< \text{Re}(\alpha)+1/2<\text{Re}(\beta)$ since  $\int_{\Gamma\backslash \mathfrak{H}}E_{\alpha}\cdot E_{\beta}\cdot  \overline{f}\, \frac{dx\,dy}{y^2}$ is holomorphic on this region.  Since it extends to meromorphic function of $\alpha$ and $\beta$ (since $f$ is cuspform), we can evaluate it via identity principle by moving $\alpha$ to $\text{Re}(\alpha)>1$  so that we can then unwind $E_{\alpha}$. 

\tab Thus, by unwinding, we have

$$\int_{\Gamma\backslash \mathfrak{H}}E_{\alpha}\cdot E_{\beta}\cdot  \overline{f}\, \frac{dx\,dy}{y^2}=\int_{\Gamma\backslash \mathfrak{H}}\sum_{\gamma\in P\backslash \Gamma}\text{Im}(\gamma z)^{\alpha}\cdot E_{\beta}\cdot  \overline{f}\, \frac{dx\,dy}{y^2}
=\int_{P\backslash \mathfrak{H}}y^{\alpha}\cdot E_{\beta}\cdot  \overline{f}\, \frac{dx\,dy}{y^2}$$
$$=\int_0^{\infty}\int_0^1y^{\alpha}\cdot E_{\beta}\cdot  \overline{f}\, \frac{dx\,dy}{y^2}$$

 since the fundamental domain of $P\backslash \mathfrak{H}$ is $\{z=x+iy\in\mathfrak{H}~|~ 0\leq x\leq 1\}$.
 Now, writing out the Fourier-Whittaker expansions for $E_{\beta}$ and $f$, we have 

$$\int_0^{\infty}\int_0^1y^{\alpha}\cdot \left( c_PE_{\beta}+\sum_{n\neq 0}\varphi(n,\beta)\cdot W_{\beta}(|n|y)\cdot e^{2\pi i n x}\right) \cdot \left( \sum_{m\neq 0 }\overline{c}_m \cdot \overline{W_s}(|m|y)\cdot e^{-2\pi i m x}\right)\, \frac{dx\,dy}{y^2}$$

\hfill where $\varphi$, $W_s$ and $c_m$ are defined in Section \ref{background}

$\displaystyle =\int_0^{\infty}\int_0^1y^{\alpha}\Big[ c_PE_{\beta}\cdot\sum_{m\neq 0}  \overline{c}_m \cdot \overline{W_s}(|m|y)\cdot e^{-2\pi i m x}\phantom{wwwwweeeeeppepeeee}$
$$\phantom{wwwweeeeeeepeeeeee}+\left(\sum_{n\neq 0}\varphi(n,\beta)\cdot W_{\beta}(|n|y)\cdot e^{2\pi i n x}\right) \cdot \left( \sum_{m\neq 0}\overline{c}_m \cdot \overline{W_s}(|m|y)\cdot e^{-2\pi i m x}\right)\Big]\, \frac{dx\,dy}{y^2}$$

$\displaystyle =\int_0^{\infty}y^{\alpha}\Big[c_PE_{\beta}\cdot\sum_{m\neq 0} \overline{c}_m \cdot \overline{W_s}(|m|y) \cdot \int_0^1 e^{-2\pi i m x}\, dx\phantom{wwweeeeee}$
$$\phantom{wweeeeeeeeee}+\sum_{m,n\neq 0}\varphi(n,\beta)W_{\beta}(|n|y)\cdot\overline{c}_m \overline{W_s}(|m|y) \int_0^1e^{2\pi i (n-m) x}\, dx \Big]\, \frac{dy}{y^2}$$

$$\displaystyle=\int_0^{\infty}y^{\alpha}\Big[ c_PE_{\beta}\cdot\sum_{m\neq 0}  \overline{c}_m \cdot \overline{W_s}(|m|y)\cdot\delta_{0,m}+\sum_{m,n\neq 0}\varphi(n,\beta)W_{\beta}(|n|y)\cdot\overline{c}_m \overline{W_s}(|m|y)\cdot \delta_{n,m} \Big]\, \frac{dy}{y^2}$$

We see that the sum is zero when $n\neq m$ (furthermore, since $f$ is a cuspform the $n=0$ term vanishes) and we get

$$\int_0^{\infty}y^{\alpha}\sum_{n\neq 0}\varphi(n,\beta)W_{\beta}(|n|y)\cdot\overline{c}_m \overline{W_s}(|m|y) \, \frac{dy}{y^2}
=\sum_{n\neq 0}\varphi(n,\beta)\cdot \overline{c}_n\cdot \int_0^{\infty}y^{\alpha} \cdot W_{\beta}(|n|y)\,\overline{W_s}(|n|y) \, \frac{dy}{y^2}$$
Replacing $y$ by $y/n$, we have
$$\sum_{n\neq 0}\frac{\varphi(n,\beta) \cdot\overline{c}_n}{n^{\alpha-1}}\cdot \int_0^{\infty}y^{\alpha}\cdot W_{\beta}(y)\,\overline{W_s}(y) \, \frac{dy}{y^2}=L(\alpha,\overline{f} \times E_{\beta})\cdot \int_0^{\infty}y^{\alpha}\cdot W_{\beta}(y)\,\overline{W_s}(y) \, \frac{dy}{y^2}$$
$$=L(\alpha,\overline{f}  \times E_{\beta})\cdot 
\frac{\pi^{\beta+\overline{s}-\alpha}}{2\,\Gamma(\beta)\Gamma(\overline{s})}\cdot
\frac{\Gamma(\frac{\alpha+\beta -\overline{s}}{2})\Gamma(\frac{\alpha-\beta +\overline{s}}{2})\Gamma(\frac{\alpha+1-\beta -\overline{s}}{2})\Gamma(\frac{\alpha-1+\beta +\overline{s}}{2})}{\Gamma(\alpha)}
=\Lambda(\alpha,\overline{f}  \times E_{\beta})$$
 \end{proof}


{\lem\label{RS=0} For any $r\neq 1$,  $\int_{\Gamma\backslash \mathfrak{H}}E_{r}\cdot \overline{f}~ \frac{dx\,dy}{y^2}=0 $ and $\int_{\Gamma\backslash \mathfrak{H}}E_{1}^*\cdot \overline{f}~ \frac{dx\,dy}{y^2}=0 $. }
\begin{proof} Since the integrals extend to meromorphic function of $r$ (since $f$ is cuspform), we can evaluate it via identity principle by moving $r$ to $\text{Re}(r)>1$  so that we can then unwind $E_{r}$. 
$$\int_{\Gamma\backslash \mathfrak{H}}E_{r}\cdot \overline{f}~ \frac{dx\,dy}{y^2}
=\int_{\Gamma\backslash \mathfrak{H}}\,\sum_{\gamma\in P\backslash \Gamma}\text{Im}(\gamma z)^{r}\,\overline{f} ~ \frac{dx\,dy}{y^2}
=\int_{P\backslash \mathfrak{H}} \text{Im}(z)^{r}\, \overline{f}~ \frac{dx\,dy}{y^2}$$ 
where the fundamental domain of $P\backslash \mathfrak{H}$ is 
$\{z=x+iy\in\mathfrak{H}~|~ 0\leq x\leq 1\}$

$$=\int_{P\backslash \mathfrak{H}}y^{r}\overline{f}(z)\, \frac{dx\,dy}{y^2}
=\sum_{n> 0}\overline{c}_n\int_{y>0} y^{r} \cdot \overline{W}_s(|n|y)\left(\int_{0\leq x\leq 1} e^{2\pi i n x}\, dx\right)\frac{dy}{y^2}$$
$$=\sum_{n> 0}\overline{c_n}\int_{y>0} y^{r}\overline{W}_s(|n|y)\delta_{n,0}\, \frac{dy}{y^2}=0$$ 
\end{proof}

Finally, we should note that constants (such as $\frac{\pi}{3}C_{\alpha}$) are orthogonal to cuspforms in $L^2(\Gamma\backslash\h)$ so 
$$\int_{\Gamma\backslash \h} \frac{\pi}{3}C_{\alpha} \cdot \overline{f}\, \frac{dx\,dy}{y^2}=0$$

We can now quickly evaluate each case of $S$ for each $\alpha$ and $\beta$ presented above.  \\

\subsection{Regimes} Recall the regimes set up in the proof of Theorem \ref{existencethm}.  Again, suppose that $\alpha\neq 1$ and $\beta\neq 1$.\\

{\bf (I):} When $1/2 \leq \text{Re}(\alpha)< \text{Re}(\alpha)+1/2<\text{Re}(\beta)$,
$$ \langle S,f\rangle_{L^2}
=\int_{\Gamma\backslash \mathfrak{H}}\left(E_{\alpha}\cdot E_{\beta}-E_{\alpha+\beta}-c_{\alpha}\cdot E_{1-\alpha+\beta}\right)\cdot \overline{ f}\, \frac{dx\,dy}{y^2}$$
$$=\int_{\Gamma\backslash \mathfrak{H}}E_{\alpha}\cdot E_{\beta}\cdot \overline{ f}-E_{\alpha+\beta}\cdot \overline{ f}-c_{\alpha}\cdot E_{1-\alpha+\beta}\cdot \overline{ f}\, \frac{dx\,dy}{y^2}$$\\


{\bf  (II):}  Suppose $1/2\leq\text{Re}(\alpha)
\leq \text{Re}(\beta)<  \text{Re}(\alpha)+1/2$ but that $\alpha\neq\beta$.\\

\tab {\bf (IIa)} If  $\text{Re}(\alpha+\beta)>3/2$ then 
$$ \langle S,f\rangle_{L^2}
=\int_{\Gamma\backslash \mathfrak{H}}\left(E_{\alpha}\cdot E_{\beta}-E_{\alpha+\beta}-c_{\beta}\cdot E_{1+\alpha-\beta}-c_{\alpha}\cdot E_{1-\alpha+\beta}\right)\cdot \overline{ f}\, \frac{dx\,dy}{y^2}$$
$$=\int_{\Gamma\backslash \mathfrak{H}}E_{\alpha}\cdot E_{\beta}\cdot\overline{ f}-E_{\alpha+\beta}\cdot\overline{ f}-c_{\beta}\cdot E_{1+\alpha-\beta}\cdot\overline{ f}-c_{\alpha}\cdot E_{1-\alpha+\beta}\cdot\overline{ f}\, \frac{dx\,dy}{y^2}$$

 \tab {\bf (IIb)} If $\text{Re}(\alpha+\beta)<3/2$ then
$$ \langle S,f\rangle_{L^2}
=\int_{\Gamma\backslash \mathfrak{H}}\left(E_{\alpha}\cdot E_{\beta}-E_{\alpha+\beta}-c_{\beta}\cdot E_{1+\alpha-\beta}-c_{\alpha}\cdot E_{1-\alpha+\beta}-c_{\alpha}c_{\beta}\cdot E_{2-\alpha-\beta}\right)\cdot \overline{ f}\, \frac{dx\,dy}{y^2}$$
$$=\int_{\Gamma\backslash \mathfrak{H}}E_{\alpha}\cdot E_{\beta}\cdot\overline{ f}-E_{\alpha+\beta}\cdot\overline{ f}-c_{\beta}\cdot E_{1+\alpha-\beta}\cdot\overline{ f}-c_{\alpha}\cdot E_{1-\alpha+\beta}\cdot\overline{ f}-c_{\alpha}c_{\beta}\cdot E_{2-\alpha-\beta}\cdot\overline{ f}\, \frac{dx\,dy}{y^2}$$ \\

 
{\bf  (III):} Suppose $\alpha=\beta$.\\

\tab {\bf (IIIa)} Suppose also that $\text{Re}(\alpha)>3/4$ then 
$$ \langle S,f\rangle_{L^2}
=\int_{\Gamma\backslash \mathfrak{H}}\left( (E_{\alpha})^2-  E_{2\alpha}-2c_{\alpha}E_1^* + \frac{\pi}{3}C_{\alpha}\right)\cdot \overline{ f}\, \frac{dx\,dy}{y^2}$$
$$=\int_{\Gamma\backslash \mathfrak{H}} (E_{\alpha})^2\cdot\overline{ f}-  E_{2\alpha}\cdot\overline{ f}-2c_{\alpha}E_1^*\cdot\overline{ f}+ \frac{\pi}{3}C_{\alpha}\cdot\overline{ f}\, \frac{dx\,dy}{y^2}$$

 
\tab {\bf (IIIb)} Now suppose $1/2\leq\text{Re}(\alpha)<3/4$
 then 
$$ \langle S,f\rangle_{L^2}
=\int_{\Gamma\backslash \mathfrak{H}}\left((E_{\alpha})^2-  E_{2\alpha}-2c_{\alpha}E_1^*-c_{\alpha}^2E_{2-2\alpha}+ \frac{\pi}{3}C_{\alpha}\right)\cdot \overline{ f}\, \frac{dx\,dy}{y^2}$$
$$=\int_{\Gamma\backslash \mathfrak{H}}(E_{\alpha})^2\cdot\overline{ f}-  E_{2\alpha}\cdot\overline{ f}-2c_{\alpha}E_1^*\cdot\overline{ f}-c_{\alpha}^2E_{2-2\alpha}\cdot\overline{ f}+ \frac{\pi}{3}C_{\alpha}\cdot\overline{ f}\, \frac{dx\,dy}{y^2}$$

\tab  In each of the above cases, we can use Lemma \ref{EaEbf} to evaluate the integral of the first term and see that each of the remaining terms will integrate to be zero using Lemma \ref{RS=0}.  Thus for each $S$, we get
$$ \langle S,f\rangle_{L^2}=  L(\alpha,\overline{f} \times E_{\beta})\cdot 
\frac{\pi^{\beta+\overline{s}-\alpha}}{2\,\Gamma(\beta)\Gamma(\overline{s})}\cdot
\frac{\Gamma(\frac{\alpha+\beta -\overline{s}}{2})\Gamma(\frac{\alpha-\beta +\overline{s}}{2})\Gamma(\frac{\alpha+1-\beta -\overline{s}}{2})\Gamma(\frac{\alpha-1+\beta +\overline{s}}{2})}{\Gamma(\alpha)}$$
$$=\Lambda(\alpha,\overline{f} \times E_{\beta})$$\\

\section{The Continuous Spectrum}\label{cont}

\tab We want to compute $$\int_{(1/2)}\langle S,E_s\rangle \cdot E_s\,ds$$ for each case of $S$.  Though we have designed $S$ so that $S\in L^2(\Gamma\backslash\h)$, there is no guarantee that  $S\cdot \overline{E_s}$ is in $L^1(\Gamma\backslash\h)$.
However, observe that on $\text{Re}(s)=1/2$,  $\langle S,E_s\rangle$ exists as a literal integral since $S$ is $\mathcal{O}(y^{\frac{1}{2}-\epsilon})$ for some $\epsilon>0$. This can be seen by observing that $E_{s}=y^{s}+c_{s} y^{1-s}+R_{s}$ where $R_{s}$ is rapidly decreasing and so $\overline{E}_s\cdot S$ is $\mathcal{O}(y^{1-\epsilon})$.  Thus $$\displaystyle \int_{\Gamma\backslash \h}  \overline{E}_s\cdot  S\,\frac{dy\,dx}{y^2}<\infty.$$

Futhermore, in what follows we will show
{\thm\label{contthm} For each $\alpha,\beta\in \mathcal{C}$, 
$$\langle S, E_s\rangle_{L^2} =
L(\overline{s},E_{\alpha}\times E_{\alpha})\cdot \int_0^\infty y^{{\overline{s}}}\cdot W_\alpha(y) W_\alpha(y) \,\frac{dy}{y^2} = \Lambda(\overline{s},E_{\alpha}\times E_{\alpha})$$ for each $S$ given in Theorem \ref{existencethm}.}\\

 \tab Knowing that these integrals converge and computing them directly are two different things.  In the style of Zagier \cite{Zagier1982}  and Casselman \cite{Casselman1993}, we will use Arthur truncation to compute these spectral integrals.  To make proper use of the truncated Eisenstein series, we will also  need that the limit of these truncated Eisenstein series converges to the Eisenstein series itself. \\

\subsection{Convergence of Truncated Eisenstein Series}

 Recall that Arthur truncation is defined as $$\displaystyle\wedge^T\! E_s:= E_s - \sum_{\gamma\in P\backslash \Gamma}\tau_s(\gamma z)\phantom{wee}\text{ where }\phantom{wee}\tau_s(z)= \left\{
     \begin{array}{lr}
       y^s+c_s y^{1-s} & y\geq T\\
       0 & y<T
     \end{array}
   \right..$$  
 For convenience we will label the sum $\displaystyle\Theta^T_s(z):=\sum_{\gamma\in P\backslash \Gamma}\tau_s(\gamma z)$ so that $\displaystyle\wedge^T\! E_s:= E_s -\Theta^T_s(z)$.\\


\tab Let $\displaystyle\Psi_{\epsilon}(z):=\sum_{\gamma\in P\backslash \Gamma}\varphi_{\epsilon}\left(\text{Im}(\gamma z)\right)$ be the Eisenstein series where $\varphi_{\epsilon}(y)= \left\{
     \begin{array}{lr}
       y^{\epsilon} & y>1\\
       0 & y<1
     \end{array}
   \right.$ and define $$\mathcal{B}_{\epsilon}^k:=\{ f\in L^2(\Gamma\backslash\h)~|~ \langle (1+ |\Psi_{\epsilon}|)^k f, f\rangle_{L^2} <\infty\}$$ for $ k\in\mathbb{Z}$ with norm
$|f|^2_{\mathcal{B}_{\epsilon}^k}=\langle(1+ |\Psi_{\epsilon}|)^k f, f\rangle$.
Let $\mathcal{B}_{\epsilon}^{-k}$ be the dual to $\mathcal{B}_{\epsilon}^k$ for each $k$. 

{\lem For some $\epsilon>0$, $S\in \mathcal{B}_{\epsilon}^{1}$.}

\begin{proof} Recall that $S$ is in $L^2(\Gamma\backslash\h)$ by design and in fact by examining the construction of each $S$ we see that $S$ is $\mathcal{O}(y^{\frac{1}{2}-\epsilon})$.  By design,  $|\Psi_{\epsilon}|\cdot S\cdot \overline{S}$ is $\mathcal{O}(y^{1-\epsilon})$ and 
 $\displaystyle\langle (1+ |\Psi_{\epsilon}|) S, S\rangle_{L^2} <\infty$ as desired.
\end{proof}
{\lem Given $\epsilon>0$  and $s$ with $\text{Re}(s)=1/2$, both $E_s$ and $\wedge^T\!E_s$ are in $\mathcal{B}_{\epsilon}^{-1}$.}

\begin{proof} Let $s$ be such that $\text{Re}(s)=1/2$.  In the cases of 
$\displaystyle\langle (1+|\Psi_{1+\epsilon}(z)|)^{-1}\cdot  E_{s},E_s\rangle_{L^2}$ and 
$\displaystyle\langle(1+|\Psi_{1+\epsilon}(z)|)^{-1}\cdot \wedge ^T E_{s}, \wedge ^T E_s\rangle_{L^2}$, both integrands are of order  $\mathcal{O}(y^{1-\epsilon})$ since $ E_{s}$ and $\wedge ^T E_{s}$ are $\mathcal{O}(y^{1/2})$.  When integrated against the measure $\frac{dy}{y^2}$, these integrals will converge.\\
\end{proof}

\tab Now we must show that the limit of the truncated Eisenstein series approaches the original Eisenstein series in this topology.\\

{\lem\label{keylem}  Given $\epsilon>0$  and $s$ with $\text{Re}(s)=1/2$, we have $\displaystyle \mathcal{B}_{\epsilon}^{-1}\!\!-\!\lim_T \wedge^T E_s= E_s$.}

\begin{proof} Consider $E_s$ where $\text{Re}(s)=1/2$ and $\sigma_o>\text{Re}(s)$. We have
$$|\wedge^TE_s|^2_{\mathcal{B}^{-1}_{\epsilon}}
=\left\langle \frac{1}{1+|\Psi_{1+\epsilon}|}\cdot \wedge ^T E_{s}, \wedge ^T E_s\right\rangle_{L^2}
=\int_{\Gamma\backslash \mathfrak{H}}\frac{1}{1+|\Psi_{1+\epsilon}|}\wedge ^T \!E_{s}\cdot \wedge ^T \!E_{\overline{s}}\,\frac{dy\,dx}{y^2} $$
$$=\int_{\Gamma\backslash \mathfrak{H}}\frac{1}{1+|\Psi_{1+\epsilon}|}  \left( E_{s}\, -\Theta^T_s(z)\right)\cdot  \left( E_{\overline{s}}\, -\Theta^T_{\overline{s}}(z)\right)\,\frac{dy\,dx}{y^2} $$
$$=\int_0^{\infty}\int_{\substack{|x|\leq 1/2 \\ x^2\geq 1-y^2}}\frac{1}{1+|\Psi_{1+\epsilon}|}  \left( E_{s}\, -\Theta^T_s(z)\right)\cdot  \left( E_{\overline{s}}\, -\Theta^T_{\overline{s}}(z)\right)\,\frac{dy\,dx}{y^2} $$
$$=\int_0^T\int_{\substack{|x|\leq 1/2 \\ x^2\geq 1-y^2}}  \frac{1}{1+|\Psi_{1+\epsilon}|}\cdot E_{s}\cdot E_{\overline{s}}\,\frac{dx\,dy}{y^2}\phantom{weeeeeeeeeeeeeeeeeeeee}$$
$$\phantom{weeeeeeeeeeeeeeeeee}+\int_T^{\infty}\int_{\substack{|x|\leq 1/2 \\ x^2\geq 1-y^2}}\frac{1}{1+|\Psi_{1+\epsilon}|}  \left( E_{s}\, -(y^s+c_sy^{1-s})\right)\cdot  \left( E_{\overline{s}}\,-(y^{\overline{s}}+c_{\overline{s}}y^{1-\overline{s}})\right)\,\frac{dy\,dx}{y^2} 
$$

Now 
$\displaystyle\frac{1}{1+|\Psi_{1+\epsilon}|}  \left( E_{s}\, -y^s-c_sy^{1-s}\right)\cdot(E_{\overline{s}}\,-y^{\overline{s}}-c_{\overline{s}}y^{1-\overline{s}})< \frac{1}{E_{\sigma_o}} E_sE_{\overline{s}}\in L^2$ thus by Lebesgue's Convergence Theorem as $T\to \infty$, the second integral disappears and this becomes 
$$\int_0^{\infty}\int_{\substack{|x|\leq 1/2 \\ x^2\geq 1-y^2}}  \frac{1}{1+|\Psi_{1+\epsilon}|}\cdot E_{s}\cdot E_{\overline{s}}\,\frac{dx\,dy}{y^2}
=\left\langle \frac{1}{1+|\Psi_{1+\epsilon}|}\cdot E_{s}, E_s\right\rangle_{L^2}
=  |E_s|^2_{\mathcal{B}^{-1}_{\epsilon}}$$\\

 \end{proof}

\subsection{Integrals of Truncated Eisenstein series}
It remains to compute $$\displaystyle\int_{\Gamma\backslash \h}\wedge^T \,\overline{E}_s\cdot  S\,\frac{dy\,dx}{y^2} $$ for each $S$. \\

\tab We have that each $S\in L^2$.  However, we will need something a bit stronger to actually compute $\langle S,\wedge^T E_s\rangle_{L^2}$.  For instance, we may know that $\displaystyle \int_{\Gamma\backslash \h}  \wedge^T \,\overline{E_s}\cdot  S\,\frac{dy\,dx}{y^2}<\infty$ but since each $S$ involved many terms that we would like to be able to separate and compute, we need that each integral exists term-wise.  We will need a few results to address each of the terms for each $S$.\\

 \tab  In order to compute term-wise, truncated pairings we will need the following three results.  Lemma \ref{EaEb} will give us the first for the part of $S$ which consists of $E_{\alpha}E_{\beta}$ paired against $\wedge^T \,\overline{E}_s$.  Lemma \ref{intid} and Theorem \ref{MS} will allow us to compute the integrals corresponding to the part of $S$ which consists of linear combinations of Eisenstein series.\\


{\lem\label{EaEb} For $\alpha,\beta\neq 1$, 

$$\displaystyle\int_{\Gamma\backslash \h}\wedge^T \,\overline{E}_s\cdot  E_{\alpha}E_{\beta}\,\frac{dy\,dx}{y^2}$$

$$=\displaystyle
\frac{1}{\overline{s}+\alpha+\beta-1}T^{{\overline{s}+\alpha+\beta-1}}
+\frac{c_{\alpha}}{\overline{s}-\alpha+\beta} T^{\overline{s}-\alpha+\beta}
+\frac{c_{\beta}}{\overline{s}+\alpha-\beta} T^{\overline{s}+\alpha-\beta}
+\frac{c_{\alpha}c_{\beta}}{\overline{s}-\alpha-\beta+1}T^{\overline{s}-\alpha-\beta+1} $$

$$\displaystyle+L(\overline{s},E_{\alpha}\times E_{\beta})\cdot \int_{ y\leq T}y^{{\overline{s}}}\cdot W_\alpha(y) W_\beta(y) \,\frac{dy}{y^2}$$

$$\displaystyle+\frac{c_{\overline{s}} }{-{\overline{s}}+\alpha+\beta}\,T^{-{\overline{s}}+\alpha+\beta}  
+\frac{c_{\alpha}c_{\overline{s}} }{1-{\overline{s}}-\alpha+\beta}\,T^{1-{\overline{s}}-\alpha+\beta}  
+\frac{c_{\beta}c_{\overline{s}}}{1-{\overline{s}}+\alpha-\beta}\, T^{1-{\overline{s}}+\alpha-\beta}  
+\frac{c_{\alpha}c_{\beta}c_{\overline{s}}}{2-{\overline{s}}-\alpha-\beta}\, T^{2-{\overline{s}}-\alpha-\beta}$$

$$\displaystyle-c_{\overline{s}}\sum_{n\neq 0}\varphi(n,\alpha)\varphi(n,\beta)n^{\overline{s}}\int_{ y\geq T} y^{-1-{\overline{s}}}\cdot W_\alpha(y) W_\beta(y)\,dy$$\\}

\begin{proof}

In the first term, we will compute $\displaystyle\int_{\Gamma\backslash \h}\wedge^T \,\overline{E}_s\cdot  E_{\alpha}E_{\beta}\,\frac{dy\,dx}{y^2}$:
 
 $$ \int_{\Gamma\backslash \h}\wedge^T \,\overline{E}_s\cdot  E_{\alpha}E_{\beta}\,\frac{dy\,dx}{y^2} 
 = \int_{\Gamma\backslash \h}\left(\sum_{\gamma\in P\backslash \Gamma } \text{Im}(\gamma z)^{\overline{s}}-\sum_{\gamma\in P\backslash \Gamma } \tau_{\overline{s}}(\gamma z)\right)\cdot  E_{\alpha}E_{\beta}\,\frac{dy\,dx}{y^2} $$
 $$ = \int_{\Gamma\backslash \h}\sum_{\gamma\in P\backslash \Gamma } \left(\text{Im}(\gamma z)^{\overline{s}}-\ \tau_{\overline{s}}(\gamma z)\right)\cdot  E_{\alpha}E_{\beta}\,\frac{dy\,dx}{y^2} 
 =\int_{P\backslash \h} \left( y^{\overline{s}}-\ \tau_{\overline{s}}( z)\right)\cdot  E_{\alpha}E_{\beta}\,\frac{dy\,dx}{y^2} $$ by unwinding

 $$=\int_{\substack{P\backslash \h \\ y\leq T}} \left( y^{\overline{s}}-\ \tau_{\overline{s}}( z)\right)\cdot  E_{\alpha}E_{\beta}\,\frac{dy\,dx}{y^2}+\int_{\substack{P\backslash \h \\ y> T}} \left( y^{\overline{s}}-\ \tau_{\overline{s}}( z)\right)\cdot  E_{\alpha}E_{\beta}\,\frac{dy\,dx}{y^2}  $$ 
 
 $$=\int_{\substack{P\backslash \h \\ y\leq T}}y^{\overline{s}}\cdot  E_{\alpha}E_{\beta}\,\frac{dy\,dx}{y^2}-\int_{\substack{P\backslash \h \\ y> T}} c_{\overline{s}} y^{1-{\overline{s}}}\cdot  E_{\alpha}E_{\beta}\,\frac{dy\,dx}{y^2}  $$   \\ 

(A)  Examining $\displaystyle\int_{\substack{P\backslash \h \\ y\leq T}}y^{\overline{s}}\cdot  E_{\alpha}E_{\beta}\,\frac{dy\,dx}{y^2}$: 
 
 \tab Recall that the fundamental domain of $P\backslash \mathfrak{H}$ is $\{z=x+iy\in\mathfrak{H}~|~ 0\leq x\leq 1\}$ so we have
 $$\int_{\substack{P\backslash \h \\ y\leq T}}y^{\overline{s}}\cdot  E_{\alpha}E_{\beta}\,\frac{dy\,dx}{y^2}
 =\int_0^1\int_{ y\leq T}y^{{\overline{s}}-2}\cdot  E_{\alpha}E_{\beta}\,dy\,dx$$
 $$ =\int_0^1\int_{ y\leq T}y^{{\overline{s}}-2}
 [y^{\alpha} +c_{\alpha} y^{1-\alpha} +\sum_{n\neq 0}\varphi(n,\alpha)W_\alpha(|n|y) e^{2\pi i n x} ][y^{\beta} +c_{\beta} y^{1-\beta} +\sum_{m\neq 0}\varphi(m,\beta)W_\beta(|m|y)e^{2\pi i m x}  ]\,dy\,dx$$
  $$ =\int_{ y\leq T}y^{{\overline{s}}-2}(y^{\alpha} +c_{\alpha} y^{1-\alpha})(y^{\beta} +c_{\beta} y^{1-\beta}) +y^{{\overline{s}}-2}\sum_{n\neq 0}\varphi(n,\alpha)\varphi(n,\beta)\cdot W_\alpha(|n|y) W_\beta(|n|y)\,dy$$
 \hfill since the product vanishes off the diagonal.\\

(1)  Examining the first term   $\displaystyle\int_{ y\leq T}y^{{\overline{s}}-2}(y^{\alpha} +c_{\alpha} y^{1-\alpha})(y^{\beta} +c_{\beta} y^{1-\beta}) \,dy$:

 $$\int_{ y\leq T}y^{{\overline{s}}-2}(y^{\alpha} +c_{\alpha} y^{1-\alpha})(y^{\beta} +c_{\beta} y^{1-\beta})\,dy
  =\int_{ y\leq T}y^{{\overline{s}}-2}(y^{\alpha+\beta} +c_{\alpha} y^{1-\alpha+\beta}+c_{\beta} y^{1+\alpha-\beta}+c_{\alpha}c_{\beta}y^{2-\alpha-\beta}) \,dy
$$
$$ =\int_{ y\leq T}
y^{\overline{s}+\alpha+\beta-2}
+c_{\alpha} y^{\overline{s}-\alpha+\beta-1}
+c_{\beta} y^{\overline{s}+\alpha-\beta-1}
+c_{\alpha}c_{\beta}y^{\overline{s}-\alpha-\beta} \,dy$$
$$ =
\frac{1}{\overline{s}+\alpha+\beta-1}y^{{\overline{s}+\alpha+\beta-1}}
+\frac{c_{\alpha}}{\overline{s}-\alpha+\beta} y^{\overline{s}-\alpha+\beta}
+\frac{c_{\beta}}{\overline{s}+\alpha-\beta} y^{\overline{s}+\alpha-\beta}
+\frac{c_{\alpha}c_{\beta}}{\overline{s}-\alpha-\beta+1}y^{\overline{s}-\alpha-\beta+1}\Big|_{y=0}^T 
$$
$$ =
\frac{1}{\overline{s}+\alpha+\beta-1}T^{{\overline{s}+\alpha+\beta-1}}
+\frac{c_{\alpha}}{\overline{s}-\alpha+\beta} T^{\overline{s}-\alpha+\beta}
+\frac{c_{\beta}}{\overline{s}+\alpha-\beta} T^{\overline{s}+\alpha-\beta}
+\frac{c_{\alpha}c_{\beta}}{\overline{s}-\alpha-\beta+1}T^{\overline{s}-\alpha-\beta+1} $$
$$-
\lim_{t\to 0^+}
\left(\frac{1}{\overline{s}+\alpha+\beta-1}t^{{\overline{s}+\alpha+\beta-1}}
+\frac{c_{\alpha}}{\overline{s}-\alpha+\beta} t^{\overline{s}-\alpha+\beta}
+\frac{c_{\beta}}{\overline{s}+\alpha-\beta} t^{\overline{s}+\alpha-\beta}
+\frac{c_{\alpha}c_{\beta}}{\overline{s}-\alpha-\beta+1}t^{\overline{s}-\alpha-\beta+1} \right)
$$\\

For $\text{Re}(\overline{s})>\text{Re}(\alpha+\beta)>1$ where $\text{Re}(\alpha)>0$ and $\text{Re}(\beta)>0$, this last term $$
\lim_{t\to 0^+}
\left(\frac{1}{\overline{s}+\alpha+\beta-1}t^{{\overline{s}+\alpha+\beta-1}}
+\frac{c_{\alpha}}{\overline{s}-\alpha+\beta} t^{\overline{s}-\alpha+\beta}
+\frac{c_{\beta}}{\overline{s}+\alpha-\beta} t^{\overline{s}+\alpha-\beta}
+\frac{c_{\alpha}c_{\beta}}{\overline{s}-\alpha-\beta+1}t^{\overline{s}-\alpha-\beta+1} \right)$$
$$=0$$

Thus, by the Identity Principle, we can meromorphically continue to get that 
$$\displaystyle\int_{ y\leq T}y^{{\overline{s}}-2}(y^{\alpha} +c_{\alpha} y^{1-\alpha})(y^{\beta} +c_{\beta} y^{1-\beta}) \,dy$$
$$ =
\frac{1}{\overline{s}+\alpha+\beta-1}T^{{\overline{s}+\alpha+\beta-1}}
+\frac{c_{\alpha}}{\overline{s}-\alpha+\beta} T^{\overline{s}-\alpha+\beta}
+\frac{c_{\beta}}{\overline{s}+\alpha-\beta} T^{\overline{s}+\alpha-\beta}
+\frac{c_{\alpha}c_{\beta}}{\overline{s}-\alpha-\beta+1}T^{\overline{s}-\alpha-\beta+1} $$\\

(2)  Examining the second term   $\displaystyle\int_{ y\leq T}y^{{\overline{s}}-2}\sum_{n\neq 0}\varphi(n,\alpha)\varphi(n,\beta)\cdot W_\alpha(|n|y) W_\beta(|n|y) \,dy$:
 $$\int_{ y\leq T}y^{{\overline{s}}-2}\sum_{n\neq 0}\varphi(n,\alpha)\varphi(n,\beta)\cdot W_\alpha(|n|y) W_\beta(|n|y) \,dy
 =\sum_{n\neq 0}\varphi(n,\alpha)\varphi(n,\beta)\int_{ y\leq T}y^{{\overline{s}}}\cdot W_\alpha(|n|y) W_\beta(|n|y) \,\frac{dy}{y^2}$$
replacing $y$ by $y/n$ we have 
$$ =\sum_{n\neq 0}\frac{\varphi(n,\alpha)\varphi(n,\beta)}{n^{\overline{s}-1}}\int_{ y\leq T}y^{{\overline{s}}}\cdot W_\alpha(y) W_\beta(y) \,\frac{dy}{y^2} 
=L(\overline{s},E_{\alpha}\times E_{\beta})\cdot \int_{ y\leq T}y^{{\overline{s}}}\cdot W_\alpha(y) W_\beta(y) \,\frac{dy}{y^2}$$

 \vspace{.5cm}

(B)  Examining $\displaystyle{\int_{\substack{P\backslash \h \\ y> T}} c_{\overline{s}} y^{1-{\overline{s}}}\cdot  E_{\alpha}E_{\beta}\,\frac{dy\,dx}{y^2}  }$: 

Recall that the fundamental domain of $P\backslash \mathfrak{H}$ is $\{z=x+iy\in\mathfrak{H}~|~ 0\leq x\leq 1\}$ so we have
$$\int_{\substack{P\backslash \h \\ y> T}} c_{\overline{s}} y^{1-{\overline{s}}}\cdot  E_{\alpha}E_{\beta}\,\frac{dy\,dx}{y^2}  =\int_0^1\int_{y\geq T} c_{\overline{s}} y^{1-{\overline{s}}}\cdot  E_{\alpha}E_{\beta}\,\frac{dy}{y^2}\, dx  $$

$$=  \int_0^1\int_{y\geq T} c_{\overline{s}} y^{1-{\overline{s}}}\left(y^{\alpha} +c_{\alpha} y^{1-\alpha} +\sum_{n\neq 0}\varphi(n,\alpha) W_\alpha(|n|y) e^{2\pi i nx}\right)\phantom{weeeeeeeeeeeeeeee}$$
$$\phantom{weeeeeeeeeeeeeeeeeeee}\cdot \left(y^{\beta} +c_{\beta} y^{1-\beta} +\sum_{m\neq 0}\varphi(m,\beta)W_\beta(|m|y)e^{2\pi i mx} \right)\,\frac{dy\, dx}{y^2} $$
  $$ =\int_{ y\geq T}c_{\overline{s}} y^{-1-{\overline{s}}}(y^{\alpha} +c_{\alpha} y^{1-\alpha})(y^{\beta} +c_{\beta} y^{1-\beta}) +c_{\overline{s}} y^{-1-{\overline{s}}}\sum_{n\neq 0}\varphi(n,\alpha)\varphi(n,\beta)\cdot W_\alpha(|n|y) W_\beta(|n|y)\,dy$$
  since the product vanishes off the diagonal.\\
  
(1)  Examining the first term $\displaystyle\int_{ y\geq T}c_{\overline{s}} y^{-1-{\overline{s}}}(y^{\alpha} +c_{\alpha} y^{1-\alpha})(y^{\beta} +c_{\beta} y^{1-\beta})\,dy$:

$$\int_{ y\geq T}c_{\overline{s}} y^{-1-{\overline{s}}}(y^{\alpha} +c_{\alpha} y^{1-\alpha})(y^{\beta} +c_{\beta} y^{1-\beta})\,dy$$
$$=\int_{ y\geq T}c_{\overline{s}} \,y^{-1-{\overline{s}}+\alpha+\beta}  
+c_{\alpha}c_{\overline{s}} \,y^{-{\overline{s}}-\alpha+\beta}  
+c_{\beta}c_{\overline{s}}\, y^{-{\overline{s}}+\alpha-\beta}  
+c_{\alpha}c_{\beta}c_{\overline{s}}\, y^{1-{\overline{s}}-\alpha-\beta}  
\,dy$$
$$=\Big(
\frac{c_{\overline{s}} }{-{\overline{s}}+\alpha+\beta}\,y^{-{\overline{s}}+\alpha+\beta}  
+\frac{c_{\alpha}c_{\overline{s}} }{1-{\overline{s}}-\alpha+\beta}\,y^{1-{\overline{s}}-\alpha+\beta}  
+\frac{c_{\beta}c_{\overline{s}}}{1-{\overline{s}}+\alpha-\beta}\, y^{1-{\overline{s}}+\alpha-\beta} $$ 
$$\phantom{weeeeeeeeeeeeeeeeeeeeeeeeeeeeeeeeeeeeeeeeeeeeeeeeeeeeeeeeeee}+\frac{c_{\alpha}c_{\beta}c_{\overline{s}}}{2-{\overline{s}}-\alpha-\beta}\, y^{2-{\overline{s}}-\alpha-\beta}\Big)  \,
\Big|_{y=T}^{\infty}$$
$$= \lim_{t\to\infty}\Big[\big(
\frac{c_{\overline{s}} }{-{\overline{s}}+\alpha+\beta}\,t^{-{\overline{s}}+\alpha+\beta}  
+\frac{c_{\alpha}c_{\overline{s}} }{1-{\overline{s}}-\alpha+\beta}\,t^{1-{\overline{s}}-\alpha+\beta}  
+\frac{c_{\beta}c_{\overline{s}}}{1-{\overline{s}}+\alpha-\beta}\, t^{1-{\overline{s}}+\alpha-\beta}  $$ 
$$\phantom{weeeeeeeeeeeeeeeeeeeeeeeeeeeeeeeeeeeeeeeeeeeeeeeeeeeeeeeeeeeeee}
+\frac{c_{\alpha}c_{\beta}c_{\overline{s}}}{2-{\overline{s}}-\alpha-\beta}\, t^{2-{\overline{s}}-\alpha-\beta}  \Big)$$
$$\phantom{weeeeeeeeee}-\Big(\frac{c_{\overline{s}} }{-{\overline{s}}+\alpha+\beta}\,T^{-{\overline{s}}+\alpha+\beta}  
+\frac{c_{\alpha}c_{\overline{s}} }{1-{\overline{s}}-\alpha+\beta}\,T^{1-{\overline{s}}-\alpha+\beta}  
+\frac{c_{\beta}c_{\overline{s}}}{1-{\overline{s}}+\alpha-\beta}\, T^{1-{\overline{s}}+\alpha-\beta}  $$
$$\phantom{weeeeeeeeeeeeeeeeeeeeeeeeeeeeeeeeeeeeeeeeeeeeeeeeeeeeeeeeeeeeee}+\frac{c_{\alpha}c_{\beta}c_{\overline{s}}}{2-{\overline{s}}-\alpha-\beta}\, T^{2-{\overline{s}}-\alpha-\beta}  \big)\Big]
$$

For $\text{Re}(\overline{s})>\text{Re}(\alpha+\beta)>1$ where $\text{Re}(\alpha)>1/2$ and $\text{Re}(\beta)>1/2$ the first term
$$\lim_{t\to\infty}\big(
\frac{c_{\overline{s}} }{-{\overline{s}}+\alpha+\beta}\,t^{-{\overline{s}}+\alpha+\beta}  
+\frac{c_{\alpha}c_{\overline{s}} }{1-{\overline{s}}-\alpha+\beta}\,t^{1-{\overline{s}}-\alpha+\beta}  
+\frac{c_{\beta}c_{\overline{s}}}{1-{\overline{s}}+\alpha-\beta}\, t^{1-{\overline{s}}+\alpha-\beta}  
$$
$$\phantom{weeeeeeeeeeeeeeeeeeeeeeeeeeeeeeeeeeeeeeeeeeeeeeeeeee}+\frac{c_{\alpha}c_{\beta}c_{\overline{s}}}{2-{\overline{s}}-\alpha-\beta}\, t^{2-{\overline{s}}-\alpha-\beta}  \big)=0$$

Thus, by the Identity Principle, we can meromorphically continue to get that 
$$\displaystyle\int_{ y\geq T}c_{\overline{s}} y^{-1-{\overline{s}}}(y^{\alpha} +c_{\alpha} y^{1-\alpha})(y^{\beta} +c_{\beta} y^{1-\beta})\,dy$$

$$=-
\Big(\frac{c_{\overline{s}} }{-{\overline{s}}+\alpha+\beta}\,T^{-{\overline{s}}+\alpha+\beta}  
+\frac{c_{\alpha}c_{\overline{s}} }{1-{\overline{s}}-\alpha+\beta}\,T^{1-{\overline{s}}-\alpha+\beta}  
+\frac{c_{\beta}c_{\overline{s}}}{1-{\overline{s}}+\alpha-\beta}\, T^{1-{\overline{s}}+\alpha-\beta}  $$
$$\phantom{weeeeeeeeeeeeeeeeeeeeeeeeeeeeeeeeeeeeeeeeeeeeeeeeeeeeeeeeeeeeee}+\frac{c_{\alpha}c_{\beta}c_{\overline{s}}}{2-{\overline{s}}-\alpha-\beta}\, T^{2-{\overline{s}}-\alpha-\beta}  \Big)
$$\\

(2)  Examining the second term $\displaystyle\int_{ y\geq T}c_{\overline{s}} y^{-1-{\overline{s}}}\sum_{n\neq 0}\varphi(n,\alpha)\varphi(n,\beta)\cdot W_\alpha(|n|y) W_\beta(|n|y)\,dy$:

$$\displaystyle\int_{ y\geq T}c_{\overline{s}} y^{-1-{\overline{s}}}\sum_{n\neq 0}\varphi(n,\alpha)\varphi(n,\beta)\cdot W_\alpha(|n|y) W_\beta(|n|y)\,dy$$

$$=\int_{ y\geq T}c_{\overline{s}} y^{-1-{\overline{s}}}\sum_{n\neq 0}\varphi(n,\alpha)\varphi(n,\beta)\cdot W_\alpha(|n|y) W_\beta(|n|y)\,dy$$

$$=\sum_{n\neq 0}\varphi(n,\alpha)\varphi(n,\beta)\int_{ y\geq T}c_{\overline{s}} y^{-1-{\overline{s}}}\cdot W_\alpha(|n|y) W_\beta(|n|y)\,dy$$

replacing $y$ by $y/n$ we have 

$$=c_{\overline{s}}\sum_{n\neq 0}\varphi(n,\alpha)\varphi(n,\beta)n^{\overline{s}}\int_{ y\geq T} y^{-1-{\overline{s}}}\cdot W_\alpha(y) W_\beta(y)\,dy$$\\

Putting (A) and (B) together, we get: 
$$\displaystyle\int_{\Gamma\backslash \h}\wedge^T \,\overline{E}_s\cdot  E_{\alpha}E_{\beta}\,\frac{dy\,dx}{y^2}$$
$$=
\frac{1}{\overline{s}+\alpha+\beta-1}T^{{\overline{s}+\alpha+\beta-1}}
+\frac{c_{\alpha}}{\overline{s}-\alpha+\beta} T^{\overline{s}-\alpha+\beta}
+\frac{c_{\beta}}{\overline{s}+\alpha-\beta} T^{\overline{s}+\alpha-\beta}
+\frac{c_{\alpha}c_{\beta}}{\overline{s}-\alpha-\beta+1}T^{\overline{s}-\alpha-\beta+1} $$

$$+L(\overline{s},E_{\alpha}\times E_{\beta})\cdot \int_{ y\leq T}y^{{\overline{s}}}\cdot W_\alpha(y) W_\beta(y) \,\frac{dy}{y^2}$$

$$+\frac{c_{\overline{s}} }{-{\overline{s}}+\alpha+\beta}\,T^{-{\overline{s}}+\alpha+\beta}  
+\frac{c_{\alpha}c_{\overline{s}} }{1-{\overline{s}}-\alpha+\beta}\,T^{1-{\overline{s}}-\alpha+\beta}  
+\frac{c_{\beta}c_{\overline{s}}}{1-{\overline{s}}+\alpha-\beta}\, T^{1-{\overline{s}}+\alpha-\beta}  
+\frac{c_{\alpha}c_{\beta}c_{\overline{s}}}{2-{\overline{s}}-\alpha-\beta}\, T^{2-{\overline{s}}-\alpha-\beta}  $$

$$-c_{\overline{s}}\sum_{n\neq 0}\varphi(n,\alpha)\varphi(n,\beta)n^{\overline{s}}\int_{ y\geq T} y^{-1-{\overline{s}}}\cdot W_\alpha(y) W_\beta(y)\,dy$$ \end{proof}


 \tab  We will also need the following two results for the parts of $S$ which consist of linear combinations of Eisenstein series.

{\lem\label{intid} $\displaystyle \int_{\Gamma\backslash \h}  \wedge^T \overline{E}_s  \cdot \wedge^T E_r\,\frac{dy\,dx}{y^2}= \int_{\Gamma\backslash \h}  \wedge^T \overline{E}_s\cdot  E_r\,\frac{dy\,dx}{y^2}$}
\begin{proof}  Recall that the fundamental domain for  $\Gamma\backslash \h$ is  $F=\{z\in\h~|~|z|\geq 1~\&~|\text{Re}(z)|\leq 1/2\}$ so rewriting our integral we have 
  $$\displaystyle \int_{\Gamma\backslash \h}  \wedge^T \,\overline{E}_s  \cdot \wedge^T E_r\,\frac{dx\,dy}{y^2}
  = \int_{0\leq y\leq\infty}\int_{\substack{|x|\leq 1/2 \\ x^2\geq 1-y^2}}\wedge^T \,\overline{E}_s  \cdot \wedge^T E_r\,\frac{dx\,dy}{y^2}$$
  $$=\int_{0\leq y\leq T}\int_{\substack{|x|\leq 1/2 \\ x^2\geq 1-y^2}}\wedge^T \,\overline{E}_s  \cdot \wedge^T E_r\,\frac{dx\,dy}{y^2} + \int_{T\leq y\leq\infty}\int_{\substack{|x|\leq 1/2 \\ x^2\geq 1-y^2}}\wedge^T \,\overline{E}_s  \cdot \wedge^T E_r\,\frac{dx\,dy}{y^2}$$ 
\tab  Notice that since the first integral is only defined for $y\leq T$ and on this region, $ \wedge^T E_r=  E_r$ by definition, 
  $$\int_{0\leq y\leq T}\int_{\substack{|x|\leq 1/2 \\ x^2\geq 1-y^2}}\wedge^T \,\overline{E}_s  \cdot \wedge^T E_r\,\frac{dx\,dy}{y^2} 
  =\int_{0\leq y\leq T}\int_{\substack{|x|\leq 1/2 \\ x^2\geq 1-y^2}}\wedge^T \,\overline{E}_s  \cdot  E_r\,\frac{dx\,dy}{y^2} $$
  Thus it remains to show this result for the second integral 
  $\displaystyle\int_{T\leq y\leq\infty}\int_{\substack{|x|\leq 1/2 \\ x^2\geq 1-y^2}}\wedge^T \,\overline{E}_s  \cdot \wedge^T E_r\,\frac{dx\,dy}{y^2}$. 
  
\tab   For $T>1$, this domain of integration is a cylinder so that 
   $$\displaystyle\int_{T\leq y\leq\infty}\int_{\substack{|x|\leq 1/2 \\ x^2\geq 1-y^2}}\wedge^T \,\overline{E}_s  \cdot \wedge^T E_r\,\frac{dx\,dy}{y^2}
   = \int_{T\leq y\leq\infty}\int_{|x|\leq 1/2 }\wedge^T \,\overline{E}_s  \cdot \wedge^T E_r\,\frac{dx\,dy}{y^2}$$  
   Writing this in terms of Fourier expansions, we have 
$$= \int_{T\leq y\leq\infty}\int_{|x|\leq 1/2 }\left(\sum_{n\neq 0}\varphi(n, s)\overline{W}_s(|n|y)  e^{-2\pi i n x} \right) \cdot \left(\sum_{m\neq 0}\varphi(m, s)W_s(|m|y)  e^{2\pi i m x} \right)\,\frac{dx\,dy}{y^2}$$  
$$= \int_{T\leq y\leq\infty}\int_{|x|\leq 1/2 }\left(\sum_{n\neq 0}\varphi(n, s)\overline{W}_s(|n|y) e^{-2\pi i n x}  \right) \cdot \left(\sum_{m\in \mathbb{Z}}\varphi(m, s)W_s(|m|y)  e^{2\pi i m x} \right)\,\frac{dx\,dy}{y^2}$$ since the integral will be zero when $n\neq m$ i.e. when $m=0$ (this computation was seen previously as $\int_0^1e^{2\pi i (-m)x}dx=\delta_{0,m}$ and the 0$^{th}$ coefficient of the first Eisenstein series has been truncated to be made 0)
   $$= \int_{T\leq y\leq\infty}\int_{|x|\leq 1/2 }\wedge^T \,\overline{E}_s  \cdot  E_r\,\frac{dx\,dy}{y^2}$$  as desired. Combining the domains as originally stated, we have 
   $$\displaystyle \int_{\Gamma\backslash \h}  \wedge^T\, \overline{E}_s  \cdot \wedge^T E_r\,\frac{dy\,dx}{y^2}= \int_{\Gamma\backslash \h}  \wedge^T\, \overline{E}_s\cdot  E_r\,\frac{dy\,dx}{y^2}$$
   
 \end{proof}

\tab We will use this to compute the pairing for the linear combination terms in $S$ with the truncated Eisenstein series. Lemma \ref{intid} allows for each of the terms in the linear combination to become
 $$\int_{\Gamma\backslash \h}\wedge^T \,\overline{E}_s\cdot  E_{r}\,\frac{dy\,dx}{y^2} 
 =\int_{\Gamma\backslash \h}\wedge^T \,\overline{E}_s\cdot  \wedge^T E_{r}\,\frac{dy\,dx}{y^2} $$ 
and then we will use Maass-Selberg and unwinding of $\wedge^T \,\overline{E}_s$.\\
 
\tab Recall the following the Maass-Selberg relation (see Casselman \cite{Casselman1993} or Garrett \cite{Garrett2016} for proof) states that\\

{\thm\label{MS} For two complex numbers $r, s\neq 1$ with $r(r-1)\neq s(s-1)$, 

 $$\int_{\Gamma\backslash \h}\wedge^T \,E_s\cdot  \wedge^T E_{r}\,\frac{dy\,dx}{y^2} $$ 
 $$=\frac{T^{r+s-1}}{r+s-1}+c_r\frac{T^{(1-r)+s-1}}{(1-r)+s-1}+c_s\frac{T^{r+(1-s)-1}}{r+(1-s)-1}+c_r c_s\frac{T^{(1-r)+(1-s)-1}}{(1-r)+(1-s)-1}.$$}\\
 
\tab Observe that when we are computing $\displaystyle\int_{\Gamma\backslash \h}\wedge^T \,\overline{E}_s\cdot  S\,\frac{dy\,dx}{y^2}$, the last few terms of $S$ will appear as $\int_{\Gamma\backslash \h}\wedge^T \,\overline{E}_s\cdot  E_{r}\,\frac{dy\,dx}{y^2} $.  Using the previous two results, for each $r$ in our linear combination $S$, we will have something of the form 
 $$\int_{\Gamma\backslash \h}\wedge^T \,\overline{E}_s\cdot  E_{r}\,\frac{dy\,dx}{y^2} $$
 $$=\frac{T^{r+\overline{s}-1}}{r+\overline{s}-1}
 +c_r\cdot \frac{T^{(1-r)+\overline{s}-1}}{(1-r)+\overline{s}-1}
 +c_{\overline{s}}\cdot\frac{T^{r+(1-\overline{s})-1}}{r+(1-\overline{s})-1}
 +c_r c_{\overline{s}}\cdot\frac{T^{(1-r)+(1-\overline{s})-1}}{(1-r)+(1-\overline{s})-1}$$

The following is a version of the Maass-Selberg relation for when $r=1$.  We follow the style of argument for the original Maass-Selberg relation, thus we will label it as a corollary.  \\

{\cor\label{MSCor} For all complex  $s$ with $0 \neq s(s-1)$, 
$$\int_{\Gamma\backslash \h}\wedge^T \,\overline{E}_s\cdot   E_{1}^*\,\frac{dy\,dx}{y^2} $$
$$=\frac{ T^{{\overline{s}}}}{{\overline{s}}}+C\frac{T^{{\overline{s}}-1}}{{\overline{s}}-1}-\frac{3}{\pi}\frac{T^{{\overline{s}}-1}}{{\overline{s}}-1}\log T +\frac{3}{\pi}\frac{T^{{\overline{s}}-1}}{({\overline{s}}-1)^2}
+c_{\overline{s}}\left(\frac{T^{1-{\overline{s}}}}{1-{\overline{s}}}-C\frac{T^{-{\overline{s}}}}{{\overline{s}}} +\frac{3}{\pi}\frac{T^{{-\overline{s}}} }{{\overline{s}}}\log T +\frac{3}{\pi}\frac{T^{-{\overline{s}}} }{{\overline{s}}^2}\right)
$$ 
}

\begin{proof}
$$ \int_{\Gamma\backslash \h}\wedge^T \,\overline{E}_s\cdot E_{1}^*\,\frac{dy\,dx}{y^2} 
 = \int_{\Gamma\backslash \h}\left(\sum_{\gamma\in P\backslash \Gamma } \text{Im}(\gamma z)^{\overline{s}}-\sum_{\gamma\in P\backslash \Gamma } \tau_{\overline{s}}(\gamma z)\right)\cdot  E_{1}^*\,\frac{dy\,dx}{y^2} $$
 $$ = \int_{\Gamma\backslash \h}\sum_{\gamma\in P\backslash \Gamma } \left(\text{Im}(\gamma z)^{\overline{s}}-\ \tau_{\overline{s}}(\gamma z)\right)\cdot  E_1^*\,\frac{dy\,dx}{y^2} 
 =\int_{P\backslash \h} \left( y^{\overline{s}}-\ \tau_{\overline{s}}( z)\right)\cdot E_{1}^*\,\frac{dy\,dx}{y^2} $$
 \hfill  by unwinding
 $$=\int_{\substack{P\backslash \h \\ y\leq T}}  y^{\overline{s}}\cdot  E_{1}^*\,\frac{dy\,dx}{y^2}-\int_{\substack{P\backslash \h \\ y> T}}c_{\overline{s}} y^{1-{\overline{s}}}\cdot  E_{1}^*\,\frac{dy\,dx}{y^2}  $$ 
  \hfill  from the definitions of $\tau_s$\\
 
(A)  Examining $\displaystyle\int_{\substack{P\backslash \h \\ y\leq T}}y^{\overline{s}}\cdot  E_{1}^*\,\frac{dy\,dx}{y^2}$: 
 
 \tab Recall that the fundamental domain of $P\backslash \mathfrak{H}$ is $\{z=x+iy\in\mathfrak{H}~|~ 0\leq x\leq 1\}$ so we have
 $$\int_{\substack{P\backslash \h \\ y\leq T}}y^{\overline{s}}\cdot  E_{1}^*\,\frac{dy\,dx}{y^2}
 =\int_0^1\int_{ y\leq T}y^{{\overline{s}}-2}\cdot  E_{1}^*\,dy\,dx
  =\int_0^1\int_{ y\leq T}y^{{\overline{s}}-2}
 \left(c_PE_{1}^* +\sum_{n\neq 0}\varphi(n,1)W_1(|n|y) e^{2\pi i n x} \right)\,dy\,dx$$
 $$ =\int_0^1\int_{ y\leq T}
 y^{{\overline{s}}-2}\cdot c_PE_{1}^* +y^{{\overline{s}}-2}\cdot \sum_{n\neq 0}\varphi(n,1)W_1(|n|y) e^{2\pi i n x} \,dy\,dx$$\\
 
\tab   (1)  Examining the first term $\displaystyle \int_0^1\int_{ y\leq T} y^{{\overline{s}}-2}\cdot c_PE_{1}^* \,dy\,dx$:

$$\int_{ y\leq T} y^{{\overline{s}}-2}\cdot c_PE_{1}^* \,dy
 = \int_{ y\leq T} y^{{\overline{s}}-2}\cdot  \left( y+C-\frac{3}{\pi}\log y \right)\,dy
  = \int_{ y\leq T}  y^{{\overline{s}}-1}+Cy^{{\overline{s}}-2}-\frac{3}{\pi}y^{{\overline{s}}-2}\log y \,dy$$
 $$ = \frac{ y^{{\overline{s}}}}{{\overline{s}}}+C\frac{y^{{\overline{s}}-1}}{{\overline{s}}-1}-\frac{3}{\pi}\frac{y^{{\overline{s}}-1}}{{\overline{s}}-1}\log y +\frac{3}{\pi}\frac{y^{{\overline{s}}-1}}{({\overline{s}}-1)^2}\Big|_0^T$$
$$= \frac{ T^{{\overline{s}}}}{{\overline{s}}}+C\frac{T^{{\overline{s}}-1}}{{\overline{s}}-1}-\frac{3}{\pi}\frac{T^{{\overline{s}}-1}}{{\overline{s}}-1}\log T +\frac{3}{\pi}\frac{T^{{\overline{s}}-1}}{({\overline{s}}-1)^2}
-\lim_{t\to 0^+}\left( \frac{ t^{{\overline{s}}}}{{\overline{s}}}+C\frac{t^{{\overline{s}}-1}}{{\overline{s}}-1}-\frac{3}{\pi}\frac{t^{{\overline{s}}-1}}{{\overline{s}}-1}\log t +\frac{3}{\pi}\frac{t^{{\overline{s}}-1}}{({\overline{s}}-1)^2}\right)$$

For $\text{Re}(\overline{s})>1$, the second term 
$$\lim_{t\to 0^+}\left( \frac{ t^{{\overline{s}}}}{{\overline{s}}}+C\frac{t^{{\overline{s}}-1}}{{\overline{s}}-1}-\frac{3}{\pi}\frac{t^{{\overline{s}}-1}}{{\overline{s}}-1}\log t +\frac{3}{\pi}\frac{t^{{\overline{s}}-1}}{({\overline{s}}-1)^2}\right)=0$$

Thus, by the Identity Principle, we can meromorphically continue to get that
$$\displaystyle \int_0^1\int_{ y\leq T} y^{{\overline{s}}-2}\cdot c_PE_1^* \,dy\,dx= \frac{ T^{{\overline{s}}}}{{\overline{s}}}+C\frac{T^{{\overline{s}}-1}}{{\overline{s}}-1}-\frac{3}{\pi}\frac{T^{{\overline{s}}-1}}{{\overline{s}}-1}\log T +\frac{3}{\pi}\frac{T^{{\overline{s}}-1}}{({\overline{s}}-1)^2}$$

\tab   (2)  Examining the second term   $\displaystyle\int_0^1\int_{ y\leq T}y^{{\overline{s}}-2}\cdot \sum_{n\neq 0}\varphi(n,1)W_1(|n|y) e^{2\pi i n x} \,dy\,dx$:
 $$\int_0^1\int_{ y\leq T}y^{{\overline{s}}-2}\cdot \sum_{n\neq 0}\varphi(n,1)W_1(|n|y) e^{2\pi i n x} \,dy\,dx=
 \int_{ y\leq T}y^{{\overline{s}}-2}\cdot \sum_{n\neq 0}\varphi(n,1)W_1(|n|y) \,dy\cdot \int_0^1e^{2\pi i n x}\,dx$$
 $$=\int_{ y\leq T}y^{{\overline{s}}-2}\cdot \sum_{n\neq 0}\varphi(n,1)W_1(|n|y) \,dy\cdot\delta_{0,n}= 0 $$

 \vspace{.5cm}

(B)  Examining $\displaystyle{\int_{\substack{P\backslash \h \\ y> T}} c_{\overline{s}} y^{1-{\overline{s}}}\cdot  E_{1}^*\,\frac{dy\,dx}{y^2}  }$: 

Recall that the fundamental domain of $P\backslash \mathfrak{H}$ is $\{z=x+iy\in\mathfrak{H}~|~ 0\leq x\leq 1\}$ so we have
$$\int_{\substack{P\backslash \h \\ y> T}} c_{\overline{s}} y^{1-{\overline{s}}}\cdot  E_{1}^*\,\frac{dy\,dx}{y^2}  =\int_0^1\int_{y\geq T} c_{\overline{s}} y^{1-{\overline{s}}}\cdot  E_{1}^*\,\frac{dy}{y^2}\, dx  $$

$$=  \int_0^1\int_{y\geq T} c_{\overline{s}} y^{1-{\overline{s}}}\left(c_PE_{1}^* +\sum_{n\neq 0}\varphi(n,1) W_1(|n|y) e^{2\pi i n x}\right)\,\frac{dy\, dx}{y^2} $$
  $$ =\int_0^1\int_{ y\geq T}c_{\overline{s}} y^{-1-{\overline{s}}}\cdot c_PE_{1}^* + c_{\overline{s}}y^{-1-{\overline{s}}}\cdot\sum_{n\neq 0}\varphi(n,1) W_1(|n|y)e^{2\pi i n x}\,dy\,dx$$
  since the product vanishes off the diagonal.\\

    \tab (1) Examining the first term $\displaystyle\int_0^1\int_{ y\geq T}c_{\overline{s}} y^{-1-{\overline{s}}}\cdot c_PE_{1}^*\,dy \, dx$:

$$\displaystyle\int_{ y\geq T}c_{\overline{s}} y^{-1-{\overline{s}}}\cdot c_PE_{1}^* \,dy  
= c_{\overline{s}}\int_{ y\geq T} y^{-1-{\overline{s}}}\cdot \left( y+C-\frac{3}{\pi}\log y \right) \,dy 
= c_{\overline{s}}\int_{ y\geq T} y^{-{\overline{s}}} +Cy^{-1-{\overline{s}}} -\frac{3}{\pi}y^{-1-{\overline{s}}} \log y \,dy  $$
$$= c_{\overline{s}}\left(\frac{y^{-{\overline{s}}+1}}{-{\overline{s}}+1}+C\frac{y^{-{\overline{s}}}}{-{\overline{s}}} -\frac{3}{\pi}\frac{y^{-{\overline{s}}} }{-{\overline{s}}}\log y +\frac{3}{\pi}\frac{y^{-{\overline{s}}} }{{\overline{s}}^2}\right)\Big|_T^{\infty}  $$
$$= c_{\overline{s}}\lim_{t\to\infty}\left(\frac{t^{-{\overline{s}}+1}}{-{\overline{s}}+1}+C\frac{t^{-{\overline{s}}}}{-{\overline{s}}} -\frac{3}{\pi}\frac{t^{-{\overline{s}}} }{-{\overline{s}}}\log t +\frac{3}{\pi}\frac{t^{-{\overline{s}}} }{{\overline{s}}^2}\right) -  c_{\overline{s}}\left(\frac{T^{-{\overline{s}}+1}}{-{\overline{s}}+1}+C\frac{T^{-{\overline{s}}}}{-{\overline{s}}} -\frac{3}{\pi}\frac{T^{-{\overline{s}}} }{-{\overline{s}}}\log T +\frac{3}{\pi}\frac{T^{-{\overline{s}}} }{{\overline{s}}^2}\right) $$

For $\text{Re}(\overline{s})>1$, the first term 
$$c_{\overline{s}}\lim_{t\to\infty}\left(\frac{t^{-{\overline{s}}+1}}{-{\overline{s}}+1}+C\frac{t^{-{\overline{s}}}}{-{\overline{s}}} -\frac{3}{\pi}\frac{t^{-{\overline{s}}} }{-{\overline{s}}}\log t +\frac{3}{\pi}\frac{t^{-{\overline{s}}} }{{\overline{s}}^2}\right)=0$$

Thus, by the Identity Principle, we can meromorphically continue to get that
$$\int_0^1\int_{ y\geq T}c_{\overline{s}} y^{-1-{\overline{s}}}\cdot c_PE_{1}^* \,dy \, dx= -c_{\overline{s}}\left(\frac{T^{1-{\overline{s}}}}{1-{\overline{s}}}-C\frac{T^{-{\overline{s}}}}{{\overline{s}}} +\frac{3}{\pi}\frac{T^{{-\overline{s}}} }{{\overline{s}}}\log T +\frac{3}{\pi}\frac{T^{-{\overline{s}}} }{{\overline{s}}^2}\right)$$

  \tab (2)  Examining the second term $\displaystyle\int_0^1\int_{ y\geq T}  c_{\overline{s}}y^{-1-{\overline{s}}}\cdot\sum_{n\neq 0}\varphi(n,1) W_1(|n|y) e^{2\pi i n x}\,dy \, dx$:

$$\displaystyle\int_0^1\int_{ y\geq T}  c_{\overline{s}}y^{-1-{\overline{s}}}\cdot\sum_{n\neq 0}\varphi(n,1) W_1(|n|y) e^{2\pi i n x}\,dy \, dx = \int_{ y\geq T}  c_{\overline{s}}y^{-1-{\overline{s}}}\cdot\sum_{n\neq 0}\varphi(n,1) W_1(|n|y) \,dy \cdot \int_0^1e^{2\pi i n x}\, dx $$
$$= \int_{ y\geq T}  c_{\overline{s}}y^{-1-{\overline{s}}}\cdot\sum_{n\neq 0}\varphi(n,1) W_1(|n|y) \,dy \cdot \delta_{0,n} = 0 $$

Thus 
$$\displaystyle{\int_{\substack{P\backslash \h \\ y> T}} c_{\overline{s}} y^{1-{\overline{s}}}\cdot E_{1}^*\,\frac{dy\,dx}{y^2}  }= -c_{\overline{s}}\left(\frac{T^{1-{\overline{s}}}}{1-{\overline{s}}}-C\frac{T^{-{\overline{s}}}}{{\overline{s}}} +\frac{3}{\pi}\frac{T^{{-\overline{s}}} }{{\overline{s}}}\log T +\frac{3}{\pi}\frac{T^{-{\overline{s}}} }{{\overline{s}}^2}\right)$$\\

\tab Putting (A) and (B) together, we get: 
$$\int_{\Gamma\backslash \h}\wedge^T \,{\overline E}_s\cdot   E_{1}^*\,\frac{dy\,dx}{y^2} $$
$$=\frac{ T^{{\overline{s}}}}{{\overline{s}}}+C\frac{T^{{\overline{s}}-1}}{{\overline{s}}-1}-\frac{3}{\pi}\frac{T^{{\overline{s}}-1}}{{\overline{s}}-1}\log T +\frac{3}{\pi}\frac{T^{{\overline{s}}-1}}{({\overline{s}}-1)^2}
+c_{\overline{s}}\left(\frac{T^{1-{\overline{s}}}}{1-{\overline{s}}}-C\frac{T^{-{\overline{s}}}}{{\overline{s}}} +\frac{3}{\pi}\frac{T^{{-\overline{s}}} }{{\overline{s}}}\log T +\frac{3}{\pi}\frac{T^{-{\overline{s}}} }{{\overline{s}}^2}\right)
$$

\end{proof}

\tab Lastly, for when $\alpha= \beta$, we will need to compute $\int_{\Gamma\backslash \h}E_s\cdot   \frac{\pi}{3}C_{\alpha} \,\frac{dy\,dx}{y^2}$.

{\lem\label{consttrunc} For each $s$, 
$$\int_{\Gamma\backslash \h}\wedge^T \,\overline{E}_s\cdot   \frac{\pi}{3}C_{\alpha}\,\frac{dy\,dx}{y^2} 
=\frac{\pi}{3}C_{\alpha}\cdot\frac{ T^{{\overline{s}}-1}}{\overline{s}-1}+\frac{\pi}{3}c_{\overline{s}}\,C_{\alpha}  \cdot \frac{T^{-\overline{s}}}{-{\overline{s}}} $$
}

\begin{proof}
$$ \int_{\Gamma\backslash \h}\wedge^T \,\overline{E}_s\cdot   \frac{\pi}{3}C_{\alpha}\,\frac{dy\,dx}{y^2} 
 =   \frac{\pi}{3}C_{\alpha} \int_{\Gamma\backslash \h}\sum_{\gamma\in P\backslash \Gamma } \text{Im}(\gamma z)^{\overline{s}}-\sum_{\gamma\in P\backslash \Gamma } \tau_{\overline{s}}(\gamma z)\,\frac{dy\,dx}{y^2} $$
 $$ =   \frac{\pi}{3}C_{\alpha} \int_{\Gamma\backslash \h}\sum_{\gamma\in P\backslash \Gamma } \left(\text{Im}(\gamma z)^{\overline{s}}-\ \tau_{\overline{s}}(\gamma z)\right)\,\frac{dy\,dx}{y^2} 
 =   \frac{\pi}{3}C_{\alpha}\int_{P\backslash \h} \left( y^{\overline{s}}-\ \tau_{\overline{s}}( z)\right)\,\frac{dy\,dx}{y^2} $$
 \hfill  by unwinding
 $$=   \frac{\pi}{3}C_{\alpha}\int_{\substack{P\backslash \h \\ y\leq T}}  y^{\overline{s}}\,\frac{dy\,dx}{y^2}-   \frac{\pi}{3}C_{\alpha}\int_{\substack{P\backslash \h \\ y> T}}c_{\overline{s}} y^{1-{\overline{s}}}\,\frac{dy\,dx}{y^2}  $$ 
  \hfill  from the definitions of $\tau_s$\\
 
(A)  Examining $\displaystyle \frac{\pi}{3}C_{\alpha}\int_{\substack{P\backslash \h \\ y\leq T}}y^{\overline{s}}\,\frac{dy\,dx}{y^2}$: 
 
 \tab Recall that the fundamental domain of $P\backslash \mathfrak{H}$ is $\{z=x+iy\in\mathfrak{H}~|~ 0\leq x\leq 1\}$ so we have
 $$ \frac{\pi}{3}C_{\alpha}\int_{\substack{P\backslash \h \\ y\leq T}}y^{\overline{s}}\,\frac{dy\,dx}{y^2}
 = \frac{\pi}{3}C_{\alpha}\int_{ y\leq T}y^{{\overline{s}}-2}\,dy
  = \frac{\pi}{3}C_{\alpha}\cdot\frac{ y^{{\overline{s}}-1}}{\overline{s}-1}\Big|_0^{T} 
  = \frac{\pi}{3}C_{\alpha}\cdot\frac{ T^{{\overline{s}}-1}}{\overline{s}-1}
- \frac{\pi}{3}C_{\alpha}\cdot\lim_{t\to 0^+} \frac{ t^{{\overline{s}}-1}}{\overline{s}-1}   $$
 
For $\text{Re}(\overline{s})>1$, the second term 
$$\lim_{t\to 0^+}\frac{ t^{{\overline{s}}-1}}{\overline{s}-1} =0$$
Thus, by the Identity Principle, we can meromorphically continue to get that

 $$ \frac{\pi}{3}C_{\alpha}\int_{\substack{P\backslash \h \\ y\leq T}}y^{\overline{s}}\,\frac{dy\,dx}{y^2}=  \frac{\pi}{3}C_{\alpha}\cdot\frac{ T^{{\overline{s}}-1}}{\overline{s}-1}
$$ 
\\

(B)  Examining $\displaystyle{ \frac{\pi}{3}C_{\alpha}\int_{\substack{P\backslash \h \\ y> T}} c_{\overline{s}} y^{1-{\overline{s}}}\,\frac{dy\,dx}{y^2}  }$: 

Recall that the fundamental domain of $P\backslash \mathfrak{H}$ is $\{z=x+iy\in\mathfrak{H}~|~ 0\leq x\leq 1\}$ so we have
$$ \frac{\pi}{3}C_{\alpha}\int_{\substack{P\backslash \h \\ y> T}} c_{\overline{s}} y^{1-{\overline{s}}}\,\frac{dy\,dx}{y^2}  
= \frac{\pi}{3}C_{\alpha}\int_{y\geq T} c_{\overline{s}} y^{-1-{\overline{s}}}\,dy
=  \frac{\pi}{3}c_{\overline{s}}C_{\alpha}  \cdot \frac{y^{-\overline{s}}}{-{\overline{s}}} \Big|_T^{infty} 
= \frac{\pi}{3}c_{\overline{s}}\,C_{\alpha}  \cdot \lim_{t\to\infty}\frac{t^{-\overline{s}}}{-{\overline{s}}}
-\frac{\pi}{3}c_{\overline{s}}\,C_{\alpha}  \cdot \frac{T^{-\overline{s}}}{-{\overline{s}}} $$

For $\text{Re}(\overline{s})>0$, the first term 
$$c_{\overline{s}}\lim_{t\to\infty}\frac{t^{-\overline{s}}}{-{\overline{s}}}=0$$

Thus, by the Identity Principle, we can meromorphically continue to get that
$$\frac{\pi}{3}C_{\alpha}\int_{\substack{P\backslash \h \\ y> T}} c_{\overline{s}} y^{1-{\overline{s}}}\,\frac{dy\,dx}{y^2}  
=-\frac{\pi}{3}c_{\overline{s}}\,C_{\alpha}  \cdot \frac{T^{-\overline{s}}}{-{\overline{s}}} $$

\tab Putting (A) and (B) together, we get: 
$$\int_{\Gamma\backslash \h}\wedge^T \,\overline{E}_s\cdot   \frac{\pi}{3}C_{\alpha}\,\frac{dy\,dx}{y^2} 
=\frac{\pi}{3}C_{\alpha}\cdot\frac{ T^{{\overline{s}}-1}}{\overline{s}-1}+\frac{\pi}{3}c_{\overline{s}}\,C_{\alpha}  \cdot \frac{T^{-\overline{s}}}{-{\overline{s}}} 
$$

\end{proof}

\tab  Finally we can apply these results to compute each $\displaystyle \int_{\Gamma\backslash \h}  \wedge^TS\cdot E_s\,\frac{dy\,dx}{y^2}$.  We will now address each of the regimes presented in Theorem \ref{existencethm}.\\

\subsection{Regimes} Recall the regimes set up in the proof of Theorem \ref{existencethm}. Again, suppose that $\alpha\neq 1$ and $\beta\neq 1$.\\

{\bf (I):} Assume $1/2 \leq \text{Re}(\alpha)< \text{Re}(\alpha)+1/2<\text{Re}(\beta)$ so $S=E_{\alpha}\cdot E_{\beta}-\left(E_{\alpha+\beta}+c_{\alpha}\cdot E_{1-\alpha+\beta}\right) $. \\

\tab First assume that $\alpha\neq 1$.   Using the above Lemma \ref{EaEb}, Lemma \ref{intid} and Theorem \ref{MS} above, after canceling terms, we have 

 $$\displaystyle \int_{\Gamma\backslash \h}  \wedge^TS\cdot E_s\,\frac{dy\,dx}{y^2}
 =  \langle  E_{\alpha}E_{\beta}, \wedge^T E_s\rangle_{L^2}
-\langle E_{\alpha+\beta}, \wedge^T E_s\rangle_{L^2}
-\langle c_{\alpha}\cdot E_{1-\alpha+\beta}, \wedge^T E_s\rangle_{L^2}$$

 $$=\frac{c_{\beta}}{\overline{s}+\alpha-\beta} T^{\overline{s}+\alpha-\beta}
+\frac{c_{\alpha}c_{\beta}}{\overline{s}-\alpha-\beta+1}T^{\overline{s}-\alpha-\beta+1} 
+L(\overline{s},E_{\alpha}\times E_{\beta})\cdot \int_{ y\leq T}y^{{\overline{s}}}\cdot W_\alpha(y) W_\beta(y) \,\frac{dy}{y^2}$$
$$
+\frac{c_{\beta}c_{\overline{s}}}{1-{\overline{s}}+\alpha-\beta}\, T^{1-{\overline{s}}+\alpha-\beta}  
+\frac{c_{\alpha}c_{\beta}c_{\overline{s}}}{2-{\overline{s}}-\alpha-\beta}\, T^{2-{\overline{s}}-\alpha-\beta} $$
$$-c_{\alpha+\beta}\cdot\frac{T^{-{\alpha-\beta}+\overline{s}}}{{-\alpha-\beta}+\overline{s}}
  -c_{\alpha+\beta} c_{\overline{s}}\cdot\frac{T^{1-{\alpha-\beta}-\overline{s}}}{1-{\alpha-\beta}-\overline{s}}
  -c_{\alpha}c_{1-\alpha+\beta}\cdot \frac{T^{\alpha-\beta+\overline{s}-1}}{\alpha-\beta+\overline{s}-1}
 -c_{\alpha}c_{1-\alpha+\beta} c_{\overline{s}}\cdot\frac{T^{\alpha-\beta-\overline{s}}}{\alpha-\beta-\overline{s}} $$
 $$\displaystyle-c_{\overline{s}}\sum_{n\neq 0}\varphi(n,\alpha)\varphi(n,\beta)n^{\overline{s}}\int_{ y\geq T} y^{-1-{\overline{s}}}\cdot W_\alpha(y) W_\beta(y)\,dy$$
 
\tab As $T\to \infty$, the polynomials will vanish on $1/2 < \text{Re}(\alpha)< \text{Re}(\alpha)+1/2<\text{Re}(\beta)$ since $\text{Re}(s)=1/2$.
Furthermore, since 
$$\displaystyle c_{\overline{s}}\sum_{n\neq 0}\varphi(n,\alpha)\varphi(n,\beta)n^{\overline{s}}\int_{ y\geq T} y^{-1-{\overline{s}}}\cdot W_\alpha(y) W_\beta(y)\,dy \to 0$$ as $T\to \infty$, we have that 
$$\langle S, E_s\rangle_{L^2} = \displaystyle \mathcal{B}^{-1}- \lim_T \langle S, \wedge^T E_s\rangle_{L^2}
=L(\overline{s},E_{\alpha}\times E_{\beta})\cdot \int_0^{\infty}y^{{\overline{s}}}\cdot W_\alpha(y) W_\beta(y) \,\frac{dy}{y^2}$$
$$=L(\overline{s},E_{\alpha}\times E_{\beta})\cdot 
\frac{\pi^{\alpha+\beta-\overline{s}}}{2\Gamma(\alpha)\Gamma(\beta)}\cdot 
\frac{\Gamma(\frac{\overline{s}+\alpha-\beta}{2})\Gamma(\frac{\overline{s}-\alpha+\beta}{2})\Gamma(\frac{\overline{s}+1-\alpha-\beta}{2}) \Gamma(\frac{\overline{s}-1+\alpha+\beta}{2})}{\Gamma(\overline{s})} 
= \Lambda(\overline{s},E_{\alpha}\times E_{\beta})$$\\


{\bf (II):} Assume $1/2\leq\text{Re}(\alpha)\leq \text{Re}(\beta)<  \text{Re}(\alpha)+1/2$ but that $\alpha\neq \beta$.
\\

\tab {\bf (IIa)} Suppose also that $\text{Re}(\alpha+\beta)> 3/2$ so that 
$S=E_{\alpha}\cdot E_{\beta}-\left(E_{\alpha+\beta}+c_{\beta}\cdot E_{1+\alpha-\beta}+c_{\alpha}\cdot E_{1-\alpha+\beta} \right).$ \\

\tab Using the above Lemma \ref{EaEb}, Lemma \ref{intid} and Theorem \ref{MS} above, after canceling terms, we have 
  $$ \displaystyle \int_{\Gamma\backslash \h}  \wedge^TS\cdot E_s\,\frac{dy\,dx}{y^2} = \langle  E_{\alpha}E_{\beta}, \wedge^T E_s\rangle_{L^2}
-\langle E_{\alpha+\beta}, \wedge^T E_s\rangle_{L^2}
-\langle c_{\beta}\cdot E_{1+\alpha-\beta}, \wedge^T E_s\rangle_{L^2}
-\langle c_{\alpha}\cdot E_{1-\alpha+\beta} , \wedge^T E_s\rangle_{L^2}$$

 $$=\frac{c_{\alpha}c_{\beta}}{-\alpha-\beta+\overline{s}+1}T^{-\alpha-\beta+\overline{s}+1} 
+L(\overline{s},E_{\alpha}\times E_{\beta})\cdot \int_{ y\leq T}y^{{\overline{s}}}\cdot W_\alpha(y) W_\beta(y) \,\frac{dy}{y^2}
+\frac{c_{\alpha}c_{\beta}c_{\overline{s}}}{-\alpha-\beta-{\overline{s}}+2}\, T^{-\alpha-\beta-{\overline{s}}+2}  $$
$$\displaystyle+c_{\overline{s}}\sum_{n\neq 0}\varphi(n,\alpha)\varphi(n,\beta)n^{\overline{s}}\int_{ y\geq T} y^{-1-{\overline{s}}}\cdot W_\alpha(y) W_\beta(y)\,dy -c_{\alpha+\beta}\cdot\frac{T^{-{\alpha-\beta}+\overline{s}}}{{-\alpha-\beta}+\overline{s}}
  -c_{\alpha+\beta} c_{\overline{s}}\cdot\frac{T^{-{\alpha-\beta}-\overline{s}+1}}{-{\alpha-\beta}-\overline{s}+1}$$
    $$ -c_{\beta}c_{1+\alpha-\beta}\cdot \frac{T^{-\alpha+\beta+\overline{s}-1}}{-\alpha+\beta+\overline{s}-1}
 -c_{\beta}c_{1+\alpha-\beta} c_{\overline{s}}\cdot\frac{T^{-\alpha+\beta-\overline{s} }}{-\alpha+\beta-\overline{s}} -c_{\alpha}c_{1-\alpha+\beta}\cdot \frac{T^{\alpha-\beta+\overline{s}-1}}{\alpha-\beta+\overline{s}-1}
 -c_{\alpha}c_{1-\alpha+\beta} c_{\overline{s}}\cdot\frac{T^{\alpha-\beta-\overline{s}}}{\alpha-\beta-\overline{s}}$$\\
 
\tab  As $T\to \infty$, the polynomials will vanish on $1/2\leq\text{Re}(\alpha)\leq \text{Re}(\beta)<  \text{Re}(\alpha)+1/2$ where  $\text{Re}(\alpha+\beta)> 3/2$ since $\text{Re}(s)=1/2$.
Furthermore, since 
$$\displaystyle c_{\overline{s}}\sum_{n\neq 0}\varphi(n,\alpha)\varphi(n,\beta)n^{\overline{s}}\int_{ y\geq T} y^{-1-{\overline{s}}}\cdot W_\alpha(y) W_\beta(y)\,dy \to 0$$ as $T\to \infty$, we have that 
$$\langle S, E_s\rangle_{L^2} = \displaystyle \mathcal{B}^{-1}- \lim_T \langle S, \wedge^T E_s\rangle_{L^2}
=L(\overline{s},E_{\alpha}\times E_{\beta})\cdot \int_0^{\infty}y^{{\overline{s}}}\cdot W_\alpha(y) W_\beta(y) \,\frac{dy}{y^2}$$
$$=L(\overline{s},E_{\alpha}\times E_{\beta})\cdot 
\frac{\pi^{\alpha+\beta-\overline{s}}}{2\Gamma(\alpha)\Gamma(\beta)}\cdot 
\frac{\Gamma(\frac{\overline{s}+\alpha-\beta}{2})\Gamma(\frac{\overline{s}-\alpha+\beta}{2})\Gamma(\frac{\overline{s}+1-\alpha-\beta}{2}) \Gamma(\frac{\overline{s}-1+\alpha+\beta}{2})}{\Gamma(\overline{s})}
= \Lambda(\overline{s},E_{\alpha}\times E_{\beta})$$\\

\tab {\bf (IIb)} Now suppose $ \text{Re}(\alpha+\beta)< 3/2$ so
$S=E_{\alpha}\cdot E_{\beta}-\left(E_{\alpha+\beta}+c_{\beta}\cdot E_{1+\alpha-\beta}+c_{\alpha}\cdot E_{1-\alpha+\beta} +c_{\alpha}c_{\beta}\cdot E_{2-\alpha-\beta}\right).$\\

\tab Using the above Lemma \ref{EaEb}, Lemma \ref{intid} and Theorem \ref{MS} above, after canceling terms, we have 
  $$ \displaystyle \int_{\Gamma\backslash \h}  \wedge^TS\cdot E_s\,\frac{dy\,dx}{y^2} $$
  $$=
   \langle  E_{\alpha}E_{\beta}, \wedge^T E_s\rangle_{L^2}
-\langle E_{\alpha+\beta}, \wedge^T E_s\rangle_{L^2}
-\langle c_{\beta}\cdot E_{1+\alpha-\beta}, \wedge^T E_s\rangle_{L^2}
-\langle c_{\alpha}\cdot E_{1-\alpha+\beta} , \wedge^T E_s\rangle_{L^2}
-\langle c_{\alpha}c_{\beta}\cdot E_{2-\alpha-\beta} , \wedge^T E_s\rangle_{L^2}$$
 $$=
L(\overline{s},E_{\alpha}\times E_{\beta})\cdot \int_{ y\leq T}y^{{\overline{s}}}\cdot W_\alpha(y) W_\beta(y) \,\frac{dy}{y^2}
\displaystyle-c_{\overline{s}}\sum_{n\neq 0}\varphi(n,\alpha)\varphi(n,\beta)n^{\overline{s}}\int_{ y\geq T} y^{-1-{\overline{s}}}\cdot W_\alpha(y) W_\beta(y)\,dy$$
$$ -c_{\alpha+\beta}\cdot\frac{T^{-{\alpha-\beta}+\overline{s}}}{{-\alpha-\beta}+\overline{s}}
   -c_{\alpha+\beta} c_{\overline{s}}\cdot\frac{T^{-{\alpha-\beta}-\overline{s}+1}}{-{\alpha-\beta}-\overline{s}+1}
    -c_{\beta}c_{\overline{s}}\cdot\frac{T^{{\alpha-\beta}-\overline{s}+1}}{{\alpha-\beta}-\overline{s}+1}$$
$$ -c_{\beta}c_{1+\alpha-\beta} c_{\overline{s}}\cdot\frac{T^{-\alpha+\beta-\overline{s} }}{-\alpha+\beta-\overline{s}}
-c_{\alpha}c_{1-\alpha+\beta}\cdot \frac{T^{\alpha-\beta+\overline{s}-1}}{\alpha-\beta+\overline{s}-1}
 -c_{\alpha}c_{1-\alpha+\beta} c_{\overline{s}}\cdot\frac{T^{\alpha-\beta-\overline{s}}}{\alpha-\beta-\overline{s}}$$
 $$-c_{\alpha}c_{\beta}c_{2-\alpha-\beta}\cdot \frac{T^{{\alpha+\beta}+\overline{s}-2}}{{\alpha+\beta}+\overline{s}-2}
 -c_{\alpha}c_{\beta}c_{2-\alpha-\beta} c_{\overline{s}}\cdot\frac{T^{\alpha+\beta-\overline{s}-1}}{\alpha+\beta-\overline{s}-1}$$\\
 
 \tab  As $T\to \infty$, the polynomials will vanish on $1/2\leq\text{Re}(\alpha)\leq \text{Re}(\beta)<  \text{Re}(\alpha)+1/2$ where  $\text{Re}(\alpha+\beta)< 3/2$ since $\text{Re}(s)=1/2$.
Furthermore, since 
$$\displaystyle c_{\overline{s}}\sum_{n\neq 0}\varphi(n,\alpha)\varphi(n,\beta)n^{\overline{s}}\int_{ y\geq T} y^{-1-{\overline{s}}}\cdot W_\alpha(y) W_\beta(y)\,dy \to 0$$ as $T\to \infty$, we have that 
$$\langle S, E_s\rangle_{L^2} = \displaystyle \mathcal{B}^{-1}- \lim_T \langle S, \wedge^T E_s\rangle_{L^2}
=L(\overline{s},E_{\alpha}\times E_{\beta})\cdot \int_0^{\infty}y^{{\overline{s}}}\cdot W_\alpha(y) W_\beta(y) \,\frac{dy}{y^2}$$
$$=L(\overline{s},E_{\alpha}\times E_{\beta})\cdot 
\frac{\pi^{\alpha+\beta-\overline{s}}}{2\Gamma(\alpha)\Gamma(\beta)}\cdot 
\frac{\Gamma(\frac{\overline{s}+\alpha-\beta}{2})\Gamma(\frac{\overline{s}-\alpha+\beta}{2})\Gamma(\frac{\overline{s}+1-\alpha-\beta}{2}) \Gamma(\frac{\overline{s}-1+\alpha+\beta}{2})}{\Gamma(\overline{s})} =\Lambda (\overline{s},E_{\alpha}\times E_{\beta})$$\\


{\bf (III):} Suppose that $\alpha=\beta$.\\

\tab {\bf (IIIa)} Also assume  $\text{Re}(\alpha)>3/4$ so $S= (E_{\alpha})^2-  E_{2\alpha}-2c_{\alpha}E_1^*+\frac{\pi}{3}C_{\alpha}$. Using the above Lemma \ref{EaEb}, Lemma \ref{intid}, Theorem \ref{MS}, Corollary \ref{MSCor} and Lemma \ref{consttrunc}   above, after canceling terms, we have

$$\displaystyle \int_{\Gamma\backslash \h}  \wedge^TS\cdot E_s\,\frac{dy\,dx}{y^2}
 =  \langle  (E_{\alpha})^2, \wedge^T E_s\rangle_{L^2}
-\langle E_{2\alpha}, \wedge^T E_s\rangle_{L^2}
-\langle 2c_{\alpha}E_1^*, \wedge^T E_s\rangle_{L^2}+\langle \frac{\pi}{3}C_{\alpha}, \wedge^T E_s\rangle_{L^2}$$
 $$=\frac{c_{\alpha}^2}{\overline{s}-2\alpha+1}T^{\overline{s}-2\alpha+1} \displaystyle+L(\overline{s},E_{\alpha}\times E_{\alpha})\cdot \int_{ y\leq T}y^{{\overline{s}}}\cdot W_\alpha(y) W_\alpha(y) \,\frac{dy}{y^2}$$

$$\displaystyle
+\frac{c_{\alpha}^2c_{\overline{s}}}{2-{\overline{s}}-2\alpha}\, T^{2-{\overline{s}}-2\alpha}
-c_{\overline{s}}\sum_{n\neq 0}\varphi(n,\alpha)\varphi(n,\alpha)n^{\overline{s}}\int_{ y\geq T} y^{-1-{\overline{s}}}\cdot W_\alpha(y) W_\alpha(y)\,dy$$
$$-c_{2\alpha}\cdot \frac{T^{(1-2\alpha)+\overline{s}-1}}{(1-2\alpha)+\overline{s}-1}
 -c_{2\alpha} c_{\overline{s}}\cdot\frac{T^{(1-2\alpha)+(1-\overline{s})-1}}{(1-2\alpha)+(1-\overline{s})-1}-2c_{\alpha}C\frac{T^{{\overline{s}}-1}}{{\overline{s}}-1}-2c_{\alpha}\frac{3}{\pi}\frac{T^{{\overline{s}}-1}}{{\overline{s}}-1}\log T +2c_{\alpha}\frac{3}{\pi}\frac{T^{{\overline{s}}-1}}{({\overline{s}}-1)^2}$$
 $$-2c_{\alpha}c_{\overline{s}}\left(2c_{\alpha}C\frac{T^{-{\overline{s}}}}{-{\overline{s}}} -\frac{3}{\pi}\frac{T^{-{\overline{s}}} }{-{\overline{s}}}\log T +\frac{3}{\pi}\frac{T^{-{\overline{s}}} }{{\overline{s}}^2}\right)+\frac{\pi}{3}C_{\alpha}\cdot\frac{ T^{{\overline{s}}-1}}{\overline{s}-1}+\frac{\pi}{3}c_{\overline{s}}\,C_{\alpha}  \cdot \frac{T^{-\overline{s}}}{-{\overline{s}}} $$ 
 
 \tab  As $T\to \infty$, the polynomials will vanish on $\text{Re}(\alpha)>3/4$ since $\text{Re}(s)=1/2$. Furthermore, since 
$$c_{\overline{s}}\sum_{n\neq 0}\varphi(n,\alpha)\varphi(n,\alpha)n^{\overline{s}}\int_{ y\geq T} y^{-1-{\overline{s}}}\cdot W_\alpha(y) W_\alpha(y)\,dy\to 0$$ 
as $T\to \infty $ we have that 

$$\langle S, E_s\rangle_{L^2} = \displaystyle \mathcal{B}^{-1}- \lim_T \langle S, \wedge^T E_s\rangle_{L^2}=
L(\overline{s},E_{\alpha}\times E_{\alpha})\cdot \int_0^\infty y^{{\overline{s}}}\cdot W_\alpha(y) W_\alpha(y) \,\frac{dy}{y^2}$$\\

\tab {\bf (IIIb)}  Suppose that $\alpha=\beta$ and $\text{Re}(\alpha)<3/4$ so $S= (E_{\alpha})^2-  E_{2\alpha}-2c_{\alpha}E_1^*-c_{\alpha}^2E_{2-2\alpha}+\frac{\pi}{3}C_{\alpha}$. Using the above Lemma \ref{EaEb}, Lemma \ref{intid}, Theorem \ref{MS}, Corollary \ref{MSCor} and Lemma \ref{consttrunc} above, after canceling terms, we have

$$\displaystyle \int_{\Gamma\backslash \h}  \wedge^TS\cdot E_s\,\frac{dy\,dx}{y^2}$$
$$ =  \langle  (E_{\alpha})^2,\wedge^T E_s\rangle_{L^2}
-\langle E_{2\alpha}, \wedge^T E_s\rangle_{L^2}
-\langle2c_{\alpha}E_1^*, \wedge^T E_s\rangle_{L^2}
-\langle c_{\alpha}^2E_{2-2\alpha}, \wedge^T E_s\rangle_{L^2}
+\langle \frac{\pi}{3}C_{\alpha}, \wedge^T E_s\rangle_{L^2}$$
 
 $$=\displaystyle
 L(\overline{s},E_{\alpha}\times E_{\alpha})\cdot \int_{ y\leq T}y^{{\overline{s}}}\cdot W_\alpha(y) W_\alpha(y) \,\frac{dy}{y^2}
 \displaystyle
-c_{\overline{s}}\sum_{n\neq 0}\varphi(n,\alpha)\varphi(n,\alpha)n^{\overline{s}}\int_{ y\geq T} y^{-1-{\overline{s}}}\cdot W_\alpha(y) W_\alpha(y)\,dy$$
$$ -c_{2\alpha}\cdot \frac{T^{(1-2\alpha)+\overline{s}-1}}{(1-2\alpha)+\overline{s}-1}
 -c_{2\alpha} c_{\overline{s}}\cdot\frac{T^{(1-2\alpha)+(1-\overline{s})-1}}{(1-2\alpha)+(1-\overline{s})-1}$$
 $$-2c_{\alpha}C\frac{T^{{\overline{s}}-1}}{{\overline{s}}-1}-2c_{\alpha}\frac{3}{\pi}\frac{T^{{\overline{s}}-1}}{{\overline{s}}-1}\log T +2c_{\alpha}\frac{3}{\pi}\frac{T^{{\overline{s}}-1}}{({\overline{s}}-1)^2}
 -2c_{\alpha}c_{\overline{s}}\left(C\frac{T^{-{\overline{s}}}}{-{\overline{s}}} -\frac{3}{\pi}\frac{T^{-{\overline{s}}} }{-{\overline{s}}}\log T +\frac{3}{\pi}\frac{T^{-{\overline{s}}} }{{\overline{s}}^2}\right)$$ 
$$ -c_{\alpha}^2c_{2-2\alpha}\cdot \frac{T^{(1-(2-2\alpha))+\overline{s}-1}}{(1-(2-2\alpha))+\overline{s}-1}
-c_{\alpha}^2c_{2-2\alpha} c_{\overline{s}}\cdot\frac{T^{(1-(2-2\alpha))+(1-\overline{s})-1}}{(1-(2-2\alpha))+(1-\overline{s})-1}$$
$$+\frac{\pi}{3}C_{\alpha}\cdot\frac{ T^{{\overline{s}}-1}}{\overline{s}-1}+\frac{\pi}{3}c_{\overline{s}}\,C_{\alpha}  \cdot \frac{T^{-\overline{s}}}{-{\overline{s}}}$$ \\

 \tab  As $T\to \infty$, the polynomials will vanish on $1/2\leq\text{Re}(\alpha)<3/4$ since $\text{Re}(s)=1/2$. Furthermore, since 
$$c_{\overline{s}}\sum_{n\neq 0}\varphi(n,\alpha)\varphi(n,\alpha)n^{\overline{s}}\int_{ y\geq T} y^{-1-{\overline{s}}}\cdot W_\alpha(y) W_\alpha(y)\,dyy\to 0$$ 
as $T\to \infty $ we have that 

$$\langle S, E_s\rangle_{L^2} = \displaystyle \mathcal{B}^{-1}- \lim_T \langle S, \wedge^T E_s\rangle_{L^2}=
L(\overline{s},E_{\alpha}\times E_{\alpha})\cdot \int_0^\infty y^{{\overline{s}}}\cdot W_\alpha(y) W_\alpha(y) \,\frac{dy}{y^2}$$\\

\tab Finally, for each $\alpha,\beta\in \mathcal{C}$, we have that 
$$\langle S, E_s\rangle_{L^2} =
L(\overline{s},E_{\alpha}\times E_{\alpha})\cdot \int_0^\infty y^{{\overline{s}}}\cdot W_\alpha(y) W_\alpha(y) \,\frac{dy}{y^2} = \Lambda(\overline{s},E_{\alpha}\times E_{\alpha}).$$

 \vspace{.5cm}

\section{The Residual Spectrum}\label{res}

\tab  We will compute the residual spectrum $\langle S,1\rangle_{L^2}$ for each $S$. 

{\thm\label{resthm} $\langle S,1\rangle_{L^2}=0$ for each $S$ presented in Theorem \ref{existencethm}.}\\

\tab We will prove this in what follows with the following Lemma and the use of truncated Eisenstein series.

{\lem\label{0} For each $\beta$ and $\alpha\neq 1$,  $\displaystyle\int_{\Gamma\backslash \mathfrak{H}}E_{\alpha}\cdot E_{\beta}-E_{\alpha+\beta}-c_{\alpha}\cdot E_{1-\alpha+\beta}\, \frac{dx\,dy}{y^2}=0$}
\begin{proof} $\displaystyle\int_{\Gamma\backslash \mathfrak{H}}E_{\alpha}\cdot E_{\beta}-E_{\alpha+\beta}-c_{\alpha}\cdot E_{1-\alpha+\beta}\, \frac{dx\,dy}{y^2}$
$$=\int_{\Gamma\backslash \mathfrak{H}}E_{\alpha}\cdot \sum_{\gamma_1\in P\backslash \Gamma}\text{Im}(\gamma_1 z)^{\beta}
-\sum_{\gamma_2\in P\backslash \Gamma}\text{Im}(\gamma_2 z)^{\alpha+\beta}
-c_{\alpha}\cdot \sum_{\gamma_3\in P\backslash \Gamma}\text{Im}(\gamma_3 z)^{1-\alpha+\beta}\, \frac{dx\,dy}{y^2}$$
$$=\int_{\Gamma\backslash \mathfrak{H}}\sum_{\gamma\in P\backslash \Gamma}\left((\gamma y)^{\beta}\cdot E_{\alpha}-(\gamma y)^{\alpha+\beta}
-c_{\alpha}\cdot (\gamma y)^{1-\alpha+\beta}\right)\, \frac{dx\,dy}{y^2}$$
\hfill by unwinding
$$=\int_{P\backslash \mathfrak{H}}\left(y^{\beta}\cdot E_{\alpha}-y^{\alpha+\beta}-c_{\alpha}\cdot  y^{1-\alpha+\beta}\right)\, \frac{dx\,dy}{y^2}$$
Now, writing out the Fourier-Whittaker expansions for $E_{\alpha}$, we have
$$=\int_{P\backslash \mathfrak{H}}\left(y^{\beta}\cdot \left( y^{\alpha}+c_{\alpha }y^{1-\alpha}+\sum_{n\neq 0}\varphi(n,\alpha)\cdot W_{\alpha}(|n|y)e^{2\pi i n x}\right)
-y^{\alpha+\beta}
-c_{\alpha}\cdot  y^{1-\alpha+\beta}\right)\, \frac{dx\,dy}{y^2}$$
$$=\int_{P\backslash \mathfrak{H}}\left(y^{\alpha+\beta}+c_{\alpha }y^{1-\alpha+\beta}+y^{\beta}\cdot\sum_{n\neq 0}\varphi(n,\alpha)\cdot W_{\alpha}(|n|y)e^{2\pi i n x}-y^{\alpha+\beta}-c_{\alpha}\cdot  y^{1-\alpha+\beta}\right)\, \frac{dx\,dy}{y^2}$$
$$=\int_{P\backslash \mathfrak{H}}y^{\beta}\cdot\sum_{n\neq 0}\varphi(n,\alpha)\cdot W_{\alpha}(|n|zy)e^{2\pi i n x}\, \frac{dx\,dy}{y^2}
=\sum_{n\neq 0}\varphi(n,\alpha)\int_0^\infty \int_0^1y^{\beta}\cdot W_{\alpha}(|n|y)e^{2\pi i n x}\, \frac{dx\,dy}{y^2}=0.$$\\

\end{proof}

 
\subsection{Regimes} Recall the regimes set up in the proof of Theorem \ref{existencethm}. Again, suppose that $\alpha\neq 1$ and $\beta\neq 1$.\\

{\bf (I):} Suppose that $1/2 \leq \text{Re}(\alpha)\leq \text{Re}(\alpha)+1/2<\text{Re}(\beta)$ then  $$\langle S,1\rangle_{L^2}= \displaystyle\int_{\Gamma\backslash \mathfrak{H}}E_{\alpha}\cdot E_{\beta}-E_{\alpha+\beta}-c_{\alpha}\cdot E_{1-\alpha+\beta}\, \frac{dx\,dy}{y^2}=0$$
by Lemma \ref{0}.\\

{\bf (II):}  Suppose $1/2\leq\text{Re}(\alpha)\leq \text{Re}(\beta)<  \text{Re}(\alpha)+1/2$.
\\

\tab {\bf (IIa)} Suppose also that $\text{Re}(\alpha+\beta)>3/2$.
Then
$$S=E_{\alpha}\cdot E_{\beta}-\left(E_{\alpha+\beta}+c_{\alpha}\cdot E_{1-\alpha+\beta}+c_{\beta}\cdot E_{1+\alpha-\beta} \right)$$

which gives
$$ \langle S,1\rangle_{L^2}
=\int_{\Gamma\backslash \mathfrak{H}}E_{\alpha}\cdot E_{\beta}-\left(E_{\alpha+\beta}+c_{\alpha}\cdot E_{1-\alpha+\beta}+c_{\beta}\cdot E_{1+\alpha-\beta} \right)\, \frac{dx\,dy}{y^2}
=-c_{\beta}\cdot \int_{\Gamma\backslash \mathfrak{H}}E_{1+\alpha-\beta}\, \frac{dx\,dy}{y^2}$$ by Lemma \ref{0}.\\

\tab We will again use Arthur truncation to compute this integral as well as a trick which involves passing the computation of a residue outside of an integral.  Given that $E_r$ is a vector-valued holomorphic function, vector-valued Cauchy (-Goursat) theory, as well as Gelfand-Pettis,  implies that we can pass the linear functional outside the integral (see \cite{Grothendieck} or \cite{Garrett2011}).\\

\tab Since $\text{Res}_{r=1}(E_r)=\frac{3}{\pi}$,
$$ \langle \wedge^T E_{1+\alpha-\beta},1\rangle_{L^2}
=\int_{\Gamma\backslash\h}  \wedge^T E_{1+\alpha-\beta}\,\frac{dy\, dx}{y^2}
= \int_{\Gamma\backslash\h}  \wedge^T E_{1+\alpha-\beta}\cdot\text{Res}_{r=1}(E_r)\cdot \frac{\pi}{ 3}\,\frac{dy\, dx}{y^2}$$ 

$$= \frac{\pi}{ 3}\cdot \text{Res}_{r=1}\left(\int_{\Gamma\backslash\h}  \wedge^T E_{1+\alpha-\beta}\cdot E_r\,\frac{dy\, dx}{y^2}\right)
= \frac{\pi}{ 3}\cdot \text{Res}_{r=1}\left(\int_{\Gamma\backslash\h}  \wedge^T E_{1+\alpha-\beta}\cdot \wedge^TE_r\,\frac{dy\, dx}{y^2}\right)$$
\hfill by Lemma \ref{intid}
$$= \frac{\pi}{ 3}\cdot \text{Res}_{r=1}\left(\frac{T^{r+\alpha-\beta}}{r+\alpha-\beta}+c_r\frac{T^{1-r+\alpha-\beta}}{1-r+\alpha-\beta}+c_{1+\alpha-\beta}\frac{T^{r-1-\alpha+\beta}}{r-1-\alpha+\beta}+c_r c_{1+\alpha-\beta}\frac{T^{-r-\alpha+\beta}}{-r-\alpha+\beta}\right)
= 0$$\\


\tab {\bf (IIb)} Now suppose also that $ \text{Re}(\alpha+\beta)< 3/2$. Then
$$S=E_{\alpha}\cdot E_{\beta}-\left(E_{\alpha+\beta}+c_{\alpha}\cdot E_{1-\alpha+\beta} +c_{\beta}\cdot E_{1+\alpha-\beta}+c_{\alpha}c_{\beta}\cdot E_{2-\alpha-\beta}\right)$$

which gives
$$ \langle S,1\rangle_{L^2}
=\int_{\Gamma\backslash \mathfrak{H}}E_{\alpha}\cdot E_{\beta}-E_{\alpha+\beta}-c_{\alpha}\cdot E_{1-\alpha+\beta} -c_{\beta}\cdot E_{1+\alpha-\beta}-c_{\alpha}c_{\beta}\cdot E_{2-\alpha-\beta}\, \frac{dx\,dy}{y^2}$$
$$=-c_{\beta}\cdot \int_{\Gamma\backslash \mathfrak{H}}E_{1+\alpha-\beta}+c_{\alpha}\cdot E_{2-\alpha-\beta}\, \frac{dx\,dy}{y^2}$$ 
by Lemma \ref{0}.\\

\tab As  {\bf (IIa)}, we again have that $ \langle \wedge^T E_{1+\alpha-\beta},1\rangle_{L^2}= 0$. We will again use Arthur truncation to compute the last integral
$$ \langle \wedge^T E_{2-\alpha-\beta},1\rangle_{L^2}
=\int_{\Gamma\backslash\h}  \wedge^T E_{2-\alpha-\beta}\,\frac{dy\, dx}{y^2}
= \int_{\Gamma\backslash\h}  \wedge^T E_{2-\alpha-\beta}\cdot\text{Res}_{r=1}(E_r)\cdot \frac{\pi}{ 3}\,\frac{dy\, dx}{y^2}$$ 
\hfill since $\text{Res}_{r=1}(E_r)=\frac{3}{\pi}$
$$= \frac{\pi}{ 3}\cdot \text{Res}_{r=1}\left(\int_{\Gamma\backslash\h}  \wedge^T E_{1+\alpha-\beta}\cdot E_r\,\frac{dy\, dx}{y^2}\right)
= \frac{\pi}{ 3}\cdot \text{Res}_{r=1}\left(\int_{\Gamma\backslash\h}  \wedge^T E_{2-\alpha-\beta}\cdot \wedge^TE_r\,\frac{dy\, dx}{y^2}\right)$$
\hfill by Lemma \ref{intid}
$$= \frac{\pi}{ 3}\cdot \text{Res}_{r=1}\left(\frac{T^{r+1-\alpha-\beta}}{r+1-\alpha-\beta}+c_r\frac{T^{-r+2-\alpha-\beta}}{-r+2-\alpha-\beta}+c_{2-\alpha-\beta}\frac{T^{r-2+\alpha+\beta}}{r-2+\alpha+\beta}+c_r c_{2-\alpha-\beta}\frac{T^{-r-1+\alpha+\beta}}{-r-1+\alpha+\beta}\right)$$
$$=0$$\\


{\bf (III):}  Suppose that $\alpha=\beta$. Unlike the other spectral integrals, we will consider this case as a limit of case {\bf (II)}.  Since both the limit and the integrals converge nicely (as already proven in Section 1), we can interchange the limit and the integral to get the following.  \\

\tab {\bf (IIIa):} Also suppose $\text{Re}(\alpha)>3/4$ so $S= (E_{\alpha})^2-  E_{2\alpha}-2c_{\alpha}E_1^*-c_{\alpha}^2E_{2-2\alpha}+\frac{\pi}{3}C_{\alpha}$ which gives
$$ \langle S,1\rangle_{L^2}=\int_{\Gamma\backslash \mathfrak{H}}(E_{\alpha})^2-  E_{2\alpha}-2c_{\alpha}E_1^*+\frac{\pi}{3}C_{\alpha}\, \frac{dx\,dy}{y^2}$$
$$=\int_{\Gamma\backslash \mathfrak{H}}\lim_{\beta\to\alpha}\left(E_{\alpha}\cdot E_{\beta}-E_{\alpha+\beta}-c_{\alpha}\cdot E_{1-\alpha+\beta}-c_{\beta}\cdot E_{1+\alpha-\beta}-c_{\alpha}c_{\beta}\cdot E_{2-\alpha-\beta}\right)\, \frac{dx\,dy}{y^2}$$
\hfill by Lemma \ref{a=b} 
$$=\lim_{\beta\to\alpha}\int_{\Gamma\backslash \mathfrak{H}}E_{\alpha}\cdot E_{\beta}-E_{\alpha+\beta}-c_{\alpha}\cdot E_{1-\alpha+\beta}-c_{\beta}\cdot E_{1+\alpha-\beta} -c_{\alpha}c_{\beta}\cdot E_{2-\alpha-\beta}\, \frac{dx\,dy}{y^2}= 0$$
by part {\bf (IIa)}.\\

\tab {\bf (IIIb):} Now suppose $1/2\leq\text{Re}(\alpha)<3/4$ so $S= (E_{\alpha})^2-  E_{2\alpha}-2c_{\alpha}E_1^*-c_{\alpha}^2E_{2-2\alpha}+\frac{\pi}{3}C_{\alpha}$ which gives
$$ \langle S,1\rangle_{L^2}=\int_{\Gamma\backslash \mathfrak{H}}(E_{\alpha})^2-  E_{2\alpha}-2c_{\alpha}E_1^*-c_{\alpha}^2E_{2-2\alpha}+\frac{\pi}{3}C_{\alpha}\, \frac{dx\,dy}{y^2}$$
$$=\int_{\Gamma\backslash \mathfrak{H}}\lim_{\beta\to\alpha}\left(E_{\alpha}\cdot E_{\beta}-E_{\alpha+\beta}-c_{\alpha}\cdot E_{1-\alpha+\beta}-c_{\beta}\cdot E_{1+\alpha-\beta}\right)\, \frac{dx\,dy}{y^2}$$
\hfill by Lemma \ref{a=b} 
$$=\lim_{\beta\to\alpha}\int_{\Gamma\backslash \mathfrak{H}}E_{\alpha}\cdot E_{\beta}-E_{\alpha+\beta}-c_{\alpha}\cdot E_{1-\alpha+\beta}-c_{\beta}\cdot E_{1+\alpha-\beta}\, \frac{dx\,dy}{y^2}= 0$$
by part {\bf (IIb)}.\\

\tab Putting these cases together we see that the residual spectrum $\langle S,1\rangle_{L^2}=0$ for each $S$.\vspace{.5cm}

\section{Spectral Decomposition and Solution}\label{sol}

\tab Finally, putting everything together we have

 $$S=\sum_{f\text{ cfm}} \langle S,f\rangle_{L^2}\cdot f+ \frac{\langle S,1\rangle_{L^2}\cdot 1}{\langle 1,1\rangle_{L^2}}+\frac{1}{4\pi i}\int_{(1/2)}\langle S,E_s\rangle_{L^2}\cdot E_s\,ds$$
$$= \sum_{f\text{ cfm}} \Lambda(\alpha,\overline{f} \times E_{\beta})\cdot f+\frac{1}{4\pi i}\int_{(1/2)}\Lambda (\overline{s},E_{\alpha}\times E_{\beta})\cdot E_s\,ds$$ where this $S\in L^2(\Gamma\backslash \h)$.\\

\tab Recall that we have found the spectral decomposition for $S=E_{\alpha}\cdot E_{\beta} - \sum_i c_i E_i+ \mathbbm{1}_{\alpha=\beta}\cdot \frac{\pi}{3}C_{\alpha}$ where $\mathbbm{1}_{\alpha=\beta}=  \left\{
     \begin{array}{lr}
       1 & \text{if }\alpha=\beta\\
       0 &  \text{if }\alpha\neq\beta
     \end{array}
   \right.$ but we want to use this to solve $\displaystyle(\Delta-\lambda_s)u_w=E_{\alpha}\cdot E_{\beta}$ on $\Gamma\backslash SL_2(\R)$.

 $$E_{\alpha}\cdot E_{\beta}=\sum_i c_{i} E_{s_i}-  \mathbbm{1}_{\alpha=\beta}\cdot \frac{\pi}{3}C_{\alpha}+\sum_{f\text{ cfm}} \Lambda(\alpha,\overline{f} \times E_{\beta})\cdot f+\frac{1}{4\pi i}\int_{(1/2)}\Lambda (\overline{s},E_{\alpha}\times E_{\beta})\cdot E_s\,ds$$ 
 where 
 $$\displaystyle\sum_i c_{i} E_{s_i}=
        E_{\alpha+\beta}+c_{\alpha} E_{1-\alpha+\beta}$$ 
on $ 1/2 \leq\text{Re}(\alpha)< \text{Re}(\alpha)+1/2<\text{Re}(\beta)$, 
        
          $$\displaystyle\sum_i c_{i} E_{s_i}=E_{\alpha+\beta}+c_{\beta} E_{1+\alpha-\beta}+c_{\alpha}\cdot E_{1-\alpha+\beta}$$ 
           on $1/2\leq\text{Re}(\alpha)\leq\text{Re}(\beta)<  \text{Re}(\alpha)+1/2$ and $ \text{Re}(\alpha+\beta)> 3/2$ but $\alpha\neq \beta$, 
       $$\displaystyle\sum_i c_{i} E_{s_i}=E_{\alpha+\beta}+c_{\beta} E_{1+\alpha-\beta}+c_{\alpha} E_{1-\alpha+\beta} +c_{\alpha}c_{\beta} E_{2-\alpha-\beta}$$ on 
 $1/2\leq\text{Re}(\alpha)\leq \text{Re}(\beta)<  \text{Re}(\alpha)+1/2$ and $\text{Re}(\alpha+\beta)< 3/2$ but $\alpha\neq \beta$,
        $$\displaystyle\sum_i c_{i} E_{s_i}=E_{\alpha}^2+2c_{\alpha} E_1^*$$ when $\alpha=\beta$ and $\text{Re}(\alpha)>3/4$ , and
        $$\displaystyle\sum_i c_{i} E_{s_i}=E_{\alpha}^2+2c_{\alpha} E_1^*+ c_{\alpha}^2 E_{2-2\alpha}$$ when $\alpha=\beta$ and $1/2\leq \text{Re}(\alpha)<3/4.$\\

\tab Now we can use this as well as the spectral relation in Section \ref{app} to solve $\displaystyle(\Delta-\lambda_w)u_w=E_{\alpha}\cdot E_{\beta}$ on $\Gamma\backslash SL_2(\R)$.  
In $\text{Re}(w)>1/2$, for  $\alpha, \beta\in C$, the solution is given by
$$u_w=\sum_i \frac{c_{i} E_{s_i}}{\lambda_{s_i}-\lambda_w}-  \mathbbm{1}_{\alpha=\beta}\cdot \frac{\frac{\pi}{3}C_{\alpha}}{\lambda_{1}-\lambda_w}+\sum_{f\text{ cfm}} \frac{\Lambda(\alpha,\overline{f} \times E_{\beta})\cdot f}{\lambda_{s_f}-\lambda_w}+\frac{1}{4\pi i}\int_{(1/2)}\Lambda (\overline{s},E_{\alpha}\times E_{\beta})\cdot \frac{E_s}{\lambda_{s}-\lambda_w}\,ds$$ and lies in $H^2(\Gamma\backslash \mathfrak{H})\oplus \mathcal{E}(\Gamma\backslash \mathfrak{H})$.
Also, note that the automorphic Sobolev space $H^k$ in which this solution exists is also defined in Section \ref{app}. This concludes our proof of Theorem \ref{main}.\\

\subsection{Meromorphic Continuation of the Solution}\label{mero}

We will now meromorphically continue the solution $$u_w=\sum_i \frac{c_{i} E_{s_i}}{\lambda_{s_i}-\lambda_w}-  \mathbbm{1}_{\alpha=\beta}\cdot \frac{\frac{\pi}{3}C_{\alpha}}{\lambda_{1}-\lambda_w}+\sum_{f\text{ cfm}} \frac{\Lambda(\alpha,\overline{f} \times E_{\beta})\cdot f}{\lambda_{s_f}-\lambda_w}+\frac{1}{4\pi i}\int_{(1/2)}\Lambda (\overline{s},E_{\alpha}\times E_{\beta})\cdot \frac{E_s}{\lambda_{s}-\lambda_w}\,ds$$
in $V:=H^2(\Gamma\backslash \mathfrak{H})\oplus\mathcal{E}(\Gamma\backslash \mathfrak{H})$ which is initially defined on $\text{Re}(w)>1/2$. \\

\tab Observe that the first three terms of $u_w$ will have meromorphic continuation.  Since Eisenstein series (and also constants) are constant in $w$ and we are only dividing by at most a simple pole given by these discrete combinations of $\alpha$ and $\beta$, the first two terms have meromorphic continuation.  In the third term of $u_w$, again the $L$-function and cuspform will be constant in $w$.  Furthermore, we can see that the eigenvalues attached to cuspforms are also discrete by examining the pre-trace formula:

$$\sum_{F:~:~|\lambda_F|\leq T}|F(z_o)|^2+\frac{| \langle F, 1 \rangle |^2}{\langle 1,1 \rangle}+\frac{1}{4\pi i}\int_{(1/2)}|E_s(z_o)|^2\, ds\ll_CT^2$$\\

\tab For the fourth term, it is important to note that the visual symmetry on the continuous spectrum in misleading.  More work must be done to meromorphically continue this piece for the spectral expansion of $u$. These meromorphic continuations do not exist in $V$ but in a larger space $M$ of moderate-growth functions that includes Eisenstein series.  For this reason meromorphic continuation is best described in terms of vector-valued integrals.  This will require a bit of topological set-up.\\

\tab Define $$\displaystyle M:=\left\{f\in C^o(\Gamma\backslash\mathfrak{H})~\big|~\sup_{\text{Im}(z)\geq \sqrt{3}/2} y^r\cdot|f(x+iy)|<\infty\text{ for some } r\in\mathbb{R}\right\}$$ The topology on $M$ is a  an inductive limit of Banach spaces \\$\displaystyle M_o^r=\left\{f\in C^o(\Gamma\backslash\mathfrak{H})~\big|~\sup_{\text{Im}(z)\geq \sqrt{3}/2} y^r\cdot|f(x+iy)|<\infty\text{ for } r\in\mathbb{R}\right\}$ obtained by the completion of $C^o(\Gamma\backslash\mathfrak{H})$ with respect to norms $\displaystyle|f|_{M_o^r}:=\sup_{\text{Im}(z)\geq \sqrt{3}/2} y^r\cdot|f(x+iy)|$ for $f\in M_o^r$.  Thus $M$ is a strict colimit in the locally convex category of Banach spaces so is quasi-complete and locally convex.\\

\tab Let $\Phi:M\to N$ be a continuous linear map to a quasi-complete locally convex topological vector space $N$ and consider the $N$-valued integrals
$$u_{w,\Phi}=\sum_i \frac{c_{i} \Phi E_{s_i}}{\lambda_{s_i}-\lambda_w}-  \mathbbm{1}_{\alpha=\beta}\cdot \frac{\frac{\pi}{3}C_{\alpha}\cdot \Phi(1)}{\lambda_{1}-\lambda_w}+\sum_{f\text{ cfm}} \frac{\Lambda(\alpha,\overline{f} \times E_{\beta})\cdot \Phi f}{\lambda_{s_f}-\lambda_w}\phantom{weeeeeeeeeeeee}$$
$$\phantom{weeeeeeeeeeeeeeeeeeeeeeeeeeeeeeeeeeee}+\frac{1}{4\pi i}\int_{(1/2)}\Lambda (\overline{s},E_{\alpha}\times E_{\beta})\cdot \frac{\Phi E_s}{\lambda_{s}-\lambda_w}\,ds$$ 
Of course, for $\Phi$ the identity map $M\to M$ gives $u_w$ itself and we anticipate that $\Phi(u_w)=u_{w,\Phi}$.\\

{\lem \label{philem} $\Phi(u_w)=u_{w,\Phi}$ in the region $\text{Re}(w)>1/2$.}

\begin{proof} Observe that 
$$\Phi(u_{w})=\sum_i \frac{c_{i} \Phi E_{s_i}}{\lambda_{s_i}-\lambda_w}-  \mathbbm{1}_{\alpha=\beta}\cdot \frac{\frac{\pi}{3}C_{\alpha}\cdot \Phi(1)}{\lambda_{1}-\lambda_w}+\sum_{f\text{ cfm}} \frac{\Lambda(\alpha,\overline{f} \times E_{\beta})\cdot \Phi f}{\lambda_{s_f}-\lambda_w}\phantom{weeeeeeeeeeeee}$$
$$\phantom{weeeeeeeeeeeeeeeeeeeeeeeeeeeeeeeeeeee}+\frac{1}{4\pi i}\Phi\left(\int_{(1/2)}\Lambda (\overline{s},E_{\alpha}\times E_{\beta})\cdot \frac{ E_s}{\lambda_{s}-\lambda_w}\,ds\right)$$ In $\text{Re}(w)>1/2$, the integral for $u_w$ is a $v$-valued holomorphic function in $w$. We have In that region, due to the properties of compactly supported continuous-integrand Gelfand-Pettis integrals \cite{Garrett2011},
$$\Phi\left(\int_{(1/2)}\Lambda (\overline{s},E_{\alpha}\times E_{\beta})\cdot \frac{ E_s}{\lambda_{s}-\lambda_w}\,ds\right)
=\Phi\left(\lim_{T\to\infty}\int_{|\text{Im}(s)|\leq T}\Lambda (\overline{s},E_{\alpha}\times E_{\beta})\cdot \frac{ E_s}{\lambda_{s}-\lambda_w}\,ds\right)$$
$$=\lim_{T\to\infty}\Phi\left(\int_{|\text{Im}(s)|\leq T}\Lambda (\overline{s},E_{\alpha}\times E_{\beta})\cdot \frac{ E_s}{\lambda_{s}-\lambda_w}\,ds\right)
=\lim_{T\to\infty}\int_{|\text{Im}(s)|\leq T}\Lambda (\overline{s},E_{\alpha}\times E_{\beta})\cdot \frac{ \Phi E_s}{\lambda_{s}-\lambda_w}\,ds$$
$$=\int_{(1/2)}\Lambda (\overline{s},E_{\alpha}\times E_{\beta})\cdot \frac{ \Phi E_s}{\lambda_{s}-\lambda_w}\,ds$$ since the limit is approached in $V\subset  M$.\\

\end{proof}

{\thm With continuous linear $\Phi:M\to N$ with $N$ quasi-complete  and locally convex, the $\Phi M$-valued function $w\mapsto u_{w,\Phi}$ has meromorphic continuation as an $N$-valued function of $w$.  Explicitly, the function

$$J_{w,\Phi}
=\sum_i \frac{c_{i} \Phi E_{s_i}}{\lambda_{s_i}-\lambda_w}-  \mathbbm{1}_{\alpha=\beta}\cdot \frac{\frac{\pi}{3}C_{\alpha}\cdot \Phi(1)}{\lambda_{1}-\lambda_w}+\sum_{f\text{ cfm}} \frac{\Lambda(\alpha,\overline{f} \times E_{\beta})\cdot \Phi f}{\lambda_{s_f}-\lambda_w}\phantom{weeeeeeeeeeeeeeeeeeeeeeeeeeeee}$$
$$\phantom{weeeeeeeeeeeee}+\frac{1}{4\pi i}\int_{(1/2)}\frac{\Lambda (1-s,E_{\alpha}\times E_{\beta})\cdot \Phi E_s - \Lambda (1-w,E_{\alpha}\times E_{\beta}) \cdot\Phi E_w}{\lambda_{s}-\lambda_w}\,ds$$
 has a meromorphic continuation to an $N$-valued function with the functional equation \\$J_{1-w,\Phi}=J_{w,\Phi}$
 and 
 $$u_{w,\Phi}=J_{w,\Phi}+\frac{\Lambda (1-w,E_{\alpha}\times E_{\beta}) \cdot\Phi E_w}{2(1-2w)}$$.}

\begin{proof} From Lemma \ref{philem}, in $\text{Re}(w)>1/2$ the expression for $u_{w,\Phi}$ converges as an $N$-valued integral.  The meromorphic continuation of $u_{w,\Phi}$ will be obtained through rearranging the integral.\\

\tab First, in $\text{Re}(w)>1/2$ we add and subtract to obtain
$$u_{w,\Phi}=\sum_i \frac{c_{i} \Phi E_{s_i}}{\lambda_{s_i}-\lambda_w}-  \mathbbm{1}_{\alpha=\beta}\cdot \frac{\frac{\pi}{3}C_{\alpha}\cdot \Phi(1)}{\lambda_{1}-\lambda_w}+\sum_{f\text{ cfm}} \frac{\Lambda(\alpha,\overline{f} \times E_{\beta})\cdot \Phi f}{\lambda_{s_f}-\lambda_w}
\phantom{weeeeeeeeeeeeeeeeeeeeeeeeeeeeeeeeeee}$$
$$\phantom{weeeeeeeeeeeeeeeeeeeeeeeee}+\frac{1}{4\pi i}\int_{(1/2)} \frac{\Lambda (1-s,E_{\alpha}\times E_{\beta})\cdot\Phi E_s}{\lambda_{s}-\lambda_w}\,ds$$ 
$$=\sum_i \frac{c_{i} \Phi E_{s_i}}{\lambda_{s_i}-\lambda_w}-  \mathbbm{1}_{\alpha=\beta}\cdot \frac{\frac{\pi}{3}C_{\alpha}\cdot \Phi(1)}{\lambda_{1}-\lambda_w}+\sum_{f\text{ cfm}} \frac{\Lambda(\alpha,\overline{f} \times E_{\beta})\cdot \Phi f}{\lambda_{s_f}-\lambda_w}\phantom{weeeeeeeeeeeeeeeeeeeeeeeeeeeee}$$
$$\phantom{weeeeee}+\frac{1}{4\pi i}\int_{(1/2)} \frac{\Lambda (1-s,E_{\alpha}\times E_{\beta})\cdot\Phi E_s - \Lambda (1-w,E_{\alpha}\times E_{\beta}) \cdot\Phi E_w}{\lambda_{s}-\lambda_w}\,ds$$
$$\phantom{weeeeeeeeeeeeeeeeeeeeeeeeeee} + \Lambda (1-w,E_{\alpha}\times E_{\beta}) \cdot\Phi E_w\frac{1}{4\pi i}\int_{(1/2)}\frac{1}{\lambda_{s}-\lambda_w}\,ds$$ 
$$=J_{w,\Phi}+ \Lambda (1-w,E_{\alpha}\times E_{\beta}) \cdot\Phi E_w \cdot\frac{1}{4\pi i}\int_{(1/2)}\frac{1}{\lambda_{s}-\lambda_w}\,ds\phantom{weeeeeeeeeeee}$$ 

\tab By residues, 
$$ \Lambda (1-w,E_{\alpha}\times E_{\beta}) \cdot\Phi E_w \cdot\frac{1}{4\pi i}\int_{(1/2)}\frac{1}{\lambda_{s}-\lambda_w}\,ds
= \Lambda (1-w,E_{\alpha}\times E_{\beta}) \cdot\Phi E_w \cdot\left( -\frac{1}{2}\cdot\text{Res}_{s=w}\frac{1}{\lambda_{s}-\lambda_w}\right)$$
$$ =\frac{\Lambda (1-w,E_{\alpha}\times E_{\beta}) \cdot\Phi E_w}{2(1-2w)}$$\\

\tab Since $\Lambda (1-w,E_{\alpha}\times E_{\beta})$ is a meromorphic $\mathbb{C}$-valued function and $w\mapsto \Phi E_w$ is a meromorphic $N$-valued function, $\Lambda (1-w,E_{\alpha}\times E_{\beta}) \cdot\Phi E_w$ is a meromorphic $N$-valued function with a meromorphic continuation from the meromorphic continuation of Eisenstein series and $\Lambda(w,E_{\alpha}\times E_{\beta})$. Observe that although the Eisenstein series is invariant under $w\mapsto 1-w$, the denominator is skew-symmetric.\\

\tab We will now meromorphically continue the integral $J_{w,\Phi}$.  First constrain $w$ so that is lies in a fixed compact set $C$ and take $T$ large enough so that $T\geq 2|w|$  for all $w\in C$.  First, for $\text{Re}(w)>1/2$ and $s=\frac{1}{2}+it$, we make an attempt to cancel the vanishing denominator when $s$ is close to $w$ by rearranging

$$J_{w,\Phi}- \left(\sum_i \frac{c_{i} \Phi E_{s_i}}{\lambda_{s_i}-\lambda_w}- \mathbbm{1}_{\alpha=\beta}\cdot \frac{\frac{\pi}{3}C_{\alpha}\cdot \Phi(1)}{\lambda_{1}-\lambda_w}+\sum_{f\text{ cfm}} \frac{\Lambda(\alpha,\overline{f} \times E_{\beta})\cdot \Phi f}{\lambda_{s_f}-\lambda_w}\right) \phantom{weeeeeeeeeeeeeeeeeeewwwweeeeeeee}$$
$$=\frac{1}{4\pi i}\int_{(1/2)}\frac{\Lambda (1-s,E_{\alpha}\times E_{\beta})\cdot \Phi E_s - \Lambda (1-w,E_{\alpha}\times E_{\beta}) \cdot\Phi E_w}{\lambda_{s}-\lambda_w}\,ds$$

$$=\frac{1}{4\pi i}\int_{|t|\geq T}\frac{\Lambda (1-s,E_{\alpha}\times E_{\beta})\cdot \Phi E_s }{\lambda_{s}-\lambda_w}\,ds
-\Lambda (1-w,E_{\alpha}\times E_{\beta}) \cdot\Phi E_w\cdot \frac{1}{4\pi i}\int_{|t|\geq T}\frac{ 1 }{\lambda_{s}-\lambda_w}\,ds$$
$$\phantom{weeeeeeeeeeeeeeeeeeeeee}+\frac{1}{4\pi i}\int_{|t|\leq T}\frac{\Lambda (1-s,E_{\alpha}\times E_{\beta})\cdot \Phi E_s - \Lambda (1-w,E_{\alpha}\times E_{\beta}) \cdot\Phi E_w}{\lambda_{s}-\lambda_w}\,ds$$\\

\tab The meromorphy of the leading integral is understood via the Plancherel Theorem on the continuous automorphic spectrum. Up to constants, the Plancherel Theorem for $L^2$ states that  $A(t)\in L^2(\mathbb{R})$ the spectral synthesis integral 
$$B=\frac{1}{4\pi}\int_{-\infty}^{\infty}A(t)\cdot E_s\,dt$$
for $z\in\h$ produces a function in $H^0$ and the map the $A\mapsto B$ gives an isometry. \\

\tab Observe that $\Lambda (1-s,E_{\alpha}\times E_{\beta})\in L^2(\frac{1}{2}+i\mathbb{R})$ since $S\in L^2(\Gamma\backslash\h)$ and $\Lambda (\overline{s},E_{\alpha}\times E_{\beta})= \langle S, E_s \rangle_{L^2(\Gamma\backslash\h)}$.  Hence for $w$ in a fixed compact, $\displaystyle{\frac{\Lambda (1-s,E_{\alpha}\times E_{\beta})}{\lambda_{s}-\lambda_w}\in L^2\left(\frac{1}{2}+i\mathbb{R}\right)}$.  Composition with Plancherel isometry  shows that 
$$w\mapsto \frac{1}{4\pi i}\int_{|t|\geq T}\frac{\Lambda (1-s,E_{\alpha}\times E_{\beta})\cdot  E_s }{\lambda_{s}-\lambda_w}\,ds$$ is a meromorphic $L^2(\frac{1}{2}+i\mathbb{R})$-valued function in $w$ in the fixed compact.  Now, since $|w|\ll T$ the meromophic continuation is given by the same integral, the invariance of the integrand under $w\mapsto 1-w$ remains.  \\


\tab In the second summand,
$$\Lambda (1-w,E_{\alpha}\times E_{\beta}) \cdot\Phi E_w\cdot \frac{1}{4\pi i}\int_{|t|\geq T}\frac{ 1 }{\lambda_{s}-\lambda_w}\,ds$$
the leading coefficient $\Lambda (1-w,E_{\alpha}\times E_{\beta}) \cdot\Phi E_w$ has meromorphic continuation and is invariant under $w\mapsto 1-w$.  Since $|w|\ll T$ the meromorphic continuation of the integrand is given by the same integral and the invariance under $w\mapsto 1-w$ remains.\\

\tab  Finally, in the remaining summand,
$$\frac{1}{4\pi i}\int_{|t|\leq T}\frac{\Lambda (1-s,E_{\alpha}\times E_{\beta})\cdot \Phi E_s - \Lambda (1-w,E_{\alpha}\times E_{\beta}) \cdot\Phi E_w}{\lambda_{s}-\lambda_w}\,ds$$ is a compactly-supported vector-valued integral.  In order to show that the integral is a meromorphic $N$-valued function of $w$, we will use the Gelfand-Pettis criterion for existence of a weak integral.  \\

\tab Let $\text{Hol}(\Omega, N)$  be the topological vector space of holomorphic $N$-valued functions on a fixed open $\Omega$ which avoids the poles if $E_w$ and has compact closure $C$.  It suffices to show that the integrand extends to a continuous $\text{Hol}(\Omega, N)$-valued function of $s$ where $\text{Hol}(\Omega, N)$ has the natural quasi-complete locally convex topology from Corollary \ref{hol}. To show that the integral extends to a holomorphic (and hence continuous) $\text{Hol}(\Omega, N)$-valued function of $s$, it suffices to show that the integral extends to a holomorphic $N$-valued function of two complex variables $s$ and $w$.\\

\tab By Cauchy-Goursat theory for vector-valued holomorphic functions (see Appendix), near a point $s_o$, the $N$-valued function $s\mapsto \Phi E_s$ has a convergent power series expansion 
$$\Phi E_s = A_0+A_1(s-s_o)+A_2(s-s_o)^2+\dots$$
with $A_i\in N$ and so $\Lambda (1-s,E_{\alpha}\times E_{\beta}) \cdot \Phi E_s$ has power series expansion
$$\Lambda (1-s,E_{\alpha}\times E_{\beta})\cdot \Phi E_s = B_0+B_1(s-s_o)+B_2(s-s_o)^2+\dots$$ 
for some $B_n\in N$. Then we have 
$$\Lambda (1-s,E_{\alpha}\times E_{\beta})\cdot \Phi E_s - \Lambda (1-w,E_{\alpha}\times E_{\beta}) \cdot\Phi E_w = B_1((s-s_o)-(w-s_o))+B_2((s-s_o)^2-(w-s_o)^2)+\dots$$
$$= ((s-s_o)-(w-s_o))\cdot \left(B_1+B_2((s-s_o)-(w-s_o))+\dots\right)$$
$$= (s-w)\cdot \left(B_1+B_2((s-s_o)-(w-s_o))+\dots\right)$$
where $ \left(B_1+B_2((s-s_o)-(w-s_o))+\dots\right)$ is a convergent power series in $s-s_o$ and $w-s_o$.  Thus the integrand, initially defined only for $s\neq w$ extends to a holomorphic $N$-valued function $F(s,w)$ including the diagonal $s=w=\frac{1}{2}+it$ with $|t|\leq T$.  Thus the $\text{Hol}(\Omega, N)$-valued function $f(s)$ given by $f(s)(w)=F(s,w)$ is holomorphic in $w$.  Thus there is a Gelfand-Pettis integral $\int_{|t|\leq T}f(\frac{1}{2}+it)\, dt$ in $\text{Hol}(\Omega, N)$ as desired.  Thus we have shown the meromorphic continuation.  The $w\mapsto 1-w$ symmetry is retained by the extension of the integral to the diagonal.\\

\end{proof}

\section{Appendix}\label{app}

\subsection{Spectral relation}
The following can be found many places including P. Garrett's \cite{Garrett2011sob} and A. DeCelles' \cite{DeCelles2016}.
\begin{thm}\label{specrel} For $f\in C_c^{\infty}(\Gamma\backslash\h)$, then $\displaystyle\langle\Delta f , E_s\rangle_{L^2(\Gamma\backslash \mathfrak{H})}  = \lambda_s \cdot\langle f , E_s\rangle_{L^2(\Gamma\backslash \mathfrak{H})}$. \end{thm}
\begin{proof}Let $f\in C_c^{\infty}(\Gamma\backslash\h)$.  
  Note that the symmetry of $\Delta$ and compact support of elements of $D$ allows integration by parts.  Then we have the following spectral relation
$$\langle\Delta f , E_s\rangle_{L^2(\Gamma\backslash \mathfrak{H})}
 = \int_{\Gamma\backslash \mathfrak{H}}\Delta f(z)\cdot E_{1-s}(z)\, \frac{dx\,dy}{y^2}
 = \int_{\Gamma\backslash \mathfrak{H}} f(z)\cdot \Delta E_{1-s}(z)\, \frac{dx\,dy}{y^2}$$ 
 $$ = \int_{\Gamma\backslash \mathfrak{H}} f(z)\cdot\lambda_s E_{1-s}(z)\, \frac{dx\,dy}{y^2}
 = \lambda_s \langle f , E_s\rangle_{L^2(\Gamma\backslash \mathfrak{H})}$$ \end{proof}
 
 For $0\leq k\in\mathbb{Z}$, the $k^{th}$-Sobolev norm on $ C_c^{\infty}(\Gamma\backslash\h)$ is given by $$|f|_k^2:=\langle(1-\Delta)^k f,f \rangle_{L^2(\Gamma\backslash \mathfrak{H})}$$ and $H^k(\Gamma\backslash \mathfrak{H})$ is the completion of $C_c^{\infty}(\Gamma\backslash\h)$ with respect to $|\cdot|_k$.\\
 
 \begin{thm} There is a continuous injection $H^k(\Gamma\backslash \mathfrak{H})\to H^{k+1}(\Gamma\backslash \mathfrak{H})$ with dense image.\end{thm}

\begin{proof} Let $f\in C_c^{\infty}(\Gamma\backslash\h)$ then $\langle -\Delta f, f\rangle\geq 0$. We would like to show that for a polynomial $p$ with non-negative real coefficients  $\langle p(-\Delta) f, f\rangle\geq 0$.  It suffices to show that  $\langle (-\Delta)^n f, f\rangle\geq 0$.

For $n=2m$ even, 
$$\langle (-\Delta)^n f, f\rangle = \langle (-\Delta)^{2m} f, f\rangle=\langle (-\Delta)^m f,  (-\Delta)^mf\rangle\geq 0.$$
For $n=2m+1$ odd, 
$$\langle (-\Delta)^n f, f\rangle = \langle (-\Delta)^{2m+1} f, f\rangle=\langle  (-\Delta)((-\Delta)^m f),  (-\Delta)^mf\rangle\geq 0.$$
This gives $$|f|^2_{k+1}=  \langle (1-\Delta)^{k+1} f, f\rangle=  \langle (1+(-\Delta))^{k} f, f\rangle+ \langle (1+(-\Delta))^{k}(-\Delta) f, f\rangle$$
$$\geq \langle (1+(-\Delta))^{k} f, f\rangle +0 = |f|^2_k$$

Thus the identity map $C_c^{\infty}(\Gamma\backslash\h)$ extends to a continuous injection $H^{k+1}\to H^k$ since $C_c^{\infty}(\Gamma\backslash\h)$ is dense in both. Furthermore, the image is dense.\\

\end{proof}

\begin{thm} The differential operator $\Delta: C_c^{\infty}(\Gamma\backslash\h)\to C_c^{\infty}(\Gamma\backslash\h)$ is continuous when the source is given the $H^{k+2}$ topology and the target is given the $H^k$ topology for $0\geq k\in\mathbb{Z}$.\end{thm}

\begin{proof} Using the latter negativity property of the previous proof, we have 
$$|\Delta f|^2_{k}=  \langle (1-\Delta)^{k}(\Delta f), (\Delta f)\rangle=  \langle (-\Delta)^2 (1+(-\Delta))^{k} f, f\rangle$$
$$\leq \langle (-\Delta)^2 (1+(-\Delta))^{k} f, f\rangle +\langle (2 (-\Delta)+1) f, f\rangle =  \langle (1+(-\Delta))^{k+2} f, f\rangle = |f|_{k+1}^2$$
\end{proof}

\begin{cor} $\Delta$ extends by continuity from test functions to a continuous linear map
$\Delta:H^{k+2}(\Gamma\backslash \mathfrak{H})\to H^k(\Gamma\backslash \mathfrak{H})$ for each $0\leq k\in\mathbb{Z}$.\end{cor}
\begin{proof} For test functions $\{f_n\}$ forming a Cauchy sequence in the $H^{k+1}$ topology, the continuity on the respective topologies on test functions means that the extension-by-continuity definition $$\Delta(H^{k+2}\!\!-\!\!\lim_n f_n)=H^k\!\!-\!\!\lim_n\Delta f_n$$ is well-defined and given a continuous map in those topologies.
\end{proof}
\begin{cor} For $f\in H^k(\Gamma\backslash\h)$, then $\displaystyle\langle\Delta f , E_s\rangle_{L^2(\Gamma\backslash \mathfrak{H})}  = \lambda_s \cdot\langle f , E_s\rangle_{L^2(\Gamma\backslash \mathfrak{H})}$.
\end{cor}
\begin{proof} Because $\langle\cdot, E_s\rangle_{L^2(\Gamma\backslash \mathfrak{H})} : L^2(\Gamma\backslash \h) \to L^2(1/2+i[0,\infty))$ is an isometric isomorphism obtained by extension by continuity on test functions, the literal spectral integrals in Theorem \ref{specrel} extend by continuity to give the result.\end{proof}
The same argument can be given for each function in 
$$\Xi = \{\text{orthonormal basis of cuspforms}\}\cup \{1 \} \cup 1/2+i[0,\infty)$$ where the half-line parametrizes the Eisenstein series $E_{1/2+it}$.\\


\subsection{Vector-valued integrals}  
There is at least one technical point to address. We will need a bit of machinery introduced by Gelfand (1936) \cite{Gelfand1936} and Pettis (1938) \cite{Pettis1938}. Their construction produces integrals of continuous vector-valued functions with compact support.  These integrals are not {\it constructed} using limits, in contrast to Bochner integrals, but instead are {\it characterized} by the desired property that they commute with linear functionals.\\

Let $V$ be a complex topological vector space.  Let $f$ be a measurable $V$-valued function on a measure space $X$.  A {\it Gelfand-Pettis integral} of $f$ is a vector $I_f\in V$ so that 
$$\alpha(I_f)=\int_X\alpha\circ f$$ for all $\alpha\in V^*$. Assuming that it exists and is unique, the vector $I_f$ is denoted $I_f=\int_Xf$.\\

Uniqueness and linearity of the integral follow from the fact that $V^*$ separates points by Hahn-Banach.  Establishing the existence of Gelfand-Pettis integrals is more delicate. 

{\thm Let $X$ be a compact Hausdorff topological space with a finite positive regular Borel measure.  Let $V$ be a quasi-complete, locally convex topological vectorspace.  Then continuous compactly-supported $V$-values functions $f$ on $X$ have Gelfand-Pettis integrals.}\\

The importance of the characterization of the Gelfand-Pettis integral is exhibited in the following corollary.

{\cor Let $T:V\to W$ be a continuous linear map of locally convex quasi-complete topological vector spaces and $f$a continuous $V$-valued function on $X$. Then
$$T\left( \int_Xf\right)=\int_XT\circ f.$$}
\begin{proof} Since $W^*$ separates points, it suffices to show that 
$$\mu\left( T\left( \int_Xf\right)\right)=\mu\left( \int_X T\circ f\right).$$
Since $\mu\circ T\in V^*$, the characterization of Gelfand-Pettis integrals gives
$$\mu\left( T\left( \int_Xf\right)\right)=(\mu\circ T)\left( \int_Xf\right)=\int_X\mu(T\circ f)= \mu\left( \int_X T\circ f\right).$$
\end{proof}\vspace{.2cm}


\subsection{Holomorphic vector-valued functions}  

We will recall some basic facts about vector-values functions, most of which we will not prove here.  However, for proofs and further explanation see Grothendieck's \cite{Grothendieck} for the original or Rudin's \cite{Rudin}.\\

\tab Let $f$ be a function of an open set $\Omega\subset \mathbb{C}$ taking values in a quasi-complete, locally convex space $V$. We say $f$ is {\it weakly holomorphic} when $\mathbb{C}$-valued functions $\lambda\circ f$ are holomorphic for all $\lambda\in V^*$.\\
 
 \tab  Let $\text{Hol}(\Omega, N)$  be the topological vector space of holomorphic $N$-valued functions on a fixed open $\Omega$.\\

 {\thm For $V$ a locally convex quasi-complete topological vector space, weakly holomorphic $V$-valued functions $f$ are strongly holomorphic in the following senses.
 
 First the usual Cauchy-theory integral formulas apply:
 $$f(z)=\frac{1}{2\pi i}\int_{\gamma}\frac{f(\zeta)}{\zeta-z}\,d\zeta$$
 with $\gamma$ a closed path around $z$ having winding number 1.   Second, the function $f(z)$ is infinitely differentiable, in fact strongly analytic, that is, expressible as a convergent power series $\displaystyle\sum_{n\geq 0} c_n(z-z_o)^n$ with coefficients $c_n\in V$ given by Gelfand Pettis integrals echoing Cauchy's formulas:
 $\displaystyle c_n=\frac{f^{(n)}(z_o)}{n!}=\frac{1}{2\pi i}\int_{\gamma}\frac{f(\zeta)}{(\zeta-z)^{n+1}}\,d\zeta$}\\
 
\tab  In \cite{Rudin}, the proof also uses the fact that {\it weak} boundedness implies boundedness to first show that $f$ is continuous.  Then recapulation in the vector-valued context is viable.  .\\

\tab Now fix a non-empty open $\Omega\subset\mathbb{C}$.  Let $V$ be quasi-complete, locally convex, with topology given by seminorms $\{\nu\}$.  The space $\text{Hol}(\Omega, N)$ of holomorphic $v$-values functions on $\Omega$ has a natural topology given by seminorms $\mu_{\nu. K}(f) = \sup_{z\in K}\nu(f(z))$ for compacts $K\subset \Omega$ seminorms $\nu$ on $V$.

{\cor\label{hol} $\text{Hol}(\Omega, N)$ is locally convex, quasi-complete.}\\


\end{document}